\documentclass{amsart}

\usepackage{array}
\usepackage{graphicx,verbatim, amsmath, amssymb, amsthm, amsfonts, amsxtra,ifthen,
caption,enumitem,
hhline,capt-of,longtable,tikz-cd,mathtools}	
\usepackage[bookmarks=true, bookmarksopen=false,%
    colorlinks=true,%
    linkcolor=darkblue,%
    citecolor=darkblue,%
    filecolor=darkblue,%
    menucolor=darkblue,%
    linktoc=all,%
    urlcolor=darkblue
]{hyperref}
\usepackage{shuffle}

\definecolor{darkblue}{cmyk}{1,0.3,0,0.1}  

\newtheorem{proposition}{Proposition}[section]
\newtheorem{theorem}[proposition]{Theorem}
\newtheorem{corollary}[proposition]{Corollary}
\newtheorem{lemma}[proposition]{Lemma}

\theoremstyle{definition}

\newtheorem{definition}[proposition]{Definition}

\newtheorem{remark}[proposition]{Remark}

\numberwithin{equation}{section}

\newcommand{\newword}[1]{\textbf{\textit{#1}}}

\newcommand{\integers}{\mathbb Z}

\newcommand{\reals}{\mathbb R}

\newcommand{\field}{\mathbb K}

\newcommand{\ep}{\varepsilon}

\newcommand{\covered}{\lessdot}
\newcommand{\set}[1]{{\left\lbrace #1 \right\rbrace}}
\newcommand{\br}[1]{{\langle #1 \rangle}}
\newcommand{\A}{{\mathcal A}}
\newcommand{\B}{{\mathcal B}}
\newcommand{\C}{{\mathcal C}}
\newcommand{\Cx}{{\C_{\mathsf{e}}}}
\newcommand{\Bx}{{B_{\mathsf{e}}}}
\newcommand{\Cpx}{{\C_{\mathsf{pe}}}}
\newcommand{\Bpx}{{B_{\mathsf{pe}}}}
\newcommand{\Crpx}{{\C_{\mathsf{rpe}}}}
\newcommand{\Brpx}{{B_{\mathsf{rpe}}}}
\newcommand{\Ct}{{\C_{\T}}}
\newcommand{\Bt}{{B_{\T}}}
\newcommand{\E}{{\mathcal E}}
\renewcommand{\P}{{\mathcal P}}

\newcommand{\R}{{\mathcal R}}
\newcommand{\T}{{\mathcal T}}
\renewcommand{\S}{{\mathcal S}}
\newcommand{\W}{{\mathcal W}}
\newcommand{\U}{{\mathcal U}}
\newcommand{\V}{{\mathcal V}}
\newcommand{\X}{{\mathcal X}}
\newcommand{\Y}{{\mathcal Y}}
\newcommand{\excep}{\mathcal{E}}
\newcommand{\pexcep}{\mathrm{p}\excep}

\newcommand{\nrig}{\mathcal{F}}
\newcommand{\homo}{{\mathcal H}}
\newcommand{\nonhomo}{{\mathcal N}}

\newcommand{\N}{\mathfrak{T}}
\newcommand{\sW}{{\mathsf W}}
\newcommand{\e}{\mathbf{e}}
\newcommand{\join}{\vee}
\newcommand{\meet}{\wedge}

\newcommand{\ck}{\spcheck}

\renewcommand{\th}{^\text{th}}

\newcommand{\fin}{\mathrm{fin}}

\newcommand{\perm}{\mathsf{perm}}

\newcommand{\mods}{\operatorname{mod}}

\newcommand{\reg}{\mathcal{R}}
\newcommand{\exreg}{\nonhomo}
\newcommand{\Hom}{\operatorname{Hom}}
\newcommand{\Ext}{\operatorname{Ext}}
\newcommand{\End}{\operatorname{End}}

\newcommand{\brick}{\operatorname{brick}}
\newcommand{\wide}{\operatorname{wide}}

\newcommand{\pewide}{\mathsf{pe\text{-}wide}}
\newcommand{\ewide}{\operatorname{e-wide}}
\newcommand{\add}{\operatorname{add}}

\newcommand{\Filt}{\operatorname{Filt}}
\newcommand{\undim}{\underline{\dim}}
\newcommand{\simp}{\operatorname{simp}}

\newcommand{\red}{\mathrm{red}}
\newcommand{\aug}{\mathrm{aug}}

\newcommand{\Pre}{\mathrm{Pre}}
\newcommand{\Post}{\mathrm{Post}}
\newcommand{\Preeq}{\mathrm{P\widetilde{re}}}
\newcommand{\Posteq}{\mathrm{P\widetilde{os}t}}
\newcommand{\pre}{\mathrm{pre}}
\newcommand{\post}{\mathrm{post}}

\newcommand{\afftype}[1]{{\widetilde{\raisebox{0pt}[6pt][0pt]{#1}}}}

\title[Para-exceptional sequences and the McCammond-Sulway lattice]{Para-exceptional sequences for tame hereditary algebras and McCammond-Sulway lattices}
\author{Eric J. Hanson}
\author{Nathan Reading}
\address{\hspace{-16pt} Department of Mathematics, North Carolina State University, Raleigh, NC 27695, USA}
\email{ejhanso3@ncsu.edu}
\email{reading@math.ncsu.edu}
\thanks{Nathan Reading was partially supported by the National Science Foundation under award number DMS-2054489.  Eric Hanson was supported by an AMS-Simons travel grant.}

\begin{document}

\begin{abstract} 
Noncrossing partition posets in a Coxeter group $W$ can fail to be lattices when $W$ is not finite.
When the lattice property fails for $W$ of affine type, McCammond and Sulway's construction provides a larger lattice that contains the noncrossing partition poset and that furthermore is a combinatorial Garside structure.
We construct a lattice, isomorphic to McCammond and Sulway's lattice, using the representation theory of a corresponding connected tame hereditary algebra and give a representation-theoretic proof that it is a combinatorial Garside structure.
To construct the lattice, we introduce para-exceptional sequences and para-exceptional subcategories in the module categories of tame hereditary algebras.
Para-exceptional sequences are generalizations of exceptional sequences obtained by enlarging the set of allowed entries to include all non-homogeneous bricks.
A para-exceptional subcategory is a subcategory obtained by applying a certain closure-like operator to the wide subcategory generated by a para-exceptional sequence.
\end{abstract}
\maketitle

\setcounter{tocdepth}{1}
\tableofcontents

\section{Introduction} \label{intro}
The fundamental combinatorial information underlying (crystallographic) Coxeter groups, exceptional sequences, and exceptional subcategories is a symmetrizable Cartan matrix.
Given such a Cartan matrix, one can construct a Kac-Moody root system and a crystallographic Coxeter group, and then a poset of noncrossing partitions in the Coxeter group.
On the other hand, a finite-dimensional hereditary algebra specifies a Cartan matrix and a Coxeter element of the associated Coxeter group, which can be used to describe the exceptional sequences and exceptional subcategories. 

Fix a Cartan matrix, an associated hereditary algebra $\Lambda$, and the corresponding Coxeter group $W$ and Coxeter element~$c$. 
The noncrossing partition poset specified by $W$ and~$c$ is the interval $[1,c]_T$ between the identity and~$c$ in a partial order on $W$ called the absolute order.
Maximal chains in $[1,c]_T$ are in bijection with 
complete exceptional sequences in the category $\mods\Lambda$ \cite{IgusaSchiffler,Krause} and elements of $[1,c]_T$ are in bijection with exceptional subcategories of $\mods\Lambda$ \cite{HuberyKrause,IgusaSchiffler,IngTho}.

When $[1,c]_T$ is a lattice, it can be used to prove important structural results~\cite{Bessis,Bra-Wa} about the Artin group associated to $W$.
When $W$ is of affine type, $[1,c]_T$ fails to be a lattice in many cases, specifically when the horizontal root system associated to~$W$ is reducible~\cite{McFailure}.
In those cases, McCammond and Sulway~\cite{McSul} obtained the same structural results about the Artin group using a supergroup of~$W$ generated by $W$ and some additional elements called \newword{factored translations}.
The interval between~$1$ and~$c$ in the supergroup is a lattice having $[1,c]_T$ as an induced subposet, which we call the \newword{McCammond-Sulway lattice}.
(In the case where the horizontal root system is irreducible, so that the noncrossing partition poset is already a lattice, there is no separate construction of the McCammond-Sulway lattice.)

The case where $W$ is of affine type is the case where the hereditary algebra~$\Lambda$ is connected and tame. 
The associated horizontal root system is reducible precisely when $\Lambda$ contains more than one non-homogeneous tube.
The purpose of this paper is to construct the McCammond-Sulway lattice using the representation theory of~$\Lambda$ and give representation-theoretic proofs of the lattice property and other properties.
The basic representation-theoretic notions are those of \newword{para-exceptional modules}, \newword{para-exceptional sequences} and the corresponding \newword{para-exceptional subcategories}.
Para-exceptional sequences are defined analogously to exceptional sequences, but the set of para-exceptional modules includes not only the exceptional modules but also all non-homogeneous non-exceptional bricks. 
We emphasize that these representation-theoretic notions are defined for all connected tame hereditary algebras $\Lambda$, including those with zero or one non-homogeneous tubes. 
(When $\Lambda$ has fewer than two non-homogeneous tubes, para-exceptional subcategories are the same as exceptional subcategories, but the para-exceptional sequences are a proper superset of the exceptional sequences.)

\subsection{Organization and main results}
In Section~\ref{back sec}, we recall background information about Coxeter groups and hereditary algebras, particularly the connection between noncrossing partitions, exceptional sequences, and exceptional subcategories (Theorems~\ref{thm:IS} and~\ref{thm:NC}).
In the process, we generalize exceptional sequences by defining $\mathcal{B}$-brick sequences for $\mathcal{B}$ a set of bricks (Definition~\ref{def:exceptional_sequence}). 
Many constructions appearing later in the paper fit into this generalized framework.

Section~\ref{chain sys sec} develops the notion of a chain system to automate the study of edge-labeled posets in terms of label sequences on their maximal chains.
Chain systems are used in this paper to describe known results connecting exceptional sequences and the noncrossing partition poset and then to extend these results para-exceptional sequences and the McCammond-Sulway lattice.

In Section~\ref{ex chain sec}, by combining and reinterpreting results from the literature, we prove that the set of complete exceptional sequences of any hereditary algebra is a chain system (Theorem~\ref{thm:exceptional_binary_chain}). 
We also give a representation-theoretic version of a known result (see \cite[Theorem~0.5.2]{Bessis}) in the theory of Artin groups:
The lattice of exceptional subcategories, with the labeling coming from exceptional sequences, is a combinatorial Garside structure (see Section~\ref{Garside sec}) whenever it is a lattice (Corollary~\ref{cor:garside_if_lattice}).

In Section~\ref{nc chain sec}, we describe maximal chains in the noncrossing partition poset in terms of a binary compatibility relation that corresponds to the binary compatibility relation that defines exceptional sequences.
(In the language of chain systems, the labels on maximal chains constitute a binary chain system.)
Surprisingly, this basic fact about noncrossing partitions is currently only known through a representation-theoretic proof due to Hubery and Krause \cite{HuberyKrause}, drawing on Igusa and Schiffler \cite{IgusaSchiffler}.

In Section~\ref{McSul chain sec}, we recall the construction of the McCammond-Sulway lattice 
and describe it explicitly in terms of a binary chain system (Theorem~\ref{Ccplus chain sys}). 
A crucial part of that description is a new fact about the McCammond-Sulway lattice:
Factored translations are in bijection with simple horizontal roots
(which correspond to the quasi-simple modules in the non-homogeneous tubes).

Section~\ref{type sec} completes the proof of Theorem~\ref{Ccplus chain sys}, using combinatorial models for the classical affine types in terms of (symmetric) noncrossing partitions of surfaces (with double points) developed in \cite{affncA,affncD}.
In light of a later result (Corollary~\ref{McSul to pewide}), the combinatorial models from \cite{affncA,affncD} are models for the lattice of para-exceptional subcategories in classical affine type.

In Section~\ref{sec:brick_sequences}, we introduce para-exceptional modules and para-exceptional sequences for a connected tame hereditary algebra~$\Lambda$. 
When $\Lambda$ has at least two non-homogeneous tubes, we give a bijection from the set of para-exceptional modules to the set of generators of the McCammond-Sulway supergroup that induces a bijection from the set of maximal para-exceptional sequences to the binary chain system for the McCammond-Sulway lattice (Theorem~\ref{thm:bijection_on_chain_systems}).
That is, maximal para-exceptional sequences correspond to maximal chains of the McCammond-Sulway lattice.

In Section~\ref{sec:one tube}, we study $\brick(\T)$-brick sequences and wide subcategories for $\T$ a tube category. 
Our results on $\brick(\T)$-brick sequences sequences draw on \cite{IgusaSen}, where they were introduced under the name ``soft exceptional sequences''.
We show that $\brick(\T)$-brick sequences constitute a binary chain system and that the corresponding poset is isomorphic to the lattice of wide subcategories of $\T$ (Theorem~\ref{thm:tube_binary_chain}). Building on a result of \cite{IgusaSen}, we show that the lattice of wide subcategories of $\T$ is isomorphic to a noncrossing partition lattice of type C (Corollary~\ref{cor:tube_type_C}).

In Section~\ref{sec:rep_theory_tubes}, we extend the results of Section~\ref{sec:one tube} to describe the regular para-exceptional sequences of a connected tame hereditary algebra.
With no assumption on the number of non-homogeneous tubes, we show that the maximal regular para-exceptional sequences constitute a binary chain system and that the corresponding labeled poset is isomorphic to the lattice of wide subcategories of the category of non-homogeneous regular modules (Theorem~\ref{thm:soft_binary_chain}).

In Section~\ref{sec:rep_theory_model}, we introduce para-exceptional subcategories for connected tame hereditary algebras,  using a closure-like operator $\W\mapsto\overline{\W}$ on wide subcategories (Definition~\ref{def:E_perp}).
When there are at least two non-homogeneous tubes, we use para-exceptional subcategories 
to give a purely representation-theoretic proof that the set of maximal para-exceptional sequences is a binary chain system and that the induced labeled poset is isomorphic to the poset of para-exceptional subcategories (Theorem~\ref{thm:binary_chain_full}).
As a consequence, we conclude that the poset of para-exceptional subcategories is isomorphic to the corresponding McCammond-Sulway lattice (Corollary~\ref{McSul to pewide}). 

In Section~\ref{sec:lattice}, we study the lattice and Garside propeties of the poset of para-exceptional subcategories.
Indeed, the isomorphism in Corollary~\ref{McSul to pewide} implies (in the case of more than one non-homogeneous tube) that 
this poset
is a lattice.
We also give a purely representation-theoretic proof of the lattice property (Theorem~\ref{thm:join}), with no assumption on the number of non-homogeneous tubes.  
We show furthermore that the poset of exceptional subcategories is a lattice if and only if it equals the poset of para-exceptional subcategories if and only if $\Lambda$ has fewer than two non-homogeneous tubes (Corollary~\ref{lattice iff}).
This is a representation-theoretic version of the fact, due to Digne~\cite{Digne1,Digne2} and McCammond~\cite{McFailure}, that the noncrossing partition poset is a lattice if and only if the horizontal root system is irreducible.
We also give a representation-theoretic proof (Theorem~\ref{thm:garside} and Corollary~\ref{lattice iff}) that the poset of para-exceptional subcategories is a combinatorial Garside structure (see Section~\ref{Garside sec}).
In the case where $\Lambda$ has at least two non-homogeneous tubes, this fact is central to McCammond and Sulway's results on Euclidean Artin groups.

Finally, in Section~\ref{sec:other}, we discuss the relationship between our results and other constructions in the literature.

\subsection{Acknowledgments} 
The authors wish to thank Claire Amiot, Grant Barkley, Judith Marquardt, Pierre-Guy Plamondon, Cyril Matousek, Salvatore Stella, and Hugh Thomas for helpful conversations.

\section{Background}\label{back sec}
In this section, we fill in  background information about Coxeter groups and hereditary algebras and the connection between them.

\subsection{The noncrossing partition poset in a Coxeter group}\label{nc sec}
A symmetrizable Cartan matrix in the sense of~\cite{Kac} is a matrix $A = (a_{i,j})_{1\le i,j\le n}$ such that the $a_{ij}$ are integers, the diagonal entries are $a_{ii}=2$ for all $i=1,\ldots,n$, the off-diagonal entries are non-positive, and there exist positive constants $d_1,\ldots,d_n$ such that $d_ia_{ij}=d_ja_{ji}$ for $1\le i,j\le n$.
The matrix $A$ is said to be \newword{symmetrizable}.
(Throughout this paper, we will take a choice of the $d_i$ that comes from representation theory.
See Section~\ref{sec:reps}.)
An important special case is the case where $A$ is symmetric (as opposed to merely symmetrizable). 
Equivalently, all the $d_i$ are equal. 

Let $V$ be a real vector space with a basis $\set{\alpha_i:i=1\ldots,n}$ called the \newword{simple roots}.
Define the \newword{simple co-roots} to be the vectors $\alpha_i\ck=d_i^{-1}\alpha_i$ for $i=1,\ldots,n$.
Define a bilinear form $K$ by $K(\alpha_i\ck,\alpha_j)=a_{ij}$.
One can check that this is symmetric. 

For each $i$, define a linear map $s_i$ on $V$ called the \newword{simple reflection} for~$i$, by setting ${s_i(\alpha_j)=\alpha_j-a_{ij}\alpha_i}$ for all $i$ and $j$.  
The simple reflections $S=\set{s_1,\ldots,s_n}$ generate a group $W$ of transformations of $V$ called the \newword{Weyl group} of $A$, which is in particular a crystallographic Coxeter group with defining generators~$S$. 

The set of \newword{real roots} is $\set{w(\alpha_i):w\in W,i=1,\ldots,n}$.
In addition, there are \newword{imaginary roots}, which we do not need to define here, except in the case where~$A$ is of affine type, which we will handle separately in Section~\ref{McSul chain sec}.
The set of all real and imaginary roots is the \newword{root system} $\Phi$ associated to $A$.

Each root $\beta\in\Phi$ is either \newword{positive} (nonzero, with nonnegative entries) or \newword{negative} (nonzero, with nonpositive entries).
We write $\Phi^+$ for the set of positive roots in~$\Phi$ and $\integers\Phi$ for the \newword{root lattice}, the set of integer-linear combinations of roots.  
Every root is an integer-linear combination of simple roots, so the simple roots also generate the root lattice.
For each real root $\beta$, define a corresponding \newword{co-root} $\beta\ck=\frac{2}{K(\beta,\beta)}\beta$.
A \newword{reflection} is a linear map with a codimension-$1$ fixed space and an eigenvalue $-1$.
Every real root $\beta$ defines a reflection $t_\beta$ that maps $x\in V$ to $x-K(\beta\ck,x)\beta$.
The reflection $t_\beta$ is an element of $W$, and specifically, if $\beta = w\alpha_i$, then $t_\beta = ws_iw^{-1}$.
This correspondence restricts to a bijection between the set of positive real roots and the set of elements of $W$ that act as reflections in $V$.
Accordingly, the set $T$ of elements of $W$ that act as reflections is $\set{ws_iw^{-1}:w\in W,i=1,\ldots,n}$.

A \newword{$T$-word} is a sequence of reflections.
Any element $w\in W$ can be written as a product of a $T$-word.
A $T$-word for $w$ that has minimal length among $T$-words for $W$ is called a \newword{reduced $T$-word}.
The length of a reduced $T$-word for $w$ is the \newword{reflection length} or \newword{absolute length} $\ell_T(w)$.
The \newword{absolute order} on $W$ is the partial order with $v\le w$ if and only if there exists a reduced $T$-word $t_1\cdots t_\ell$ for $w$ and an index $i\in\set{1,\ldots,\ell}$ such that $t_1\cdots t_i$ is a $T$-word for $v$.  
(The $T$-word $t_1\cdots t_i$ will then always be reduced.)

A \newword{Coxeter element} is the product of the elements of $S$ in some order, each appearing exactly once.
Given a Coxeter element $c$ of $W$, the interval $[1,c]_T$ in the absolute order is known as the \newword{noncrossing partition poset} for~$W$ and~$c$.

The maximal chains in $[1,c]_T$ are naturally in bijection with reduced $T$-words for $c$.
By definition, every reduced $T$-word for $c$ has the same finite length, and it is immediate from a result of Dyer~\cite[Theorem~1.1]{DyerLength} that this length is~$|S|=n$.

Noncrossing partition posets were introduced by Bessis \cite{Bessis} and by Brady and Watt \cite{Bra-Wa} in the case where $W$ is finite, in order to give a ``dual presentation'' of the associated Artin group of spherical type.
Analogous results for Artin groups of Euclidean type were obtained by McCammond and Sulway~\cite{McSul}, and the purpose of this paper is to realize McCammond and Sulway's construction in the representation theory of hereditary algebras. 

\subsection{Hereditary algebras} \label{sec:reps}
We now recall background information about finite-dimensional hereditary algebras. 
We follow exposition from \cite{PreprojTame,RingelBraid}.
See those papers for additional information and historical context.

Let $\Lambda$ be a finite-dimensional associative algebra over a field $\field$. 
Let $\mods\Lambda$ be the category of finite-dimensional (over $\field$) left $\Lambda$-modules. 
By a \newword{module}, we always mean an object of $\mods\Lambda$. 
We assume that $\Lambda$ is \newword{hereditary}, meaning that $\Ext^i(X,Y) = 0$ for all modules $X, Y$ and all $i > 1$.  

We always consider modules up to isomorphism, and this point of view has important effects on exposition:
The word ``unique'' means ``unique up to isomorphism''; when $X$ and $Y$ are isomorphic, we write $X = Y$ and consider $X$ and $Y$ to be equal; and in a ``set of modules'', distinct elements are assumed to be non-isomorphic.

We now describe how $\Lambda$ specifies a Cartan matrix and a Coxeter element of the associated Coxeter group.
We denote by $\{S_1,\ldots,S_n\}$ the set of simple modules. 
The endomorphism rings of a simple module $S_i$ and its projective cover $P_i$ are isomorphic division $\field$-algebras.
For $i \in \{1,\ldots,n\}$, we set $d_i = \dim_\field(\End(S_i))$. 
For $i, j \in \{1,\ldots,n\}$, $\Ext^1(S_i,S_j)$ carries the structure of an $(\End(S_j),\End(S_i))$-bimodule.
We set
\[a_{i,j} = \begin{cases} 2 & \text{if $i = j$}\\-\dim_{\End(S_i)}\Ext^1(S_i,S_j)-\dim_{\End(S_i)}\Ext^1(S_j,S_i) & \text{if $i \neq j$}.\end{cases}\]
Now clearly $d_ia_{i,j} = d_j a_{j,i}$ for all $i, j$. Thus $A = (a_{i,j})_{1 \leq i,j \leq n}$ is a symmetrizable (crystallographic, generalized) Cartan matrix in the sense of Section~\ref{nc sec}. 
Conversely, every symmetrizable (crystallographic, generalized) Cartan matrix arises from a finite-dimensional hereditary algebra, 
see e.g. \cite[Proposition~VIII.6.7]{ARS}.

We identify the Grothendieck group $K_0(\mods\Lambda)$ with the root lattice $\integers\Phi\subseteq V$ by identifying the class of $S_i$ with $\alpha_i$. 
We write $\undim X$ for the class of a module $X$ and call this the \newword{dimension vector}. 
Explicitly, the $\alpha_i$-coordinate of $\undim X$ is
\[(\undim X)_i = \dim_{\End(P_i)}\Hom(P_i,X).\]
Note that $(\undim X)_i$ is also the multiplicity of $S_i$ as a composition factor of $X$. 

Order the simple reflections $S$ as $s_{i_1},\ldots,s_{i_n}$ such that $\Ext^1(S_{i_j},S_{i_k})=0$ whenever $j>k$.
This condition does not uniquely determine the total order, but does completely determine the Coxeter element $c = s_{i_1}\cdots s_{i_n}$, because $s_is_j=s_js_i$ whenever $\Ext^1(S_i,S_j)=\Ext^1(S_j,S_i) = 0$.  
(This convention and the opposite convention for defining $c$, i.e.\ requiring $\Ext^1(S_{i_j},S_{i_k})=0$ for $j<k$, both appear in the literature.)
Writing $\tau$ for the Auslander-Reiten translate of $\mods\Lambda$, our convention implies that $\undim \tau X = c(\undim X)$ for any module $X$ which does not have a nonzero projective direct summand.

We conclude this subsection with some comments on quivers and their representations. We refer to \cite{ASS,ARS} for additional background information about quiver representations and to \cite{DlabRingel} for additional background information about the more general $\field$-species.

An \newword{acyclic quiver} is a finite acyclic directed graph $Q$. The ($\field$-)\newword{path algebra} of an acyclic quiver $Q$ is the finite-dimensional hereditary algebra $\field Q$ with basis (as a vector space) the set of paths in $Q$ and multiplication given by concatenation of paths (or 0 when concatenation is not possible). A ($\field$-)\newword{representation} $X$ of $Q$ is the data of a vector space $X(i)$ for each vertex $i$ of $Q$ and a linear map $X(a): X(i) \rightarrow X(j)$ for each arrow $i \xrightarrow{a} j$ of $Q$. With the correct notion of morphism, the category of finite-dimensional representations of $Q$ is then equivalent to the category $\mods \field Q$. Under this equivalence, the simple modules $S_i$ are in bijection with the vertices of $Q$ (and as representations consist of a 1-dimensional vector space at a single vertex and 0-dimensional vector spaces at the other vertices). The dimension vector of a module $X$ then records the dimensions of the vector spaces comprising the corresponding representation.

Now if $\field$ is algebraically closed, then every finite-dimensional hereditary algebra is Morita equivalent to a path algebra. 
Moreover, each path algebra specifies a \emph{symmetric} Cartan matrix.
Since we wish to consider Cartan matrices that are symmetrizable but not necessarily symmetric, it is necessary to work more generally over hereditary algebras. 
There are also generalizations of quivers and their path algebras known as \newword{$\field$-species} and their \newword{tensor algebras}, and every symmetrizable Cartan matrix is specified by the tensor algebra of a $\field$-species for some (finite) field $\field$. 
Moreover, if $\field$ is a perfect field, then every finite-dimensional hereditary $\field$-algebra is Morita equivalent to the tensor algebra of a $\field$-species. 
Over arbitrary fields, however, there exist finite-dimensional algebras which are not Morita equivalent to tensor algebras of species, see e.g. \cite[Theorem~1]{DlabRingelTame}.

\subsection{Exceptional sequences}\label{ex sec}

A module $X$ is \newword{indecomposable} if ${X \neq 0}$ and any direct sum decomposition $X = X' \oplus X''$ has $X' = 0$ or $X'' = 0$.

We write $\W \subseteq \mods\Lambda$ to mean that $\W$ is a subcategory of $\mods\Lambda$. By this, we will always mean that $\W$ is a full additive subcategory which is closed under isomorphisms. To specify a subcategory $\W \subseteq \mods\Lambda$, it therefore suffices to specify the set of indecomposable modules which lie in $\W$. We say that $\W$ is \newword{representation-finite} if this set is finite. Otherwise, we say that $\W$ is \newword{representation-infinite}. 

Certain types of modules will play an important role throughout the paper.
A module $X \in \mods\Lambda$ is said to be a \newword{brick} if $\End(X)$ is a division ring, \newword{rigid} if $\Ext^1(X,X) = 0$, and \newword{exceptional} if it is both rigid and a brick.
The symbol~$\excep$ will stand for the set of exceptional modules. 
Furthermore, for $\W \subseteq \mods\Lambda$ a subcategory, $\brick(\W)$ will be the set of bricks that lie in $\W$.  

\begin{definition}\label{def:exceptional_sequence}  
Let $\mathcal{B}$ be a set of bricks.
A \newword{$\mathcal{B}$-brick sequence} is a sequence $(X_1,\ldots,X_k)$ of bricks in $\B$ satisfying $\Hom(X_i,X_j) = 0 = \Ext^1(X_i,X_j)$ for all $1 \leq j < i \leq k$. 
A $\mathcal{B}$-brick sequence $(X_1,\ldots,X_k)$ is \newword{maximal} if there does not exist an index $i \in \{1,\ldots,k+1\}$ and a brick $Y \in \mathcal{B}$ such that $(X_1,\ldots,X_{i-1},Y,X_i,\ldots,X_k)$ is a $\mathcal{B}$-brick sequence.
If $\W \subseteq \mods\Lambda$ is a subcategory, then a $\mathcal{B}$-brick sequence $(X_1,\ldots,X_k)$ is \newword{maximal in $\W$} if it is a maximal $(\mathcal{B} \cap \brick(\W))$-brick sequence.
 
An $\excep$-brick sequence is called an \newword{exceptional sequence}. 
An exceptional sequence $(X_1,\ldots,X_k)$ is called \newword{complete} if $k = n$.
\end{definition}

We will see later in Proposition~\ref{prop:complete_ex_maximal} that maximal and complete exceptional sequences coincide.
The main results of this paper concern para-exceptional sequences in a connected tame algebra, which are $\B$-brick sequences for $\B$ slightly larger than~$\E$. Para-exceptional sequences also include the ``soft exceptional sequences'' in tube categories studied in \cite{IgusaSen}, see Sections~\ref{sec:one tube} and~\ref{sec:rep_theory_tubes}  for details.

The study of exceptional sequences of modules over hereditary algebras was initiated in \cite{CB,RingelBraid}. 
These papers in particular establish a transitive braid group action on the set of complete exceptional sequences. 
This action, and the analogous action on $T$-words, are the key to the proof of the following theorem, which follows from \cite[Theorem~4.1]{IgusaSchiffler} and \cite[Lemma~6.2]{Krause}. (Note that \cite[Theorem~4.1]{IgusaSchiffler} is proved under the assumption that the field $\field$ is algebraically closed, but the proof is readily adapted to arbitrary fields. See also \cite[Proposition~4.6]{HuberyKrause}, \cite[Corollary~6.4]{Krause}, or \cite[Corollary~3.8.1.2]{RingelCatalan}.)

\begin{theorem}\label{thm:IS}
The map $(X_1,\ldots,X_n) \mapsto t_{\undim X_1}\cdots t_{\undim X_n}$ is a bijection from the set of complete exceptional sequences of $\mods\Lambda$ to the set of reduced $T$-words for~$c$.
\end{theorem}

It is a consequence of Theorem~\ref{thm:IS} that the map $X \mapsto t_{\undim X}$ is a bijection from~$\excep$ to the set $\{t \in T \mid t \leq_T c\}$ of reflections in the noncrossing partition poset. Given a positive real root $\gamma$ with $t_\gamma \leq_T c$, we write $X_{t_\gamma}$ or~$X_\gamma$ for the unique (up to isomorphism) exceptional module with $\undim X_\gamma = \gamma$.

\subsection{Wide subcategories}\label{sec:subcats}

A subcategory $\W \subseteq \mods\Lambda$ is called \newword{wide} if it is closed under extensions, kernels, and cokernels.   
(These are sometimes called \newword{thick} subcategories in the literature.)

\begin{remark}\label{rem:exact}
Wide subcategories can equivalently be defined as exact-embedded abelian subcategories. 
In particular, they are abelian categories in their own right.   
Requiring that the inclusion $\W \subseteq \mods\Lambda$ be an exact embedding (or equivalently that $\W$ be closed under extensions) then says that $\Ext^1(-,-)|_\W = \Ext^1_\W(-,-)$, where the left hand side is the restriction of the Ext-functor of $\mods\Lambda$ and the right hand side is the Ext-functor computed using the inherited abelian structure of $\W$. 
This allows us to unambiguously write $\Ext^1(X,Y)$ for $X, Y \in \W$ without needing to specify the category in which this functor is constructed. 
Note also that the hereditary property is inherited by wide subcategories; that is, the higher Ext-functors likewise satisfy $0 = \Ext^i(-,-)|_\W = \Ext^i_\W(-,-)$.
\end{remark}

We write $\wide \Lambda$ for the poset of wide subcategories ordered by inclusion. 
It is well known, and easily checked, that $\wide \Lambda$ is a complete lattice with meet operation given by intersection. 
The maximum and minimum elements of this lattice are the whole category $\mods\Lambda$ and the subcategory $0$ consisting of only the zero module. 
A lower interval $[0,\W]$ in this lattice is precisely the lattice of wide subcategories of $\W$ (viewing $\W$ as an abelian category in its own right).

Since $\wide \Lambda$ is a complete lattice and the meet is intersection, for any subcategory $\X \subseteq \mods\Lambda$, there is a smallest wide subcategory $\sW(\X)$ containing~$\X$, 
and given modules $X_1,\ldots,X_k$, there is a smallest wide subcategory $\sW(X_1,\ldots,X_k)$ containing them.
A wide subcategory $\W$ is called \newword{exceptional} if there exists an exceptional sequence $(X_1,\ldots,X_k)$ such that $\W = \sW(X_1,\ldots,X_k)$.   We write $\ewide \Lambda$ for the subposet of $\wide \Lambda$ consisting of the exceptional subcategories. 
This subposet is generally not a lattice, see e.g. \cite[Example~3.2.3]{RingelCatalan} or Corollary~\ref{lattice iff}.

Our interest in $\ewide \Lambda$ is due to it relationship to the noncrossing partition poset. 
Given $\W \in \ewide\Lambda$, order the modules that are simple in $\W$ into an exceptional sequence $(X_1,\ldots,X_k)$ and define $\phi(\W) = t_{\undim X_1}\cdots t_{\undim X_k}$. The following theorem is proved in \cite[Section~3.2]{IngTho} for path algebras of finite or affine type, as \cite[Theorem~4.3]{IgusaSchiffler} for $\field$ algebraically closed, and as part of \cite[Corollary~7.5]{HuberyKrause} in general.

\begin{theorem}\label{thm:NC}The map $\phi$ is a poset isomorphism from $\ewide\Lambda$ to $[1,c]_T$.
\end{theorem}

For $X \in \mods\Lambda$, the \newword{additive closure} $\add(X)$ of $X$ is the smallest subcategory of $\mods\Lambda$ containing $X$.
Its indecomposable modules are precisely the indecomposable direct summands of $X$. 

The \newword{right-perpendicular category} of a subcategory $\W\subseteq\mods\Lambda$ is 
\[\W^\perp = \{X \in \mods\Lambda \mid \forall W \in \W: \Hom(W,X) = 0 = \Ext^1(W,X)|\}.\]
The \newword{left-perpendicular category} of $\W$ is 
\[{}^\perp \W = \{X \in \mods\Lambda \mid \forall W \in \W: \Hom(X,W) = 0 = \Ext^1(X,W)\}.\]
We treat $(-)^\perp$ and ${}^{\perp}(-)$ like exponents and intersections like products for order of operations. For example, ${}^\perp \W \cap \V = ({}^\perp \W) \cap \V$.

The \newword{extension-closure} of $\W$ is the subcategory $\Filt(\W)$ that consists of all modules $X$ admitting a finite filtration $0 = X_0 \subseteq X_1 \subseteq \cdots \subseteq X_k = X$ with $X_i/X_{i-1} \in \W$ for all $i$.
We also extend each of these definitions to a sequence of modules $(X_1,\ldots,X_k)$
by replacing $\W$ with $\add(X_1 \oplus \cdots \oplus X_k)$, so that, for example, $(X_1,\ldots,X_k)^\perp=(\add(X_1 \oplus \cdots \oplus X_k))^\perp$.

An \newword{orthogonal decomposition} of a wide subcategory $\W$ is an identity $\W = \U \oplus \V$ such that $\U$ and $\V$ are wide subcategories, $\V \subseteq \U^\perp \cap {}^\perp \U$, and every $X \in W$ can be decomposed as $X = X_\U \oplus X_\V$ with $X_\U \in \U$ and $X_\V \in \V$. 
We say that $\W$ is \newword{connected} if every orthogonal decomposition $\W = \U \oplus \V$ has $\U = 0$ or $\V = 0$. 
The algebra $\Lambda$ is said to be \newword{connected} if $\mods\Lambda$ is connected as a wide subcategory.

A module $X$ in a wide subcategory $\W$ is \newword{simple in $\W$} if $\W$ contains no proper submodules of $X$.
We write $\simp(\W)$ for the set of modules that are simple in $\W$. 
A \newword{semibrick} is a set $\X$ of bricks such that $\Hom(X,Y)=0=\Hom(Y,X)$ for all $X\neq Y\in\X$.
The following is proved in \cite[Section~1.2]{RingelSpecies}.

\begin{proposition}\label{prop:semibrick}
The map $\W\mapsto\simp{\W}$ is a bijection from $\wide \Lambda$ to the set of semibricks in $\mods\Lambda$, with inverse $\X\mapsto\Filt(\X)$.
\end{proposition}

\begin{corollary}\label{cor:semibrick}
Let $X$ be a brick. Then $X$ is the unique brick in $\sW(X) = \Filt(X)$.
\end{corollary}

\begin{proof}
The fact that $\sW(X) = \Filt(X)$ follows from Proposition~\ref{prop:semibrick}.  
Now let ${Y \in \Filt(X)}$ be a brick. By the definition of $\Filt(X)$, it follows that there is an injection $\iota: X \rightarrow Y$ and a surjection $q: Y \rightarrow X$. The composition $\iota \circ q$ is nonzero, and so it is invertible by the definition of a brick. 
We conclude that $X = Y$.
\end{proof}

The following lemma is a straightforward extension of \cite[Proposition~1.1]{GL}.

\begin{lemma}\label{lem:ortho_wide}
Let $\X \subseteq \mods\Lambda$ be an arbitrary subcategory. 
Then $\X^\perp = (\sW(\X))^\perp$ and ${}^\perp \X = {}^\perp(\sW(\X))$. 
Moreover, both of these are wide subcategories.
\end{lemma}

\begin{remark}\label{ortho rem}
The operators $(-)^\perp$ and $^\perp(-)$ both yield order-reversing maps $\wide \Lambda \rightarrow \wide \Lambda$, but they are generally not anti-automorphisms. We will recall in Corollary~\ref{cor:anti_isom_fg} that they do, however, induce anti-automorphisms of $\ewide \Lambda$.
\end{remark}

We conclude by mentioning the following fact about wide subcategories, which is \cite[Theorem~A.4 and Remark~A.5]{HuberyKrause}.
\begin{lemma}\label{lem:rep_finite_exceptional}
Let $\W \in \wide \Lambda$. If $\W$ is representation-finite, then $\W$ is exceptional. 
\end{lemma}

\begin{remark}\label{fin lat remark}
When $\mods\Lambda$ is itself representation-finite, all of its wide subcategories are representation-finite and thus exceptional by Lemma~\ref{lem:rep_finite_exceptional}.
In particular, $\ewide\Lambda$ is a lattice in this case.
By Theorem~\ref{thm:NC}, we recover the fact that $[1,c]_T$ is a lattice when $W$ is finite \cite{BWlattice}.
\end{remark}

\section{Chain systems}\label{chain sys sec}
In this section, we define the notion of a chain system on an alphabet and show that a chain system is equivalent information to an edge-labeled poset satisfying certain conditions.
We also define a special kind of chain system called a binary chain system.
The section ends with some remarks on a motivating case, the theory of combinatorial Garside structures. 
Later in the paper, we use binary chain systems to formalize the connection between, on one side, reduced $T$-words for $c$ and noncrossing partitions and, on the other side, exceptional sequences and wide subcategories.
Having formalized that connection, we further use binary chain systems to extend the connection, in the affine case, and prove our main result:  A representation-theoretic construction of the McCammond-Sulway lattice. 

\subsection{Definition of a chain system}

Given an \newword{alphabet} (i.e.\ a set) $\A$, write $\A^*$ for the set of finite-length words in~$\A$ and write $\emptyset$ for the empty word.
We view~$\A$ as the subset of~$\A^*$ consisting of one-letter words.
Given $x,y\in\A^*$, write $xy$ for their concatenation.
Let $\C$ be a subset of $\A^*$.  
Let $\A_\C$ be the set of all letters occurring in words in $\C$.
We can harmlessly take $\A=\A_\C$ (as opposed to $\A\supsetneq\A_\C$), because letters in $\A$ that don't occur in $\C$ are irrelevant in what follows.
We allow~$\A$ and $\C$ to be infinite.
A \newword{prefix} of $\C$ is a word that is a prefix (initial segment) of a word in~$\C$.
A \newword{postfix} of $\C$ is a word that is a postfix (final segment) of a word in $\C$.
Write $\Pre(\C)$ for the set of prefixes of $\C$ and $\Post(\C)$ for the set of postfixes of $\C$.

Two prefixes $p_1,p_2\in\Pre(\C)$ are equivalent if there exists ${x\in\Post(\C)}$ such that $p_1x$ and $p_2x$ are both in~$\C$.
Write $\Preeq(\C)$ for the set of equivalence classes of prefixes of $\C$.
Similarly, two postfixes $x_1,x_2\in\Post(\C)$ are equivalent if there exists $p\in\Pre(\C)$ such that $px_1$ and $px_2$ are both in~$\C$.
Write $\Posteq(\C)$ for the set of equivalence classes of postfixes of $\C$.
We will use the symbol~$\equiv$ for both of these equivalence relations, trusting to context to make clear whether we are referencing the equivalence relation on prefixes or the equivalence relation on postfixes.

\begin{definition}\label{chain sys def}
The set $\C$ is a \newword{chain system} on $\A$ if it satisfies the following conditions.
\begin{enumerate}[label=\rm(\roman*), ref=(\roman*)]
\item \label{bound}
There is a finite upper bound on the length of words in $\C$.
\item \label{conv}
Given $P\in\Preeq(\C)$ and $X\in\Posteq(\C)$, the following are equivalent: 
\begin{enumerate}[label=\rm(\alph*), ref=(\alph*)]
\item \label{conv exists}
There exists $p\in P$ and $x\in X$ with $px\in\C$;
\item \label{conv all}
$px\in\C$ for all $p\in P$ and $x\in X$.
\end{enumerate}
\item \label{non subst}
If $p\in\Pre(\C)$, $a\in \A$, $w\in \A^*$, and $x\in\Post(\C)$ satisfy $pax,pwx\in\C$, then $w=a$.
\end{enumerate}
Suppose $\C$ and $\C'$ are chain systems.
An \newword{isomorphism} from $\C$ to $\C'$ is a bijection from $\A_\C$ to $\A_{\C'}$ that induces a bijection from $\C$ to $\C'$.
\end{definition}

\begin{remark}\label{something like a converse}
By definition, given $p\in\Pre(\C)$ and any two postfixes $x_1,x_2\in\Post(\C)$, if $px_1,px_2\in C$ then $x_1\equiv x_2$.
Similarly, given $x\in\Post(\C)$ and $p_1,p_2\in\Pre(\C)$, if $p_1x,p_2x\in\C$ then $p_1\equiv p_2$.
Condition \ref{conv} of Definition~\ref{chain sys def} provides converses to these facts:
If $p\in\Pre(\C)$ and $x_1,x_2\in\Post(\C)$ have $px_1\in\C$ and $x_1\equiv x_2$, then $px_2\in\C$; and 
if $p_1p_2\in\Pre(\C)$ and $x\in\Post(\C)$ have $p_1x\in\C$ and $p_1\equiv p_2$, then $p_2x\in\C$.
\end{remark}

The following useful lemma addresses what might have seemed like a missing condition in Definition~\ref{chain sys def}.

\begin{lemma}\label{maximal}
Suppose $\C$ satisfies Conditions \ref{bound} and~\ref{conv} in the definition of a chain system.
If $p\in\Pre(\C)$, $w\in \A^*$, and $x\in\Post(\C)$ satisfy $px,pwx\in\C$, then $w=\emptyset$.
\end{lemma}
\begin{proof}
Suppose $px,pwx\in\C$, so that $p\equiv pw$ is the equivalence relation on prefixes.
Writing $w^i$ for the concatenation of $i$ copies of $w$, we prove by induction that $pw^ix\in\C$.
The base case $i=0$ is given.
If $i>0$, then by induction $pw^{i-1}x\in\C$.
By Condition \ref{conv}, since $p\equiv pw$, we have $pww^{i-1}x=pw^ix\in\C$.
If $w\neq\emptyset$, this is a contradiction to Condition~\ref{bound}.
\end{proof}

We define maps $\post:\Preeq(\C)\to\Posteq(\C)$ and $\pre:\Posteq(\C)\to\Preeq(\C)$ as follows.
If $p\in P\in\Preeq(\C)$ and $px\in\C$, then $\post(P)$ is the class of $x$ in the equivalence relation on $\Post(\C)$.
If $x\in X\in\Posteq(\C)$ and $px\in\C$, then $\pre(X)$ is the class of $p$ in the equivalence relation on $\Pre(\C)$.

If $p_1,p_2\in P\in\Preeq(\C)$ then any $x\in\Post(\C)$ has $p_1x\in\C$ if and only if $p_2x\in\C$, by the definition of a chain system.
Thus $\post(P)$ is independent of the choice of a representative of $P$.
Choosing some $p\in P$, $\post(P)$ is independent of the choice of $x$ such that $px\in\C$.
As noted in Remark~\ref{something like a converse}, if $x_1,x_2\in\Post(\C)$ have $px_1,px_2\in\C$, then $x_1\equiv x_2$.
Thus $\post$ is a well defined map from $\Preeq(\C)$ to $\Posteq(\C)$.
The left-right dual argument shows that $\pre$ is a well defined map from $\Posteq(\C)$ to $\Preeq(\C)$.
Knowing that the maps are well defined, the following fact is immediate.

\begin{proposition}\label{chain sys bij}
If $\C$ is a chain system, then the map $\post:\Preeq(\C)\to\Posteq(\C)$ is a bijection with inverse $\pre$.
\end{proposition}

\begin{definition}\label{bin chain sys def}
A \newword{binary compatibility relation} on an alphabet $\A$ is a set of two-letter words in $\A^*$.
If $B$ and $B'$ are binary compatibility relations, respectively on alphabets~$\A$ and~$\A'$, an \newword{isomorphism} from $B$ to $B'$ is a bijection from~$\A$ to~$\A'$ that induces a bijection from $B$ to~$B'$.
Given a binary compatibility relation $B$, a \newword{$B$-sequence} is a word $a_1\cdots a_k\in \A^*$ such that $a_ia_j\in B$ for all $1\le i<j\le k$.
A \newword{maximal $B$-sequence} is a $B$-sequence $a_1\cdots a_k$ such that, for all $i\in\set{0,\ldots,k}$ and $a\in \A$, the word $a_1\cdots a_iaa_{i+1}\cdots a_k$ is not a $B$-sequence.
Write $\C_B$ for the set of maximal $B$-sequences.
A chain system $\C$ is a \newword{binary chain system} if there exists a binary compatibility relation $B$ such that $\C=\C_B$.
\end{definition}

Not every binary compatibility relation $B$ defines a binary chain system.
(For example, if $\A=\set{a,b,c}$ and $B=\set{bc}$, then $\C_B=\set{a,bc}$ fails Condition~\ref{non subst}.)
We will not consider the question of which binary compatibility relations define binary chain systems.
Instead, our interest is in the following fact, which is immediate from the definitions.

\begin{proposition}\label{bin isom}
Suppose $B$ and $B'$ are binary compatibility relations on alphabets~$\A$ and $\A'$ and $\omega:\A\to\A'$ is a bijection. 
\begin{enumerate}[label=\bf\arabic*., ref=\arabic*]
\item \label{already}
If $\C_B$ and $\C_{B'}$ are binary chain systems, then $\omega$ is an isomorphism from~$\C_B$ to~$\C_{B'}$ if and only if it is an isomorphism from $B$ to $B'$.
\item \label{not yet}
If $\omega$ is an isomorphism from $B$ to $B'$ and $\C_B$ is a binary chain system, then~$\C_{B'}$ is a binary chain system and $\omega$ is an isomorphism from~$\C_B$ to~$\C_{B'}$.
\end{enumerate}
\end{proposition}

\subsection{Chain systems and labeled posets}
The term ``chain system'' refers to what we will use $\C$ to construct:
a finite-height edge-labeled partially ordered set and a bijection from the set of its maximal chains to~$\C$, specifically the map sending a maximal chain to the word given by reading edge-labels on the chain from bottom to top.

We define a binary relation $\le_{\pre}$ on $\Preeq(\C)$.
Classes $P$ and $Q$ have $P\le_{\pre}Q$ if and only if there exists $p\in P$ and $q\in Q$ such that $p$ is a prefix of $q$ (that is, there exists $x\in \A^*$ such that $q=px$).
We also define a binary relation $\le_{\post}$ on $\Posteq(\C)$.
Classes $X$ and $Y$ have ${X\le_{\post}Y}$ if and only if there exists $x\in X$ and $y\in Y$ such that $y$ is a postfix of $x$ (that is, there exists $p\in \A^*$ such that $x=py$).

\begin{proposition}\label{chain sys poset}
Suppose $\C$ is a chain system.
\begin{enumerate}[label=\bf\arabic*., ref=\arabic*]
\item \label{pre partial}
$\le_{\pre}$ is a partial order on $\Preeq(\C)$ with finite height, a unique minimal element $\set{\emptyset}$, and a unique maximal element $\C$.
\item \label{isom}
The map $\post$ is an isomorphism from $(\Preeq(\C),\le_{\pre})$ to $(\Posteq(\C),\le_{\post})$ with inverse map $\pre$.
\item \label{edges pre}
If $P\covered_{\pre} Q$, given $p\in P$, there exists a unique $a\in \A$ with $pa\in Q$.
This $a$ depends only on $P$ and $Q$, not the choice of $p\in P$.
\item \label{edges post}
If $X\covered_{\post} Y$, given $y\in Y$, there exists a unique $a\in \A$ with $ay\in X$.
This $a$ depends only on $X$ and $Y$, not the choice of $y\in Y$.
\item \label{isom label}
The isomorphism $\post$ takes the edge-labeling of $(\Preeq(\C),\le_{\pre})$ obtained from Assertion~\ref{edges pre} to the edge-labeling of $(\Posteq(\C),\le_{\post})$ obtained from Assertion~\ref{edges post}.
\item \label{bottom unique conclusion}
If two maximal chains $P_0\covered_{\pre}\cdots\covered_{\pre}P_k$ and $Q_0\covered_{\pre}\cdots\covered_{\pre}Q_\ell$ have the same labels on the first $i$ edges \emph{from the bottom}, then $P_i=Q_i$.
\item \label{top unique conclusion}
If two maximal chains $P_0\covered_{\pre}\cdots\covered_{\pre}P_k$ and $Q_0\covered_{\pre}\cdots\covered_{\pre}Q_\ell$ have the same labels on the first $i$ edges \emph{from the top}, then $P_{k-i}=Q_{\ell-i}$.
\item \label{chains}
The map that takes a maximal chain of $(\Preeq(\C),\le_{\pre})$ to a word by reading labels, bottom to top, is a bijection from the set of maximal chains to~$\C$.
\end{enumerate}
\end{proposition}
\begin{proof}
It is immediate that $\le_{\pre}$ is reflexive and transitive, that 
$\set{\emptyset}\in\Preeq(\C)$ has $\set{\emptyset}\le_{\pre}P$ for all $P\in\Preeq(\C)$, and that 
$\C\in\Preeq(\C)$ has $P\le_{\pre}\C$ for all $P\in\Preeq(\C)$.

To prove that $\le_{\pre}$ is antisymetric, suppose $P,Q\in\Preeq(\C)$ have $P\le_{\pre}Q$ and $Q\le_{\pre}P$.
Thus there exist $p\in P$, $q\in Q$, and $x\in \A^*$ such that $q=px$ and $p'\in P$, $q'\in Q$, and $x'\in \A^*$ such that $p'=q'x'$.
There exists $z\in \A^*$ such that $pxz=qz\in\C$.
Since $p\equiv p'$, also $p'xz\in\C$, so that $q'x'xz\in\C$.
Since $q\equiv q'$, also $qx'xz\in\C$, so that $pxx'xz\in\C$.
But now, since both $pxz$ and $pxx'xz$ are in $\C$, Lemma~\ref{maximal} says that $xx'=\emptyset$.
Therefore $p=q$, so that $P=Q$.
We have established Assertion~\ref{pre partial}, except for the assertion on finite height.

If $P\le_{\pre}Q$, then there exist $p\in P$, $q\in Q$, and $x\in \A^*$ such that $px=q$.
Let $Y=\post(P)$ and $Z=\post(Q)$.
Then for any $y\in Y$ and $z\in Z$, the words $py$ and $qz=pxz$ are in $\C$.
Thus $y\equiv xz$, so that $xz\in Y$.
Since $z$ is a postfix of $xz$, we have $Y\le_{\post}X$.
This and the left-right dual argument establish Assertion~\ref{isom}.

Suppose $P\covered_{\pre} Q$ and let $p\in P$.
Since $P<_{\pre} Q$, there exists a nonempty word $x=a_1\cdots a_k\in \A^*$ such that $px\in Q$.
We will prove that $k=1$.
The equivalence classes of $p,pa_1,pa_1a_2,\ldots,px$ constitute a weakly increasing sequence in $(\Preeq(\C),\le_{\pre})$.
Since $P\covered_{\pre} Q$, this sequence contains only two equivalence classes.
Thus if $k>1$, then at least one $i\in\set{0,\ldots,k-1}$ has $pa_1\cdots a_i\equiv pa_1\cdots a_{i+1}$.
Let $y\in\Post(\C)$ have $pa_1\cdots a_iy\in\C$.
Then also $pa_1\cdots a_ia_{i+1}y\in\C$, contradicting Lemma~\ref{maximal}.
We see that $k=1$, so that $pa=q$ for some $a\in \A$.
If another element $\tilde a\in \A$ has $p\tilde a\in Q$, then choosing $y\in\Post(\C)$ with $qy\in\C$, we have $pay,p\tilde ay\in\C$, so $a=\tilde a$ because $\C$ is a chain system.
If $p'\in P$, then by the same proof, there is a unique $a'$ with $p'a'\in Q$.
Since~$\C$ is a chain system and $p'a'\equiv pa=q$, we have $p'a'y\in\C$.  
But since $p'\equiv p$, also $pa'y\in\C$, so $pa'\equiv q$, or in other words $pa'\in Q$.
By the uniqueness already proved, $a'=a$.
We have proved Assertion~\ref{edges pre}.
The left-right dual argument proves Assertion~\ref{edges post}.

Suppose $P\covered_{\pre} Q$ and let $a\in \A$ be the label associated to this edge.
Thus for $p\in P$, we have $pa\in Q$.
If $\post(Q)=Y$, and $y\in Y$, then $pay\in\C$ and thus $ay$ is a representative of the class $\post(P)$.
We see that $a$ is the label associated to the edge $\post(P)\covered\post(Q)$.
We have proved Assertion~\ref{isom label}.

Given a maximal chain $P_0\covered_{\pre}\cdots\covered_{\pre}P_k$ with label sequence $a_1\cdots a_k$ from bottom to top, Assertion~\ref{edges pre} implies that $a_1\cdots a_i\in P_i$.
Assertion~\ref{bottom unique conclusion} follows.
The left-right dual of this argument, using Assertion~\ref{edges post}, implies a version of Assertion~\ref{top unique conclusion} with ``$\covered_{\pre}$'' replaced by ``$\covered_{\post}$'' throughout.
Assertion~\ref{top unique conclusion} then follows by Assertion~\ref{isom label}.

Given a maximal chain $P_0\covered\cdots\covered P_k$ in $(\Preeq(\C),\le_{\pre})$, Assertion~\ref{edges pre} lets us build a word in $\C$ one letter at a time from the bottom to the top, and this word is precisely the word obtained by reading the labels on the chain from the bottom to the top.
Thus the map described in Assertion~\ref{chains} is well defined.
Assertion~\ref{bottom unique conclusion} (or separately Assertion~\ref{top unique conclusion}) implies that the map is one-to-one.

Given $w=a_1\cdots a_k\in\C$, let $P_i$ be the equivalence class of the prefix $a_1\cdots a_i$, so that  $P_0\le_{\pre}P_1\le_{\pre}\cdots\le_{\pre}P_k$.
If $P_{i-1}=P_i$ then $a_1\cdots a_{i-1}\equiv a_1\cdots a_i$, and arguing as in the proof of Assertion~\ref{edges pre}, we see that $a_1\cdots a_{i-1}a_i^\ell a_{i+1}\cdots a_k\in\C$ for all $\ell\ge0$, and this contradiction shows that $P_{i-1}<P_i$ for all~$i$.
Suppose there exists $Q\in\Preeq(\C)$ with $P_{i-1}\le_{\pre}Q\le_{\pre}P_i$.
Define $p_-=a_1\cdots a_{i-1}\in P_{i-1}$ and $p=a_1\cdots a_i\in P_i$ and let $y\in\Post(\C)$ have $py\in\C$.
Since $P_{i-1}\le_{\pre}Q$, there exists $q\in Q$ and $w\in \A^*$ such that $q=p_-w$.
Since $Q\le_{\pre}P_i$, there exists $q'\in Q$ and $w'\in \A^*$ such that $p=q'w'$.
Then $q'w'y\in\C$, and since $\C$ is a chain system and $q'\equiv q$, also $qw'y\in\C$, and therefore $p_-ww'y\in\C$.
But also $p_-a_iy=py\in\C$.
By the definition of a chain system, $ww'=a$, so either $w=\emptyset$ and $Q=P_{i-1}$ or $w'=\emptyset$ and $Q=P_i$.
We see that $P_{i-1}\covered_{\pre}P_i$ and that the label of this cover is $a_i$.
Thus $P_0\covered_{\pre}P_1\covered_{\pre}\cdots\covered_{\pre}P_k$ maps to $w$, and we have shown that the map from maximal chains to $\C$ is onto.
We have proved Assertion~\ref{chains}. 
Finally, since the lengths of the words in $\C$ have a finite upper bound (by definition), Assertion~\ref{chains} implies the finite height property of Assertion~\ref{pre partial}.
\end{proof}

Proposition~\ref{chain sys poset} shows that a chain system $\C$ determines an edge-labeled poset with labels in $\A$.
To make the notation less cluttered, we refer to this edge-labeled poset as $\P(\C)$, with partial order $\le$.
In light of Proposition~\ref{chain sys bij}, it is equally reasonable to realize $\P(\C)$ as $(\Preeq(\C),\le_{\pre})$ or $(\Posteq(\C),\le_{\post})$.
Conversely, we have the following proposition, suggested by Proposition~\ref{chain sys poset}.\ref{bottom unique conclusion} and Proposition~\ref{chain sys poset}.\ref{top unique conclusion}.

\begin{proposition}\label{poset to chain sys}
Suppose $\P$ is a poset with finite height having edge-labels in $\A$ satisfying the following properties.
\begin{enumerate}[label=\rm(\roman*), ref=(\roman*)]
\item \label{bottom unique}
If two maximal chains $x_0\covered\cdots\covered x_k$ and $y_0\covered\cdots\covered y_\ell$ have the same labels on the first $i$ edges \emph{from the bottom} for some $i\ge0$, then $x_i=y_i$.
\item \label{top unique}
If two maximal chains $x_0\covered\cdots\covered x_k$ and $y_0\covered\cdots\covered y_\ell$ have the same labels on the first $i$ edges \emph{from the top} for some $i\ge0$, then $x_{k-i}=y_{\ell-i}$.
\end{enumerate}
Then the set $\C(\P)$ of words obtained by reading edge-labels of maximal chains from bottom to top is a chain system.
The map sending $x\in\P$ to the set of label sequences of saturated chains with top element $x$ is an isomorphism from $\P$ to $(\Preeq(\C),\le_{\pre})$.
\end{proposition}

\begin{proof}
To avoid confusing the numbered conditions in this proposition with the numbered conditions in Definition~\ref{chain sys def}, we refer to the ``first condition of Definition~\ref{chain sys def}'', etc.\ and refer to conditions in this proposition as ``Condition~(i)'', etc.

The finite height of $\P$ implies the first condition of Definition~\ref{chain sys def}.

Condition~\ref{bottom unique} implies that for any $p\in\Pre(\C(\P))$, there is a unique saturated chain including the bottom element of $\P$ whose label sequence, reading from the bottom up, is~$p$.
Write $\max(p)$ for the top element of that chain.
Similarly, by Condition~\ref{top unique}, for any $x\in\Post(\C(\P))$, there is a unique saturated chain including the top element of $\P$ whose label sequence, reading up to the top, is~$x$.
Write $\min(x)$ for the bottom element of that chain.
Then $px\in\C(\P)$ if and only if $\max(p)=\min(x)$.
Therefore also, $p,p'\in\Pre(\C(\P))$ have $p\equiv p'$ if and only if $\max(p)=\max(p')$, and $x,x'\in\Post(\C(\P))$ have $x\equiv x'$ if and only if $\min(x)=\min(x')$.
The second condition of Definition~\ref{chain sys def} follows.

Finally, suppose for some label $a\in \A$ and some word $w=a_1\cdots a_k\in \A^*$, that there are maximal chains with labels $pax$ and $pwx$.
Then $\max(p)\covered\max(pa)$ and $\max(p)\covered\max(pa_1)\covered\cdots\covered\max(pa_1\cdots a_k)$.
But $pa\equiv pw$, so $\max(pa)=\max(pw)$, and therefore $k=1$ and $a_1=a$.
This is the third condition of Definition~\ref{chain sys def}.

We see that $\C(\P)$ is a chain system.
The isomorphism from $\P$ to $(\Preeq(\C),\le_{\pre})$ is now immediate.
\end{proof}

An \newword{isomorphism} between \emph{edge-labeled} posets $\P$ and $\P'$ is an isomorphism of posets $\eta:\P\to\P'$, in the usual sense, together with a bijection~$\lambda$ from the set of labels appearing on $\P$ to the set of labels appearing on $\P'$ such that, for each edge $x\covered y$ in $\P$ with label~$a$, the label on $\eta(x)\covered\eta(y)$ in $\P'$ is $\lambda(a)$.
The following proposition is immediate from Propositions~\ref{chain sys poset} and~\ref{poset to chain sys}.

\begin{proposition}\label{isom isom}
Two chain systems $\C$ and $\C'$ are isomorphic if and only if $\P(\C)$ and $\P(\C')$ are isomorphic as labeled posets.
\end{proposition}

\subsection{The shuffle product of chain systems}\label{shuffle sec}
Given two words $x=a_1a_2\cdots a_k$ and $y=b_1b_2\cdots b_\ell$ with $\set{a_1,a_2,\ldots,a_k}\cap\set{b_1,b_2,\ldots,b_\ell}=\emptyset$, a \newword{shuffle} of $x$ and $y$ is a word $z=c_1c_2\cdots c_{k+\ell}$ with the following two properties:
If we delete the letters in $\set{b_1,b_2,\ldots,b_\ell}$ from $z$, we obtain $x$.
If we delete the letters in $\set{a_1,a_2,\ldots,b_k}$ from $z$, we obtain $y$.
There are $\binom{k+\ell}{k}$ shuffles of $x$ and~$y$.
Write $x\shuffle y$ for the set of all shuffles of $x$ and~$y$.

Suppose $\C$ and $\C'$ are sets of words on \emph{disjoint} alphabets $\A$ and $\A'$.
The set
\[\C\shuffle\C'=\bigcup_{x\in\C,y\in\C'}x\shuffle y\]
of all shuffles of words in $\C$ with words in $\C'$ is called the \newword{shuffle product} of $\C$ and~$\C'$.
Given an element $z$ of $\C\shuffle\C'$, if $x\in\C$ and $y\in\C'$ are the words such that $z\in x\shuffle y$, then write $z|_\A$ for $x$ and $z|_{\A'}$ for $y$.
It is immediate that ${\Pre(\C\shuffle\C')=\Pre(\C)\shuffle\Pre(\C')}$ and ${\Post(\C\shuffle\C')=\Post(\C)\shuffle\Post(\C')}$.
Furthermore, if $p_1,p_2\in\Pre(\C\shuffle\C')$, then $p_1\equiv p_2$ as prefixes of $\C\shuffle\C'$ if and only if $p_1|_\A\equiv p_2|_\A$ as prefixes of $\C$ and $p_1|_{\A'}\equiv p_2|_{\A'}$ as prefixes of $\C'$.
The analogous statement is true for postfixes.

Given two edge-labeled posets $\P$ and $\P'$, we edge-label the product $\P\times\P'$ according to the following rules:
The label on an edge $(x,x')\covered(y,x')$ is the same as the label on $x\covered y$ in $\P$.
The label on an edge $(x,x')\covered(x,y')$ is the same as the label on $x'\covered y'$ in $\P'$.

The following proposition is immediate from the definitions.

\begin{proposition}\label{shuf prod}
Suppose $\C$ is a set of words in the alphabet $\A$, suppose $\C'$ is a set of words in the alphabet $\A'$, and suppose $\A\cap\A'=\emptyset$.
Then 
\begin{enumerate}[label=\bf\arabic*., ref=\arabic*]
\item \label{shuf chain sys}
$\C\shuffle\C'$ is a chain system if and only if $\C$ and $\C'$ are both chain systems.
In that case, $\P(\C\shuffle\C')$ and $\P(\C)\times\P(\C')$ are isomorphic as edge-labeled posets.
(The bijection on the set of edge labels is the identity map on $\A\cup\A'$.)  
\item \label{shuf bin}
Suppose $B$ and $B'$ are binary compatibility relations on $\A$ and $\A'$ respectively.
The following are equivalent.
\begin{enumerate}[label=\rm(\roman*), ref=(\roman*)]
\item $\C$ and $\C'$ are binary chain systems with binary compatibility relations $B$ and $B'$ respectively. \item
$\C\shuffle\C'$ is a binary chain system with binary compatibility relation ${B\cup B'\cup\set{aa':a\in\A,a'\in\A'}\cup\set{a'a:a\in\A,a'\in\A'}}$.
\end{enumerate}
\end{enumerate}
\end{proposition}

\subsection{Chain systems and combinatorial Garside structures}\label{Garside sec}
Our interest in labeled posets, and thus in chain systems, is motivated in part by the theory of combinatorial Garside structures.
For a gentle introduction to Garside structures in a slightly more restricted setting, see~\cite{McCammondGarside}, or for more generality, see \cite[Section~2]{McSul}.  
One difference between the two references is that \cite{McSul} relaxes the requirement that the poset be graded.  
We give a simpler definition that is less restrictive than \cite{McCammondGarside} but more restrictive than \cite{McSul}.
Suppose each label has a weight and that the set of all weights is a finite subset of the positive real numbers.
The weight of a maximal chain is the sum of the weights of its labels.
We will say that a labeled poset is \newword{weighted-graded} there exists such an assignment of weights with the property that every maximal chain has the same weight. 
A \newword{combinatorial Garside structure} is a weighted-graded labeled poset that has a unique minimal element and a unique maximal element and finite height, that it is balanced and group-like, and that is a lattice.
We close this section by defining the balanced and group-like conditions in the language of chain systems and proving some simple results.
For the rest of the section,~$\C$ is a chain system.

The labeled poset $\P(\C)$ is \newword{balanced} if and only if $\Pre(\C)=\Post(\C)$.

Given $x\le y\in\P(\C)$, define the \newword{restriction} $\C|_{[x,y]}$ of $\C$ to $[x,y]$ to be the set of words obtained by reading labels, from $x$ up to $y$ on saturated chains between $x$ and $y$.
It follows from Proposition~\ref{poset to chain sys} that $\C|_{[x,y]}$ is a chain system whenever $x\le y\in\P(\C)$.
The labeled poset $\P(\C)$ is \newword{group-like} if and only if the following condition holds whenever $x\le y\le z$ and $x'\le y'\le z'$:
If two of the sets $\C|_{[x,y]}$, $\C|_{[x,z]}$ and $\C|_{[y,z]}$ are equal to their counterparts with primed symbols, then so is the third.
We give a simpler and formally stronger condition that implies that $\P(\C)$ is group-like.
Say that $\C$ has the \newword{restriction property} if whenever $x\le y$ and $x'\le y'$ in $\P(\C)$ and $\C|_{[x,y]}$ and $\C|_{[x',y']}$ are not disjoint, then $\C|_{[x,y]}=\C|_{[x',y']}$.

\begin{proposition}\label{res prop}
If $\C$ is a chain system with the restriction property, then $\P(\C)$ is group-like.
\end{proposition}
\begin{proof}
Suppose $\C$ has the restriction property and $x\le y\le z$ and $x'\le y'\le z'$ in~$\P(\C)$.
If $\C|_{[x,y]}=\C|_{[x',y']}$ and $\C|_{[y,z]}=\C|_{[y',z']}$, then for any $w_1\in\C|_{[x,y]}$ and $w_2\in\C|_{[y,z]}$, the concatenation $w_1w_2$ is in $\C|_{[x,z]}$ and in $\C|_{[x',z']}$.
By the restriction property, $\C|_{[x,z]}=\C|_{[x',z']}$.

Suppose $\C|_{[x,y]}=\C|_{[x',y']}$ and $\C|_{[x,z]}=\C|_{[x',z']}$.
Since $\C|_{[x,z]}=\C|_{[x',z']}$, the labeled posets $\P(\C|_{[x,z]})$ and $\P(\C|_{[x',z']})$ coincide, so $[x,z]$ and $[x',z']$ are isomorphic as labeled posets.
We describe the isomorphism explicitly using the description in Proposition~\ref{poset to chain sys}.
Fix a saturated chain from~$\emptyset$ to~$x$, with label sequence $w_1$ and a saturated chain from~$\emptyset$ to~$x'$ with label sequence $w'_1$.
For every element $u$ of $[x,z]$, take a saturated chain from $x$ to $u$ with label sequence $w_2$.
The isomorphism sends $u$ to the unique element $u'$ of $[x',z']$ such that $w'_1w_2$ labels a chain from $\emptyset$ to $u'$
Since also $\C|_{[x,y]}=\C|_{[x',y']}$, the isomorphism sends $y$ to $y'$.
Suppose $w_2$ is such that $w_1w_2$ labels a saturated chain from $\emptyset$ to $y$, so that $w'_1w_2$ labels a chain from $\emptyset$ to $y'$.
For any $w_3$ such that $w_1w_2w_3$ labels a chain from $\emptyset$ to $z$, the word $w'_1w_2w_3$ labels a chain from $\emptyset$ to $z'$.
On the other hand, for any $w'_3$ such that $w'_1w_2w'_3$ labels a chain from $\emptyset$ to $z'$, the word $w_1w_2w'_3$ labels a chain from $\emptyset$ to~$z$.
We see that $\C|_{[y,z]}=\C|_{[y',z']}$.

If instead $\C|_{[x,z]}=\C|_{[x',z']}$ and $\C|_{[y,z]}=\C|_{[y',z']}$, then we argue similarly, considering saturated chains ending at $\C$ rather than starting at $\emptyset$.
\end{proof}

\section{The chain system of exceptional sequences}\label{ex chain sec}
In this section and the next, we reinterpret Theorems~\ref{thm:IS} and~\ref{thm:NC} in terms of binary chain systems.
This section is devoted to showing that the complete exceptional sequences in $\mods\Lambda$ form a binary chain system and using exceptional subcategories to realize the prefix (equivalently postfix) order on this chain system.
This mostly amounts to collecting results from \cite{CB,HuberyKrause,IgusaSchiffler,RingelBraid}. 
We also show that the corresponding labeled poset is a combinatorial Garside structure whenever the poset of exceptional subcategories is a lattice.
This fact is a representation-theoretic version of an observation of Michel that is recorded in \cite[Theorem~0.5.2]{Bessis} and the paragraphs before and after that theorem.

Consider the binary compatibility relation $\Bx$ on the alphabet $\excep$ consisting of the two-term exceptional sequences $(X,Y)$.  
We write $\Cx$ for the set~$\C_{\Bx}$ of maximal $\Bx$-sequences in the sense of Definition~\ref{bin chain sys def}.
The set $\Cx$ is precisely the set of maximal exceptional sequences in the sense of Definition~\ref{def:exceptional_sequence}.
The following proposition says that~$\Cx$ also coincides with the set of complete exceptional sequences. 

\begin{proposition}\label{prop:complete_ex_maximal}
Let $(X_1,\ldots,X_k)$ be an exceptional sequence. Then the following are equivalent.
\begin{enumerate}[label=\rm(\roman*), ref=(\roman*)]
\item \label{cm1} $(X_1,\ldots,X_k) \in \Cx$, that is, $(X_1,\ldots,X_k)$ is maximal.
\item \label{cm2} $(X_1,\ldots,X_k)$ is complete, that is, $k = n$.
\item \label{cm3} $\sW(X_1,\ldots,X_k) = \mods\Lambda$.
\item \label{cm4} $(X_1,\ldots,X_k)^\perp = 0$.
\item \label{cm5} ${}^\perp(X_1,\ldots,X_k) = 0$.
\end{enumerate}
\end{proposition}

Proposition~\ref{prop:complete_ex_maximal} mostly follows from results of \cite{CB}, which as observed in \cite{RingelBraid} hold even when $\field$ is not algebraically closed.
See also \cite[Corollary~A.7]{HuberyKrause}.
We state one of these results here for future reference, namely \cite[Lemma~1]{CB}.

\begin{lemma}\label{lem:exceptional_index_arbitrary}
If $(X_1,\ldots,X_k)$ is an exceptional sequence, then $k \leq n$ and for any $0 \leq i \leq k$ there exist $Y_1,\ldots,Y_{n-k}$ such that $(X_1,\ldots,X_i,Y_1,\ldots,Y_{n-k},X_{i+1},\ldots,X_k)$ is a complete exceptional sequence.
\end{lemma}

\begin{proof}[Proof and citations for Proposition~\ref{prop:complete_ex_maximal}]
The equivalence \ref{cm1}$\iff$\ref{cm2} is immediate from Lemma~\ref{lem:exceptional_index_arbitrary}.
The implication \ref{cm2}$\implies$\ref{cm3} is \cite[Lemma~3]{CB}.
The implications \ref{cm3}$\implies$\ref{cm4} and \ref{cm3}$\implies$\ref{cm5} follow from Lemma~\ref{lem:ortho_wide}.
We demonstrate the contrapositive of \ref{cm4}$\implies$\ref{cm1}:
If $(X_1,\ldots,X_k)$ is not maximal, then Lemma~\ref{lem:exceptional_index_arbitrary} says that there is an exceptional sequence $(Y,X_1,\ldots,X_k)$. 
Thus by definition, $Y \in (X_1,\ldots,X_k)^\perp$, so \ref{cm4} fails. 
The implication \ref{cm5}$\implies$\ref{cm1} is analogous. 
\end{proof}

In light of Proposition~\ref{prop:complete_ex_maximal}, we freely use ``maximal exceptional sequence'' to mean ``complete exceptional sequence'' when referencing results from the literature.

We will use the following characterization of exceptional subcategories to prove an extension of Proposition~\ref{prop:complete_ex_maximal} (Proposition~\ref{prop:maximal_exceptional}) and two further consequences (Corollaries~\ref{cor:middle_exceptional} and~\ref{cor:anti_isom_fg}).
See also \cite[Lemma~4]{CB}.

\begin{proposition}\label{prop:wide_fg}
Let $\W$ be a wide subcategory. 
Then the following are equivalent.
\begin{enumerate}[label=\rm(\roman*), ref=(\roman*)]
\item \label{wfg1} $\W$ is exceptional. 
\item \label{wfg2} $\W = (X_1,\ldots,X_\ell)^\perp$ for some exceptional sequence $(X_1,\ldots,X_\ell)$.
\item \label{wfg3} $\W = {}^\perp(X_1,\ldots,X_\ell)$ for some exceptional sequence $(X_1,\ldots,X_\ell)$. 
\item \label{wfg4} There exists a finite-dimensional hereditary algebra $\Lambda'$ admitting an exact equivalence $\W \simeq \mods\Lambda'$.
\end{enumerate}
\end{proposition}

\begin{proof}[Proof and citations]
The equivalence of \ref{wfg1}--\ref{wfg3} can be found as part of \cite[Theorem~A.4]{HuberyKrause}. The implication \ref{wfg1}$\implies$\ref{wfg4} is \cite[Lemma~5]{CB}. 
The implication \ref{wfg4}$\implies$\ref{wfg1} follows from applying Proposition~\ref{prop:complete_ex_maximal} in the category $\mods\Lambda'$. 
\end{proof}

\begin{proposition}\label{prop:maximal_exceptional}
Let $(X_1,\ldots,X_k)$ be an exceptional sequence.   
Then each of the following conditions is equivalent to $(X_1,\ldots,X_k)$ being maximal.
\begin{enumerate}[label=\rm(\roman*), ref=(\roman*)]
\item For $0\le i\le j\le k$, $\sW(X_{i+1},\ldots,X_j) = {}^\perp(X_1,\ldots,X_i) \cap (X_{j+1},\ldots,X_k)^\perp$.
\item There exist $i\le j$ with $\sW(X_{i+1},\ldots,X_j) = {}^\perp(X_1,\ldots,X_i) \cap (X_{j+1},\ldots,X_k)^\perp$.
\end{enumerate}
\end{proposition}

\begin{proof}
Fix indices $i \leq j \in \set{0,\ldots,k}$ and let $\W = {}^\perp(X_1,\ldots,X_i) \cap (X_{j+1},\ldots,X_k)^\perp$.
Lemma~\ref{lem:exceptional_index_arbitrary} implies that $(X_1,\ldots,X_k)$ is maximal if and only if $(X_{i+1},\ldots,X_j)$ is maximal in $\W$. 
Proposition~\ref{prop:wide_fg} says that $\W\simeq\mods\Lambda'$ for some finite-dimensional hereditary algebra $\Lambda'$. 
Applying Proposition~\ref{prop:complete_ex_maximal} in the category $\mods\Lambda'$, we see that $(X_{i+1},\ldots,X_j)$ is maximal in $\W$ if and only if $\sW(X_{i+1},\ldots,X_j) = \W$. 
\end{proof}

\begin{corollary}\label{cor:middle_exceptional}
If $(X_1,\ldots,X_k)$ is an exceptional sequence, then for $i \in \set{0,\ldots,k}$, ${}^{\perp}(X_1,\ldots,X_i) \cap (X_{i+1},\ldots,X_k)^\perp$ is an exceptional subcategory .
\end{corollary}

\begin{proof}
Let $i \in \set{0,\ldots,k}$. 
By Lemma~\ref{lem:exceptional_index_arbitrary}, there exist modules $Y_1,\ldots,Y_j$ such that $(X_1,\ldots,X_i,Y_1,\ldots,Y_j,X_{i+1},\ldots,X_k)$ is a maximal exceptional sequence. 
Thus ${}^{\perp}(X_1,\ldots,X_{i-1}) \cap (X_i,\ldots,X_k)^\perp = \sW(Y_1,\ldots,Y_j)$ by Proposition~\ref{prop:maximal_exceptional}.
\end{proof}

\begin{corollary}\label{cor:anti_isom_fg}
The map $\W\mapsto\W^\perp$ is an anti-automorphism of the poset $\ewide \Lambda$, with inverse $\W \mapsto{}^\perp\W$.
\end{corollary}

\begin{proof}
Let $\W \in \ewide \Lambda$, and specifically, let $(X_1,\ldots,X_k)$ be an exceptional sequence such that $\W = \sW(X_1,\ldots,X_k)$.
We will show that $\W^\perp \in \ewide \Lambda$ and that ${}^\perp(\W^\perp) = \W$. 
The proof that ${}^\perp\W \in \ewide \Lambda$ and $({}^\perp\W)^\perp = \W$ is analogous.
Proposition~\ref{prop:wide_fg} and Lemma~\ref{lem:ortho_wide} imply that there is an exceptional sequence $(Y_1,\ldots,Y_i)$ with $\sW(Y_1,\ldots,Y_i) = \W^\perp = (X_1,\ldots,X_k)^\perp$.
Thus,  by Proposition~\ref{prop:maximal_exceptional}, $(Y_1,\ldots,Y_i,X_1,\ldots,X_k)$ is a maximal exceptional sequence. 
Applying Proposition~\ref{prop:maximal_exceptional} and Lemma~\ref{lem:ortho_wide} once more, we conclude that $\W = \sW(X_1,\ldots,X_k) = {}^\perp(Y_1,\ldots,Y_i) = {}^\perp(\W^\perp)$.
\end{proof}

The following proposition and theorem are the main results of the section.

\begin{proposition}\label{prop:perp_prefix}
Let $(X_1,\ldots,X_k)$ and $(X'_1,\ldots,X'_i)$ be exceptional sequences. Then the following are equivalent.
\begin{enumerate}[label=\rm(\roman*), ref=(\roman*)]
\item \label{pp1} $(X_1,\ldots,X_k)$ and $(X'_1,\ldots,X'_i)$ are equivalent as prefixes in $\Cx$.
\item \label{pp2} $(X_1,\ldots,X_k)$ and $(X'_1,\ldots,X'_i)$ are equivalent as postfixes in $\Cx$.
\item \label{pp3} $\sW(X_1,\ldots,X_k) = \sW(X'_1,\ldots,X'_i)$.
\item \label{pp4} $(X_1,\ldots,X_k)^\perp = (X'_1,\ldots,X'_i)^\perp$.
\item \label{pp5} ${}^\perp(X_1,\ldots,X_k) = {}^{\perp}(X'_1,\ldots,X'_i)$.
\end{enumerate}
\end{proposition}

\begin{proof}
If \ref{pp1} holds, then there exist $Y_1,\ldots,Y_j$ such that $(X_1,\ldots,X_k,Y_1,\ldots,Y_j)$ and $(X'_1,\ldots,X'_i,Y_1,\ldots,Y_j)$ are maximal exceptional sequences. Proposition~\ref{prop:maximal_exceptional} then implies that $\sW(X_1,\ldots,X_k) = (Y_1,\ldots,Y_j)^\perp = \sW(X'_1,\ldots,X'_i)$, so \ref{pp3} holds.

If \ref{pp3} holds, then \ref{pp5} holds by Lemma~\ref{lem:ortho_wide}.

Suppose \ref{pp5} holds.
By Lemma~\ref{lem:exceptional_index_arbitrary}, there exists a maximal exceptional sequence $(X_1,\ldots,X_k,Y_1,\ldots,Y_j)$. 
Proposition~\ref{prop:maximal_exceptional} and \ref{pp5} together say that $\sW(Y_1,\ldots,Y_j) = {}^{\perp}(X_1,\ldots,X_k) = {}^{\perp}(X'_1,\ldots,X'_i)$.
By \ref{pp5}, $(X'_1,\ldots,X'_i,Y_1,\ldots,Y_j)$ is also an exceptional sequence.  Proposition~\ref{prop:maximal_exceptional} implies that $(X'_1,\ldots,X'_i,Y_1,\ldots,Y_j)$ is maximal. 
We have shown that \ref{pp1}$\implies$\ref{pp3}$\implies$\ref{pp5}$\implies$\ref{pp1}.
A similar argument proves that \ref{pp2}$\implies$\ref{pp3}$\implies$\ref{pp4}$\implies$\ref{pp1}.
\end{proof}

\begin{theorem}\label{thm:exceptional_binary_chain}
Let $\Lambda$ be a finite-dimensional hereditary algebra. 
\begin{enumerate}[label=\bf\arabic*., ref=\arabic*]
\item \label{Cx bcs}
The set $\Cx$ of maximal exceptional sequences is a binary chain system.
\item \label{Pc to fwide}
The map $\Preeq(\Cx) \rightarrow \ewide \Lambda$ that sends $P \in \Preeq(\Cx)$ to $\sW(X_1,\ldots,X_k)$, for any representative $(X_1,\ldots,X_k) \in P$, is an isomorphism of posets.
\item \label{Cx labels cov}
If $\V,\W\in\ewide\Lambda$ have $\W\subseteq\V$, then $\V$ covers $\W$ in $\ewide\Lambda$ if and only if $\V\cap{}^\perp\W$ contains a unique brick.
\item \label{Cx labels}
If $P\covered P'$ in $\Preeq(\Cx)$ is labeled by $X$ and sent by the  isomorphism in Assertion~\ref{Pc to fwide} to $\W \covered \V$, then $X$ is the unique brick in $\V \cap {}^\perp \W$.
\end{enumerate}
\end{theorem}

\begin{proof}
We will prove that $\Cx$ is a chain system, and it is then immediate that $\Cx$ is a binary chain system.
First, Proposition~\ref{prop:complete_ex_maximal} implies Condition \ref{bound} in Definition~\ref{chain sys def}.
To prove Condition~\ref{conv}, let $\X\in \Preeq(\Cx)$ and $\Y\in \Posteq(\Cx)$ and suppose $(X_1,\ldots,X_k)\in\X$ and $(Y_1,\ldots,Y_\ell)\in\Y$ have  $(X_1,\ldots,X_k,Y_1,\ldots,Y_\ell) \in \Cx$.
For any $(X'_1,\ldots,X'_{k'})\in\X$ and $(Y'_1,\ldots,Y'_{\ell'})\in\Y$, 
Proposition~\ref{prop:perp_prefix} says that $\sW(X'_1,\ldots,X'_{k'})=\sW(X_1,\ldots,X_k)$, which equals $(Y_1,\ldots,Y_\ell)^\perp$ by Proposition~\ref{prop:maximal_exceptional}.
Thus applying Lemma~\ref{lem:ortho_wide}, Proposition~\ref{prop:perp_prefix}, and then Lemma~\ref{lem:ortho_wide} again, this is
\[((\sW(Y_1,\ldots,Y_\ell))^\perp= (\sW(Y'_1,\ldots,Y'_{j'}))^\perp= (Y'_1,\ldots,Y'_{j'})^\perp.\]
Since $\sW(X'_1,\ldots,X'_{k'})=(Y'_1,\ldots,Y'_{j'})^\perp$, the sequence $(X'_1,\ldots,X'_{k'},Y'_1,\ldots,Y'_j)$ is exceptional, and is in $\Cx$ by Proposition~\ref{prop:maximal_exceptional}. 
This is Condition~\ref{conv}.

Condition \ref{non subst} in Definition~\ref{chain sys def} follows from \cite[Lemma~2]{CB}, but we give an argument here. 
Suppose that $(X_1,\ldots,X_n)$ is a maximal exceptional sequence and that there exists an index $1 \leq i \leq n$ and exceptional modules $(Y_1,\ldots,Y_j)$ such that $(X_1,\ldots,X_{i-1},Y_1,\ldots,Y_j,X_{i+1},\ldots,X_n)$ is a maximal exceptional sequence. 
Proposition~\ref{prop:maximal_exceptional} implies that $\sW(X_i) = \sW(Y_1,\ldots,Y_j)$.
Corollary~\ref{cor:semibrick} says that $X_i$ is the unique brick in $\sW(X_i)$, so $j=1$ and $Y_1=X_i$.
This is Condition \ref{non subst} in Definition~\ref{chain sys def}, and we have proved Assertion~\ref{Cx bcs}.

We prove Assertion~\ref{Pc to fwide} using an argument similar to the proof of \cite[Theorem~4.3]{IgusaSchiffler}. 
The fact that the map from $\Preeq(\Cx)$ to $\ewide\Lambda$ is well-defined and one-to-one follows immediately from Proposition~\ref{prop:perp_prefix}.
Moreover, Lemma~\ref{lem:exceptional_index_arbitrary} implies that any exceptional sequence is a prefix of a word in $\Cx$, and so the map is bijective. 
It remains to show that order relations are preserved. 
Let $\W, \V \in \ewide\Lambda$ with corresponding equivalence classes $P_\W, P_\V \in \Preeq(\Cx)$. 
It is clear that $P_\W \leq_\pre P_\V$ implies $\W \subseteq \V$. 
Thus suppose that $\W \subseteq \V$.
Let $(X_1,\ldots,X_k)$ be an exceptional sequence such that $\W = \sW(X_1,\ldots,X_k)$. 
By Proposition~\ref{prop:wide_fg}, there exists an exceptional sequence $(Y_1,\ldots,Y_j)$ such that $\V = (Y_1,\ldots,Y_j)^\perp$. 
Since $\W \subseteq \V$, it follows that $(X_1,\ldots,X_k,Y_1,\ldots,Y_j)$ is an exceptional sequence. 
Lemma~\ref{lem:exceptional_index_arbitrary} lets us extend this to a maximal exceptional sequence $(X_1,\ldots,X_k,Z_1,\ldots,Z_i,Y_1,\ldots,Y_j)$. 
Then Proposition~\ref{prop:perp_prefix} implies that $\sW(X_1,\ldots,X_k,Z_1,\ldots,Z_i) = (Y_1,\ldots,Y_j)^\perp = \V$.
Thus $(X_1,\ldots,X_k,Z_1,\ldots,Z_i) \in P_\V$, and so $P_\W \leq_\pre P_\V$.
This is Assertion~\ref{Pc to fwide}.

Suppose $\V,\W\in\ewide\Lambda$ have $\W\subseteq\V$.
By Assertion~\ref{Pc to fwide}, there is a maximal exceptional sequence $(X_1,\ldots,X_n)$ and indices $i\le j$ such that ${\W = \sW(X_1,\ldots,X_i)}$ and $\V = \sW(X_1,\ldots,X_j)$. 
Arguing as in the previous paragraph, we see that ${{}^\perp \W \cap \V= {}^\perp(X_1,\ldots,X_{i-1}) \cap (X_{j+1},\ldots,X_n)^\perp = \sW(X_i,X_{i+1},\ldots,X_j)}$.
Thus each of $X_i,X_{i+1},\ldots,X_j$ is a brick in $\V\cap{}^\perp\W$. Now if $\V\cap{}^\perp \W$ contains a unique brick, then $j = i$ and so $\V \covered \W$ by Assertion~\ref{Pc to fwide}. Conversely, suppose that $\V \covered \W$. Then $i = j$, the cover relation is labeled by $X_i$, and Corollary~\ref{cor:semibrick} implies that $X_i$ is the unique brick in $\sW(X_i) = \V\cap {}^\perp \W$. $X_i$ is also the label.
We have thus proved Assertions~\ref{Cx labels cov} and~\ref{Cx labels}.
\end{proof}

We conclude this section by showing that $\P(\Cx)$ is a combinatorial Garside structure whenever $\ewide \Lambda$ is a lattice.
Recall from Section~\ref{chain sys sec} the notation $\C|_{[x,y]}$ for $x\le y$ in $\P(\C)$.
The following lemma can be seen as an extension of Theorem~\ref{thm:exceptional_binary_chain}.\ref{Cx labels cov}.
We use the notation $\W\mapsto P_\W$ for the inverse of the isomorphism in Theorem~\ref{thm:exceptional_binary_chain}.\ref{Pc to fwide}.

\begin{lemma}\label{lem:interval_labels_excep}
Suppose $\W\subseteq\V$ in $\ewide\Lambda$.
Then $\Cx|_{[P_\W,P_\V]}$ is the set of exceptional sequences $(X_1,\ldots,X_k)$ such that $\sW(X_1,\ldots,X_k) = {}^\perp\W\cap\V$.
\end{lemma}

\begin{proof}
Suppose first that $(X_1,\ldots,X_k) \in \Cx|_{[P_\W,P_\V]}$. 
By Theorem~\ref{thm:exceptional_binary_chain}, there exists a maximal exceptional sequence $(Y_1,\ldots,Y_i,X_1,\ldots,X_k,Z_1,\ldots,Z_j)$ such that $\sW(Y_1,\ldots,Y_i) = \W$ and $\sW(Y_1,\ldots,Y_i,X_1,\ldots,X_k) = \V$. 
Arguing as in the proof of Theorem~\ref{thm:exceptional_binary_chain}.\ref{Cx labels cov}, we have
$\sW(X_1,\ldots,X_k) = \V \cap {}^\perp\W$, as desired.

Conversely, suppose that $(X_1,\ldots,X_k)$ is an exceptional sequence such that $\sW(X_1,\ldots,X_k) = {}^\perp\W \cap \V$. Since $\W,\V\in\ewide\Lambda$, there is an exceptional sequence $(Y_1,\ldots,Y_i)$ with $\W = \sW(Y_1,\ldots,Y_i)$ and Proposition~\ref{prop:wide_fg} says there is an exceptional sequence $(Z_1,\ldots,Z_j)$ with $\V = (Z_1,\ldots,Z_j)^\perp$. 
Lemma~\ref{lem:ortho_wide} and Proposition~\ref{prop:maximal_exceptional} imply that $(Y_1,\ldots,Y_i,X_1,\ldots,X_k,Z_1,\ldots,Z_j)$ is a maximal exceptional sequence with $\sW(Y_1,\ldots,Y_i,X_1,\ldots,X_k) = \V$. 
Thus $(X_1,\ldots,X_k)\in\Cx|_{[P_\W,P_\V]}$.
\end{proof}

\begin{corollary}\label{cor:garside_if_lattice}
Suppose that $\ewide \Lambda$ is a lattice. Then $\P(\Cx)$ is a combinatorial Garside structure.
\end{corollary}

\begin{proof}
Proposition~\ref{prop:complete_ex_maximal} says that all maximal exceptional sequences have the same length. It follows that $\P(\Cx)$ is (weighted-)graded.
By Proposition~\ref{chain sys poset} and Theorem~\ref{thm:exceptional_binary_chain}.\ref{Cx bcs}, $\P(\Cx)$ has finite height, a unique minimal element, and a unique maximal element. 
Lemma~\ref{lem:exceptional_index_arbitrary} implies that $\P(\Cx)$ is balanced.
Suppose $P_\W\le P_\V$ and $P_{\W'}\le P_{\V'}$ are order relations in $\P(\Cx)$ such that ${\Cpx|_{[P_\W,P_\V]} \cap \Cpx|_{[P_{\W'},P_{\V'}]} \neq \emptyset}$.
Lemma~\ref{lem:interval_labels_excep} implies that $\V \cap {}^\perp\W = \V' \cap {}^\perp\W'$, but then applying Lemma~\ref{lem:interval_labels_excep} again, we see that $\Cx|_{[P_\W,P_\V]} = \Cx|_{[P_{\W'},P_{\V'}]}$.
We see that $\Cx$ has the restriction property, so $\P(\Cx)$ is group-like by Proposition~\ref{res prop}.
Finally $\P(\Cx) \cong \ewide \Lambda$ by Theorem~\ref{thm:exceptional_binary_chain}.\ref{Pc to fwide}, and $\ewide \Lambda$ is a lattice by hypothesis.
\end{proof}

\section{The chain system of the noncrossing partition poset}\label{nc chain sec}
We now describe the realization of the noncrossing partition poset as a labeled poset associated to a binary chain system.  
We could obtain this characterization quickly from Theorems~\ref{thm:IS} and~\ref{thm:exceptional_binary_chain}, but instead describe the chain system directly in terms of Coxeter groups and roots systems, in order to highlight the one piece that is missing from a proof entirely in terms of Coxeter groups and roots systems.
Much of the material here is already known, and the primary novelty in this section is using the formalism of chain systems.
This formulation in terms of chain system is in preparation for proving the new affine results in the language of chain systems.  

Recall from Section~\ref{nc sec} that, given a word whose letters are the elements of $S$, each appearing exactly once, the product of the word is called a \newword{Coxeter element}.
The given word is the \newword{defining word} for the Coxeter element.
We fix a Coxeter element $c$, and we fix a defining word for $c$.
Let $\C_c$ be the set of all reduced $T$-words for~$c$.
Recall from Section~\ref{nc sec} that every reduced $T$-word for $c$ has the same finite length $|S|=n$.
Given a $T$-word $x$, we write $\Pi x$ for the product of the sequence.  

\begin{lemma}\label{nc pre eq}
Let $W$ be a Coxeter group and let $c$ be a Coxeter element of~$W$.
The equivalence relation on prefixes of $\C_c$ is $p_1\equiv p_2$ if and only if $\Pi p_1=\Pi p_2$.
Equivalent prefixes have the same length.
\end{lemma}
\begin{proof}
If $p_1$ and $p_2$ are prefixes of $\C_c$ then by definition, they are equivalent if and only if there exists a postfix $x$ of $\C_c$ such that both $p_1x$ and $p_2x$ are in $\C_c$.
If there exists such an $x$, then $\Pi p_1\cdot\Pi x=c$ and $\Pi p_2\cdot\Pi x=c$, so $\Pi p_1=\Pi p_2$.
Conversely, if $\Pi p_1=\Pi p_2$, then since $p_1$ and $p_2$ are both prefixes, they have the same length, because if, say, $p_2$ is shorter, then we could substitute $p_2$ for $p_1$ in a word in $\C_c$, contradicting the fact that the word is reduced.
Furthermore, there exists a postfix~$x$ such that $p_1x\in C_c$, and since $\Pi p_1=\Pi p_2$, also $p_2x$ is a $T$-word for $c$, also reduced because  $p_1$ and $p_2$ have the same length.
That is $p_2x\in\C_c$.
\end{proof}

The proof of the following lemma is left-right dual to the proof of Lemma~\ref{nc pre eq}.
\begin{lemma}\label{nc post eq}
Let $W$ be a Coxeter group and let $c$ be a Coxeter element of~$W$.
The equivalence relation on postfixes of $\C_c$ is $x_1\equiv x_2$ if and only if $\Pi x_1=\Pi x_2$.
Equivalent postfixes have the same length.
\end{lemma}

\begin{proposition}\label{Twords is chain sys}
Let $W$ be a Coxeter group and let $c$ be a Coxeter element of~$W$.
Then set $\C_c$ is a chain system on the alphabet $T$.
\end{proposition}
\begin{proof}
Every word in $\C_c$ has length~$n$.
Suppose $P\in\Preeq(\C_c)$ and $X\in\Posteq(\C_c)$ and suppose there exists $p\in P$ and $x\in X$ with $px\in\C_c$, so that $px$ is a reduced $T$-word for $c$. 
Lemmas~\ref{nc pre eq} and~\ref{nc post eq} imply that $p'x'$ is a reduced $T$-word for $c$ for any $p'\in P$ and $x'\in X$.
That is, $p'x'\in\C_c$.
If $p\in\Pre(\C_c)$, $a\in T$, $w\in T^*$, and $x\in\Post(\C_c)$ satisfy $pax,pwx\in\C_c$, then $\Pi p\cdot\Pi a\cdot\Pi x=c=\Pi p\cdot\Pi w\cdot\Pi x$, and therefore $w=a$.
\end{proof}

Combining Propositions~\ref{chain sys poset} and~\ref{Twords is chain sys}, we obtain an edge-labeled poset $\P(\C_c)$, defined on $\Preeq(\C_c)$ or on $\Posteq(\C_c)$, with labels in~$T$.
Lemma~\ref{nc pre eq} implies the following proposition.
Thus, moving forward, we will refer to $[1,c]_T$ rather than~$\P(\C_c)$.

\begin{proposition}\label{PCc nc}
Let $W$ be a Coxeter group and let $c$ be a Coxeter element of~$W$.
The map $p\mapsto\Pi p$ induces an isomorphism from $\P(\C_c)$, considered as a poset on $\Preeq(\C_c)$, to the noncrossing partition poset $[1,c]_T$.
The edge $P\covered Q$ is labeled by the unique element $t\in T$ such that $\Pi(pt)\in Q$ for any $p\in P$.
\end{proposition}

We will describe $\C_c$ in more detail.
One important tool is a natural action of the braid group on $n$ strands on $T$-words of length~$n$.
Taking generators $\sigma_1,\ldots,\sigma_n$ such that $\sigma_i$ braids strands $i$ and $i+1$, the action of $\sigma_i$ on a $T$-word $t_1\cdots t_n$ is to replace the $2$-letter subsequence $t_it_{i+1}$ with the $2$-letter subsequence $(t_it_{i+1}t_i)t_i$.

Another important tool is a bilinear form $E_c$ defined on $V$ using the Cartan matrix and the defining word for the Coxeter element~$c$.
\[
E_c(\alpha\ck_i,\alpha_j)=\begin{cases}
a_{ij}&\text{if }s_i\text{ follows }s_j\text{ in }c,\\
1&\text{if }i=j,\text{ or}\\
0&\text{if }s_i\text{ precedes }s_j\text{ in }c.
\end{cases}
\]
The form $K$ given by the Cartan matrix is the symmetrization of $E_c$, i.e.\ $K(x,y)=E_c(x,y)+E_c(y,x)$ for all $x,y\in V$.
As a consequence, it is easy to check that $E_c(\beta\ck,\beta)=E_c(\beta,\beta\ck)=1$ for every real root $\beta$. 
(See \cite[Lemma~2.4]{affdenom}.)

The form $E_c$ is the same as the form $\br{\,\cdot\,,\,\cdot\,}$ that appears in \cite{HuberyKrause}, except that the left argument of $E_c$ is the right argument of $\br{\,\cdot\,,\,\cdot\,}$.
(This can be seen in the proof of \cite[Lemma~3.1]{HuberyKrause}, where the form is defined in terms of the entries of the Cartan matrix.)
In this paper, we reserve the notation $\br{\,\cdot\,,\,\cdot\,}$ for the pairing between a vector space and its dual.
The form $E_{c^{-1}}$ on $\integers\Phi$ coincides, via the map $X\mapsto \undim X$, with the Euler form on the Grothendieck group: 
\[E_{c^{-1}}(\undim X, \undim Y) = \dim_\field\Hom(X,Y) - \dim_\field \Ext^1(X,Y).\]

The following theorem gathers results of Igusa and Schiffler~\cite{IgusaSchiffler}, Hubery and Krause~\cite{HuberyKrause}, and others, tied together with simple arguments.
It amounts to four equivalent descriptions of $\C_c$ (one of which is the definition).

\begin{theorem}\label{Ec nc} 
Let $W$ be a crystallographic Coxeter group and let $c$ be a Coxeter element of~$W$.
Suppose $t_1\cdots t_k$ is a $T$-word and $\gamma_1,\ldots,\gamma_k$ are the corresponding positive real roots.  
Then the following are equivalent.
\begin{enumerate}[label=\rm(\roman*), ref=(\roman*)]
\item \label{red Tword}
$t_1\cdots t_k$ is a reduced $T$-word for $c$.
\item \label{max pairwise le}   
Both of the following conditions hold:
Every letter $t_i$ has $t_i\le_Tc$; and $t_1\cdots t_k$ is maximal with respect to the property that $1<t_it_j\le_Tc$ for all  $i$ and $j$ with $1\le i<j\le k$. 
(That is, the sequence $t_1\cdots t_k$ has the property and for all $u\in T$ and all $i\in\set{0,1,\ldots,k}$, the sequence $t_1\cdots t_iut_{i+1}\cdots t_k$ does not have the property.)
\item \label{complete backwards}
$E_{c^{-1}}(\gamma_i,\gamma_j)=0$ for all $i$ and $j$ with $1\le j<i\le k$ and $\gamma_1,\ldots,\gamma_k$ is a $\integers$-basis for the root lattice $\integers\Phi$.
\item \label{braid}
$t_1\cdots t_k$ is obtained from the defining word for $c$, as a product of the elements of~$S$, by the action of the braid group.
\end{enumerate}
\end{theorem}

\begin{proof}[Proof and citations]
Conditions~\ref{complete backwards} and~\ref{braid} immediately imply that $k=n$.
The fact that $k=n$ in~\ref{red Tword}, as already mentioned, is Dyer~\cite[Theorem~1.1]{DyerLength}.
A priori, sequences satisfying~\ref{max pairwise le} might be longer or shorter than~$n$, and the only proof we know that they have length $n$ is the concatenation of results described here.

The definition of the braid group action implies that \ref{braid}$\implies$\ref{red Tword}.
The implication \ref{red Tword}$\implies$\ref{braid} is \cite[Theorem~1.4]{IgusaSchiffler}, which is proved using the combinatorics of Coxeter groups.
A simpler combinatorial proof was given later in~\cite{BDSW}.
The implication \ref{braid} $\implies$ \ref{complete backwards} is the concatenation of \cite[Proposition~2.6]{HuberyKrause} and \cite[Lemma~2.7(1)]{HuberyKrause}.

Condition~\ref{complete backwards} coincides with the definition, in \cite[Sections~2--3]{HuberyKrause} of a complete real exceptional sequence $\gamma_1,\ldots,\gamma_k$.
(In \cite{HuberyKrause}, exceptional sequences may contain negative roots, but this difference with \ref{complete backwards} is inconsequential.)
By the definition of $E_c$, we have $E_{c^{-1}}(\gamma_i,\gamma_j)=0$ for all $j<i$ when $t_1\cdots t_n$ is the defining word for~$c$.  
The simple roots are a basis for $\integers\Phi$, so the defining word for~$c$ satisfies \ref{complete backwards}.
By \cite[Proposition~2.4]{HuberyKrause} (in light of \cite[Lemma~3.1]{HuberyKrause}), all sequences $t_1\cdots t_k$ satisfying \ref{complete backwards} have the same product, so all sequences satisfying \ref{complete backwards} are reduced $T$-words for~$c$.
In other words, \ref{complete backwards} $\implies$ \ref{red Tword}.

These quoted results show the equivalence of \ref{red Tword}, \ref{complete backwards}, and \ref{braid} using only the combinatorics of Coxeter groups and root systems.
(In \cite{HuberyKrause}, this combinatorics appears in the guise of generalized Cartan lattices \cite[Section~3]{HuberyKrause}.)
However, the equivalence of \ref{max pairwise le} with these other conditions is currently only known through representation-theory, using Theorem~\ref{thm:IS}.
Specifically, the equivalence of~\ref{red Tword} and~\ref{max pairwise le} is \cite[Lemma~6.3]{HuberyKrause}.
\end{proof}

\begin{remark}\label{would be nice}
The proof quoted here of Theorem~\ref{Ec nc} suggests the problem of giving a completely combinatorial proof.  
Presumably, such a proof would hold without the crystallographic condition, either leaving out condition~\ref{complete backwards} or finding a stand-in for requiring a $\integers$-basis for the root lattice.
(A possible stand-in is suggested by~\cite{BaumeisterWegener}, which shows that in the case where $\Phi$ is finite, a set of roots spanning the root lattice lattice is equivalent to the corresponding reflections generating the group.)
\end{remark}

In light of Proposition~\ref{Twords is chain sys}, the equivalence of \ref{red Tword} and \ref{max pairwise le} in Theorem~\ref{Ec nc} can be restated as the following corollary.

\begin{corollary}\label{chain sys nc} 
Let $W$ be a crystallographic Coxeter group and let $c$ be a Coxeter element of~$W$.
Then $\C_c$ is a binary chain system on the alphabet $\set{t\in T:t\le_T c}$, and its associated binary compatibility relation is $\set{tt':1<_Ttt'\le_Tc}$.
\end{corollary}

The following fact is immediate from Theorems~\ref{thm:NC} and ~\ref{thm:exceptional_binary_chain} and Corollary~\ref{chain sys nc}. 

\begin{corollary}\label{cor:bijection_on_chain_systems_exceptional}
Let $\Lambda$ be a finite-dimensional hereditary algebra with Coxeter element $c$. 
Then the map $X \mapsto t_{\undim X}$ is an isomorphism of binary chain systems. 
In particular, $(X,Y)$ is an exceptional sequence if and only if $1 <_T t_{\undim X}t_{\undim Y} \leq_T c$.
\end{corollary}

\section{The McCammond-Sulway lattice and its chain system}\label{McSul chain sec}
The Cartan matrix $A$ is of \newword{finite type} if and only if it is positive definite.
This is if and only if the corresponding Coxeter group $W$ is finite, and again equivalently, the root system $\Phi$ is finite.
When $W$ is finite, the noncrossing partition poset $[1,c]_T$ is a lattice~\cite{BWlattice}.
(See also Remark~\ref{fin lat remark}.)
The lattice property is crucial to a ``dual presentation'' of the associated Artin group of spherical type~\cite{Bessis,Bra-Wa}, which is the key to proving important properties of the Artin group.

The Cartan matrix $A$ is of \newword{affine type} if it is not of finite type, is positive semidefinite, and has the property that for every $i\in\set{1,\ldots,n}$, the submatrix of $A$ obtained by deleting row $i$ and column $i$ is of finite type.
In this case, $W$ is called an \newword{affine Coxeter group}.
When $A$ is of affine type, it has a $1$-dimensional kernel.
The kernel is $\integers\delta$ for a vector $\delta$ with nonnegative integer coordinates.
The imaginary roots of $\Phi$ are the nonzero integer multiples of $\delta$, and when there can be no confusion, we call $\delta$ simply the \newword{imaginary root}.
There is a finite root system $\Phi_\fin$ such that every root in $\Phi$ is a positive scaling of a vector of the form $\beta+m\delta$ for $\beta\in\Phi_\fin$ and $m\in\integers$.  

When $W$ is of affine type, $[1,c]_T$ is not always a lattice~\cite{Digne1,McFailure}.
In this case, work of McCammond and Sulway~\cite{McSul} constructs a lattice containing $[1,c]_T$ as an induced subposet, in order to prove of important properties of the associated Artin group of Euclidean type~\cite{McSul,PaoSal}.
The key idea is to enlarge the generating set $T$ to include some elements not in $W$ and to construct the poset analogous to $[1,c]_T$ in the supergroup of $W$ generated by the larger generating set.
The new elements are ``factored translations'', obtained by factoring the translations that appear in $[1,c]_T$.
We now give some details, using a modified construction from~\cite{affncA} that allows additional degrees of freedom in the definition but produces isomorphic results.
Proofs and additional details are found in \cite{McSul} and~\cite{affncA}.

In Section~\ref{nc chain sec}, we defined $W$ as a group of transformations of a vector space~$V$.
Dually, $W$ is a group of linear transformations of the dual vector space $V^*$.
When~$W$ is of affine type, writing $\br{\,\cdot\,,\,\cdot\,}:V^*\times V\to\reals$ for the natural pairing, the group fixes the affine hyperplane $E=\set{x\in V^*:\br{x,\delta}=1}$ in $V^*$ as a set.
The restriction of~$W$ to~$E$ is a group of Euclidean rigid motions of~$E$, generated by affine reflections.
We also consider the linear subspace $E_0=\set{x\in V^*:\br{x,\delta}=0}$ parallel to $E$, because elements of $W$ that act as translations on $E$ are described by vectors in~$E_0$.

The Coxeter element $c$ acts on $E$ as a translation along a line $U_0$, called the \newword{Coxeter axis}, composed with a rigid motion that fixes the Coxeter axis pointwise but fixes no other points.
A \newword{horiontal root} is a root that is orthogonal to the Coxeter axis, and a \newword{horizontal reflection} a reflection associated to a horizontal root. 
A \newword{vertical root} is a real root that is not horizontal and a  \newword{vertical reflection} a reflection associated to a vertical root.

The set $\Upsilon^c$ of horizontal roots is a root system in a broader sense than usual.
It is closed under the reflections defined by its roots and has a subset $\Xi^c$ analogous to the simple roots of a root system.
See \cite{Deodhar,DyerReflection} or \cite[Section~2.4]{typefree} for details. 

The root system $\Upsilon^c$ has, in essence, $0$, $1$, $2$, or $3$ irreducible components, each of which is of affine type $\widetilde{A}$.
(The case where $\Upsilon^c$ has $0$ components arises precisely when $W$ is affine of rank $2$.)
These are not orthogonal components in the usual sense because they all share the same imaginary roots, the integer multiples of~$\delta$.
However, if real roots $\beta,\gamma\in\Upsilon^c$ are in different components, then $K(\beta,\gamma)=0$.
Write $\Upsilon_1,\ldots,\Upsilon_m$ for these components, for $m\in\set{0,1,2,3}$.
It is only meaningful to factor translations when $m\ge2$ because these are precisely the cases where $[1,c]_T$ fails to be a lattice~\cite{Digne1,Digne2,McFailure}.
Write~$U_i$ for the linear span of $\set{K(\,\cdot\,,\beta):\beta\in\Upsilon_i}\subseteq V^*$.
Then $E_0$ is an orthogonal direct sum ${U_0\oplus\cdots\oplus U_m}$.
Thus if $w\in W$ acts on $E$ as a translation with translation vector $\lambda\in E_0$, we can write $\lambda=\lambda_0+\lambda_1+\cdots+\lambda_m$ with $\lambda_i\in U_i$ for $i=0,\ldots,m$.

We factor only those translations that are in the interval $[1,c]_T$.
We fix real numbers $q_1,\ldots,q_m$ summing to~$1$.
If $w\in[1,c]_T$ acts as a translation by a vector $\lambda\in E_0$, we define translations $f_1,\ldots,f_m$ with translation vectors $\lambda_i+q_i\lambda_0$ for $i=1,\ldots,m$.
Thus $w$ is the composition (in any order) of $f_1,\ldots,f_m$.
The set $F$ of \newword{factored translations} consists of all the factors $f_i$ associated to all translations~$w\in[1,c]_T$.
Each translation~$w$ in $[1,c]_T$ is a product of two reflections, and thus has length $\ell_T(w)=2$.
Accordingly, we assign each factored translation the length $\frac2m$, where~$m$ is the number of components of $\Upsilon^c$.

Each factored translation acting on $E$ extends uniquely to a linear transformation on $V^*$, so there is a well defined supergroup of $W$ generated by these transformations and the group $W$ of transformations of $V^*$.
Every element of the supergroup can be written as a word in the generators $T\cup F$, and thus we have analogous notions of reduced (i.e.\ minimal length) words, the prefix order, and the interval $[1,c]_{T\cup F}$.
This is a labeled poset, with each edge $v\covered w$ labeled by $v^{-1}w\in T\cup F$.
(Since the elements of $F$ are not involutions, conceivably $v^{-1}w$ might be the inverse of an element of $F$, but Proposition~\ref{if a fact}.\ref{all factors}, below, rules out that possibility.)
Different choices of the $q_i$ give isomorphic labeled posets \cite[Theorem~5.1]{affncA}, and this poset is a lattice \cite[Theorem~8.9]{McSul}.
The set of reduced words for $c$ in the alphabet $T\cup F$ is the set of label sequences on maximal chains of $[1,c]_{T\cup F}$, reading from the bottom.
Write $\C_c^+$ for this set of words.
We will see soon, as a consequence of Proposition~\ref{if a fact}, that $\C_c\subset\C_c^+$, for $\C_c$ as in Section~\ref{nc chain sec}.

We will show that $\C_c^+$ is a binary chain system.
As a first step, we describe the alphabet of the chain system $\C_c$ explicitly in the case where $W$ is of affine type.
Let~$T_V$ be the set of all vertical reflections.
Let $T_H$ be the set of all horizontal reflections whose corresponding root has non-full support in the basis of simple roots of~$\Phi$.
The following proposition is established in \cite[Section~5]{McSul}.
(The characterization of $T_H$ given here follows from the explanation at the end of \cite[Definition~5.5]{McSul}, which shows that for each horizontal root $\beta$, there are exactly two roots of the form $\beta+k\delta$ for $k\in\integers$ whose corresponding reflection is in $[1,c]_T$.
But there are exactly two choices of $k$ such that $\beta+k\delta$ has non-full support, and it is easy to see that the corresponding reflections are in $[1,c]_T$.)

\begin{proposition}\label{Cc alph}
Let $W$ be a Coxeter group of affine type and let $c$ be a Coxeter element of~$W$.
Then $\set{t\in T:t\le_Tc}=T_V\cup T_H$.
\end{proposition}

The action of $c$ permutes the simple roots $\Xi^c$ of $\Upsilon^c$, and this permutation of~$\Xi^c$ has exactly one cycle in each component of $\Upsilon^c$. We denote by $r_\beta$ the length of the cycle containing $\beta$; that is, $r_\beta$ is the smallest positive integer such that $c^{r_\beta}(\beta) = \beta$.
The positive roots orthogonal to reflections in $T_H$ are the roots of the form
\begin{equation}\label{eqn:beta_k}\beta_{(k)}=\beta+c(\beta)+\cdots+c^{k-1}(\beta)\end{equation}
for some $\beta\in\Xi^c$ and some positive integer $k<r_\beta$.
We write $t_{\beta,k}$ for $t_{\beta_{(k)}}$.

To state the main theorem, we define an equivalence relation $\equiv_\Upsilon$ on the simple roots $\Xi^c$ of $\Upsilon^c$ with $\beta\equiv_\Upsilon\gamma$ if and only if $\beta$ and $\gamma$ are in the same component of~$\Upsilon^c$.

\begin{theorem}\label{Ccplus chain sys}
Let $W$ be a Coxeter group of affine type and let $c$ be a Coxeter element of~$W$ such that $\Upsilon^c$ has more than one component.
The set $\C_c^+$ is a binary chain system with alphabet $T_V\cup T_H\cup F$.
There is a bijection $\beta\mapsto f_\beta$ from $\Xi^c$ to $F$ such that the binary compatibility relation defining $\C_c^+$ is the union of the following sets of two-letter words:   
\begin{align*}
&\bigcup_{\substack{\beta,\gamma\in\Xi^c,\,\beta\not\equiv_\Upsilon\gamma\\1\le k<r_\beta}}\bigl\{f_\beta 
f_\gamma,\,f_\gamma t_{\beta,k},\,t_{\beta,k}f_\gamma\bigr\},\\
&\bigl\{f_\gamma t_{\beta,k}:\beta,\gamma\in\Xi^c,\,\beta\equiv_\Upsilon\gamma\text{ and }\gamma\not\in\{c(\beta),c^2(\beta),\ldots,c^k(\beta)\}\bigr\},\\
&\bigl\{t_{\beta,k}f_\gamma:\beta,\gamma\in\Xi^c,\,\beta\equiv_\Upsilon\gamma\text{ and }\gamma\not\in \{\beta,c(\beta),\ldots,c^{k-1}(\beta)\}\bigr\},\\
&\bigl\{tt':t,t'\in(T_V\cup T_H),1<_Ttt'\le_Tc\bigr\}.
\end{align*}
\end{theorem}

\begin{remark}\label{it don't work}
When $\Upsilon^c$ has only one component, $F$ is the set of all translations in $[1,c]_T$.
In this case, $[1,c]_T$ is already a lattice~\cite{McFailure}, the factored translation construction is not needed, and the conclusion of Theorem~\ref{Ccplus chain sys} is false.
The issue is that each translation in $[1,c]_T$ is a product of two reflections, so for any word in $\C_c^+$ containing a letter in $F$, that letter can be replaced by two letters, violating the definition of a chain system.
Nevertheless, we have defined $\C_c^+$ in all cases, and the results proved below without hypotheses on $\Upsilon^c$ are true in all cases.
\end{remark}

We will prove Theorem~\ref{Ccplus chain sys} below.  
One important tool is the following proposition, essentially due to \cite{McSul} but gathered together as \cite[Proposition~5.4]{affncA}.

\begin{proposition}\label{if a fact}
For an arbitrary choice of real numbers $q_1,\ldots,q_m$ summing to~$1$, suppose a translation appears as one of the letters in a reduced word for $c$ in the alphabet $T\cup F\cup F^{-1}$.
\begin{enumerate}[label=\bf\arabic*., ref=\arabic*]
\item \label{fact horiz}
Each reflection in the word is horizontal.
\item \label{ref in int}
Each reflection in the word is contained in $[1,c]_T$.
\item \label{n-2 ref}
There are exactly $n-2$ reflections in the word.
\item \label{k trans} 
There are exactly $m$ translations in the word.
\item \label{all factors}
The translations in the word are the $m$ factors of some translation in $[1,c]_T$.
(In particular, they are elements of $F$, whereas \textit{a priori} they might only be inverse to elements of $F$.)
\item \label{reorder}
The reduced word can be reordered, by swapping letters that commute in the group, so that, for each~$i$, all reflections associated to $\Psi_i$ and all factored translations associated to $U_i$ are adjacent to each other.
\item \label{fact c}
For each $i$, let $c_i$ be the product of the subword consisting only of reflections associated to $\Psi_i$ and all factored translations associated to $U_i$.
Then $c_i$ depends only on $c$ and the constants $q_i$, not on the choice of reduced word. 
\end{enumerate}
\end{proposition}

Proposition~\ref{if a fact} implies in particular that reduced words for $c$ in the alphabet $T\cup F$ have length~$n$ (assigning, as described above, the length $\frac2m$ to each element of~$F$).
Thus $\C_c\subset\C_c^+$.

\begin{proposition}\label{Ccplus alph}
Let $W$ be a Coxeter group of affine type and let $c$ be a Coxeter element of~$W$ such that $\Upsilon^c$ has more than one component.
The set of letters appearing in words in $\C_c^+$ is $T_V\cup T_H\cup F$.
\end{proposition}
\begin{proof}
Since $\C_c^+$ contains $\C_c$, Proposition~\ref{Cc alph} implies that every letter in $T_V\cup T_H$ appears in $\C_c^+$.
The factored translations are obtained by factoring translations that are in $[1,c]_T$.
Given a translation $w\in[1,c]_T$ factored as $f_1\cdots f_m$, since $\ell_T(w)=2$, there is a $T$-word $t_1\cdots t_{n-2}$ such that $wt_1\cdots t_{n-2}=c$.
Thus $f_1\cdots f_mt_1\cdots t_{n-2}=c$, and we conclude that every letter in $F$ appears in $\C_c^+$.
Now Propositions~\ref{Cc alph}, \ref{if a fact}.\ref{ref in int}, and~\ref{if a fact}.\ref{all factors} imply that every letter appearing in $\C_c^+$ is in $T_V\cup T_H\cup F$.
\end{proof}

The following proposition is a restatement of Proposition~\ref{if a fact}.\ref{fact horiz}.
\begin{proposition}\label{Ccplus disj}
Let $W$ be a Coxeter group of affine type and let $c$ be a Coxeter element of~$W$ such that $\Upsilon^c$ has more than one component.
No word in $\C_c^+$ contains both a letter in $F$ and a letter in $T_V$.
\end{proposition}

The lattice $[1,c]_{T\cup F}$ is obtained from $[1,c]_T$ by adjoining finitely many elements, and we will describe them more specifically.

The \newword{factorable group} is the group generated by $T_H\cup F$.
We define the interval $[1,c]_{T_H\cup F}$ in the factorable group with the same notion of length as above.
Let $\C_c^F$ be the set of reduced words for $c$ in the alphabet $T_H\cup F$.
Thus $\C_c^F$ is the set of label sequences on maximal chains in $[1,c]_{T_H\cup F}$.
The interval $[1,c]_{T_H\cup F}$ is an $m$-fold direct product of factors $[1,c_i]_{T_H\cup F}$ for $c_i$ as in Proposition~\ref{if a fact}.\ref{fact c}.
Each factor is isomorphic, as a labeled poset, to the noncrossing partition lattice associated to a Coxeter group of finite type B of rank equal to the number of simple roots in $\Xi^c$ in the $i\th$ component of $\Upsilon^c$ \cite[Proposition~7.6]{McSul}.
(For a representation-theoretic proof, see Corollary~\ref{CcF cor}.)
As a consequence, Corollary~\ref{chain sys nc} implies the following.

\begin{proposition}\label{factorable bin}
Let $W$ be a Coxeter group of affine type and let $c$ be a Coxeter element of~$W$ such that $\Upsilon^c$ has more than one component.
Then $\C_c^F$ is a binary chain system with $\P(\C_c^F)\cong[1,c]_{T_H\cup F}$.
Its binary compatibility relation consists of the two-letter prefixes of label sequences of maximal chains in $[1,c]_{T_H\cup F}$.
\end{proposition}

Proposition~\ref{factorable bin} and Proposition~\ref{if a fact}.\ref{all factors}--\ref{fact c} imply the following proposition.

\begin{proposition}\label{factorable facts}
Let $W$ be a Coxeter group of affine type and let $c$ be a Coxeter element of~$W$ such that $\Upsilon^c$ has more than one component.
\begin{enumerate}[label=\bf\arabic*., ref=\arabic*]
\item 
If~$a,b\in T_H\cup F$ are associated to different components of $\Upsilon^c$, then $ab$ is in the binary compatibility relation for~$\C_c^F$.
\item
If $a,b\in T_H\cup F$ are associated to the same component $\Upsilon_i$ of~$\Upsilon^c$, then $ab$ is in the binary compatibility relation for~$\C_c^F$ if and only if $1\neq ab\in[1,c_i]_{T_H\cup F}$.
\item
If $a,b\in F$ and are associated to the same component of~$\Upsilon^c$, then $ab$ is not in the binary compatibility relation for~$\C_c^F$.
\end{enumerate}
\end{proposition}

The lattice $[1,c]_{T\cup F}$ is the union of $[1,c]_T$ and $[1,c]_{T_H\cup F}$ by \cite[Lemma~7.2]{McSul}.
This union is not disjoint:
An element $w$ of $[1,c]_{T_H\cup F}$ is also in $[1,c]_T$ if and only if a reduced word for $w$ in the alphabet $T_H\cup F$ either has no factored translations or $m$ factored translations.
When $m=1$ (i.e.\ when $\Upsilon^c$ has only one component), $[1,c]_{T\cup F}=[1,c]_T$ and there are no cover relations in $[1,c]_{T\cup F}$ labeled by elements of $F$.
In any case, the set $\C_c^+$ is the disjoint union of $\C_c$ and $\C_c^F$, but when ${m=1}$, the maximal chains of $[1,c]_{T\cup F}$ are in bijection with $\C_c$ (as a consequence of Corollary~\ref{chain sys nc} and Proposition~\ref{chain sys poset}.\ref{chains}), not with the larger set $\C_c^+$.

\begin{proposition}\label{Ccplus bin}
Let $W$ be a Coxeter group of affine type and let $c$ be a Coxeter element of~$W$ such that $\Upsilon^c$ has more than one component.
The set $\C_c^+$ is a binary chain system.
The associated binary compatibility relation is the union of the binary compatibility relations for~$\C_c$ and~$\C_c^F$.
If is also the set of two-letter words in the alphabet $T_V\cup T_H\cup F$ whose products are nontrivial elements of $[1,c]_{T\cup F}$.
\end{proposition}
\begin{proof}
For any maximal chain $x_0\covered\cdots\covered x_k$ in $[1,c]_{T\cup F}$ and any $i$ from $1$ to $k$, the element $x_i$ is the product of the labels on the edges $x_0\covered x_1$, $x_1\covered x_2$, \ldots, $x_{i-1}\covered x_i$.
Thus $[1,c]_{T\cup F}$ satisfies property \ref{bottom unique} in Proposition~\ref{poset to chain sys}.
Similarly, $x_i^{-1}c$ is the product of the labels on the edges $x_i\covered x_{i+1}$, $x_{i+1}\covered x_{i+2}$, \ldots, $x_{k-1}\covered x_k$, so that $[1,c]_{T\cup F}$ satisfies property \ref{top unique} as well.
Thus the set $\C([1,c]_{T\cup F})$ of words obtained by reading edge-labels of maximal chains from bottom to top is a chain system by Proposition~\ref{poset to chain sys}.
Because $m>1$, no element of $T\cup F$ is a product of other elements of $T\cup F$, and thus reduced words in the alphabet $T\cup F$ are in bijection with $\C([1,c]_{T\cup F})$.
In other words, $\C([1,c]_{T\cup F})$ coincides with~$\C_c^+$.

Corollary~\ref{chain sys nc} and Proposition~\ref{Cc alph} combine to say that $\C_c$ is a binary chain system with binary compatibility relation consisting of the two-letter words in the alphabet $T_V\cup T_H$ whose products are nontrivial elements of $[1,c]_T$.
Proposition~\ref{factorable bin} says that $\C_c^F$ is a binary chain system with binary compatibility relation consisting of the two-letter words in the alphabet $T_H\cup F$ whose products are nontrivial elements of $[1,c]_{T_H\cup F}$.
Let $B$ be the union of these two binary compatibility relations.
By Proposition~\ref{Ccplus disj}, $B$ is the set of two-letter words in the alphabet $T_V\cup T_H\cup F$ whose products are nontrivial elements of $[1,c]_{T\cup F}$.
We will show that $\C_c^+$ is the set of words $t_1,\ldots t_k$ in the alphabet $T_V\cup T_H\cup F$ that are maximal with respect to the property that $t_it_j\in B$ for all $i$ and $j$ with $1\le i<j\le k$.
We already know, by Proposition~\ref{Ccplus alph}, that the alphabet for $\C_c^+$ is $T_V\cup T_H\cup F$.

Suppose $t_1\cdots t_k\in\C_c^+$.
If $t_1,\ldots t_k$ has a letter in $F$, then Proposition~\ref{Ccplus disj} says that all letters of $t_1,\ldots t_k$ are in $T_H\cup F$.
Then $t_1\cdots t_k\in\C_c^F$, and Proposition~\ref{Ccplus disj} implies that if any letter is inserted into $t_1,\ldots t_k$ and the result has the property that all two-letter subwords are in $B$, that letter is in $T_H\cup F$.
Since $\C_c^F$ is a binary chain system, we conclude that $t_1\cdots t_k$ is maximal with respect to the property that all of its two-letter subwords are in $B$.
On the other hand, suppose $t_1,\ldots t_k$ has no letter in $F$.
Every element of $T_H$ is a reflection orthogonal to the direction of the Coxeter axis.
Thus if all of the letters $t_1,\ldots t_k$ are in $T_H$, then $t_1,\ldots t_k$ fixes the Coxeter axis pointwise and thus can't equal $c$, because $c$ acts on the Coxeter axis as a translation.
We conclude that $t_1,\ldots t_k$ has at least one letter in $T_V$.
Thus by Proposition~\ref{Ccplus disj}, if any letter is inserted into $t_1,\ldots t_k$ and the result has the property that all two-letter subwords are in $B$, that letter is in $T_V\cup T_H$.
Since~$\C_c$ is a binary chain system, we again see that $t_1\cdots t_k$ is maximal with respect to the property that all of its two-letter subwords are in $B$.

Conversely, suppose $t_1\cdots t_k$ is maximal with respect to the property that all of its two-letter subwords are in $B$.
If $t_1,\ldots t_k$ has no letter in $F$, then the letters of $t_1\cdots t_k$ are in $T_V\cup T_H$ and it is maximal with respect to the property that all of its two-letter subwords are in $\set{tt':1<_Ttt'\le_Tc}$.
Thus $t_1,\ldots t_k$ is in $\C_c$ and therefore in $\C_c^+$.
If $t_1,\ldots t_k$ has a letter in $F$, then Proposition~\ref{Ccplus disj} says that the letters of $t_1,\ldots t_k$ are all in $T_H\cup F$.
Furthermore, $t_1,\ldots t_k$ is maximal with respect to the property that all of its two-letter subwords are are nontrivial elements of $[1,c]_{T_H\cup F}$.
Thus $t_1,\ldots t_k$ is in $\C_c^F$ and therefore in $\C_c^+$.
\end{proof}

\begin{proposition}\label{ttprime}
Let $W$ be a Coxeter group of affine type and let $c$ be a Coxeter element of~$W$ such that $\Upsilon^c$ has more than one component.
If $t,t'\in T_H$ and $tt'$ is a prefix of a word in $\C_c^F$, then $tt'$ is a prefix of a word in $\C_c$.
\end{proposition}
\begin{proof}
Suppose $tt'$ is a prefix of a word $t_1\cdots t_k$ in $\C_c^F$.
Proposition~\ref{if a fact}.\ref{k trans} and Proposition~\ref{if a fact}.\ref{all factors} imply that the postfix $t_3\cdots t_k$ contains a factorization of a translation $w\in[1,c]_T$.
Using the braid action on $t_1\cdots t_k$, we make a word that has $tt'$ as a prefix and has the factors of $w$ adjacent.
Replacing the factors of $w$ with a reduced $T$-word for $w$, we have constructed a reduced $T$-word for $c$ having $tt'$ as a prefix.
\end{proof}

\begin{remark}\label{ttprime remark}  
The converse of Proposition~\ref{ttprime} is not necessary to our proof of Theorem~\ref{Ccplus chain sys}.
However, the converse is true and is proved as Corollary~\ref{ttprime iff} using the connection to para-exceptional sequences.
\end{remark}

Propositions~\ref{Ccplus alph}--\ref{ttprime} accomplish much of the proof of Theorem~\ref{Ccplus chain sys}.
Specifically, Proposition~\ref{Ccplus alph} determines the alphabet of $\C_c^+$ and Proposition~\ref{Ccplus bin} says that $\C_c^+$ is a binary chain system whose associated binary compatibility relation is the set of two-letter prefixes of words in $\C_c^+$.
Corollary~\ref{chain sys nc} and Proposition~\ref{Cc alph} say that the two-letter prefixes of words in $\C_c$ are $\set{tt':t,t'\in(T_V\cup T_H),1<_Ttt'\le_Tc}$.
Proposition~\ref{ttprime} says that the two-letter prefixes, consisting of two reflections, of words in $\C_c^F$ are also contained in $\set{tt':t,t'\in(T_V\cup T_H),1<_Ttt'\le_Tc}$.
Thus, in light of Proposition~\ref{Ccplus disj}, it remains to show that there is a bijection $\beta\mapsto f_\beta$ from~$\Xi^c$ to~$F$ such that the two-letter prefixes of words in~$\C_c^+$ that contain a letter in $F$ and a letter in $T_H$ or $2$ letters in $F$ are as described in Theorem~\ref{Ccplus chain sys}.
Proposition~\ref{factorable facts} describes the case where the two letters are associated to different components of~$\Upsilon^c$ or when two factored translations are associated to the same component.
Thus we can complete the proof of Theorem~\ref{Ccplus chain sys} by proving the following proposition.

\begin{proposition}\label{good bij}
Suppose $\Upsilon^c$ has more than one component.
There exists a bijection $\beta\mapsto f_\beta$ from $\Xi^c$ to $F$ such that
\begin{enumerate}[label=\rm(\roman*), ref=(\roman*)]
\item \label{right component}
For all $\beta\in\Xi^c$, the factored translation $f_\beta$ is associated (in the sense of Proposition~\ref{if a fact}.\ref{reorder}) to the component of $\Upsilon^c$ that contains $\beta$, and 
\item \label{ft tf}
For all $\beta,\gamma\in\Xi^c$ in the same component of $\Upsilon^c$ and all positive integers $k$ such that $\beta\not\in\set{c(\beta),c^2(\beta),\ldots,c^k(\beta)}$,
\begin{enumerate}[label=\rm(\alph*), ref=(\alph*)]
\item \label{ft}
$f_\gamma t_{\beta,k}$ is a prefix of $\C_c^+$ if and only if $\gamma\not\in\{c(\beta),c^2(\beta),\ldots,c^k(\beta)\}$, and 
\item \label{tf}
$t_{\beta,k}f_\gamma $ is a prefix of $\C_c^+$ if and only if $\gamma\not\in \{\beta,c(\beta),\ldots,c^{k-1}(\beta)\}$.
\end{enumerate}
\end{enumerate}
\end{proposition}
The proof will occupy the next section.

\section{Type-by-type arguments using the combinatorial models}\label{type sec}
This section is the proof of Proposition~\ref{good bij}.
By rescaling roots, we can assume that $\Phi$ is a standard affine root system.
That is, $\Phi$ is of type $A^{(1)}$ through $G^{(1)}$ in Kac's notation \cite[Table Aff~1]{Kac}, but referred to here as $\widetilde A$ through $\widetilde G$.
We treat the classical types using the combinatorial models and then describe the computations that settle the exceptional types.
We give only the details we need from each combinatorial model, referring the reader to \cite{affncA,affncD} for details.
In particular, each combinatorial model suggests a particular choice of $q_1,q_2$ or $q_1,q_2,q_3$, and we assume that choice without specifying it.
(Recall that the labeled poset $[1,c]_{T\cup F}$ does not depend on the choice of the $q_i$.)

Throughout the proof, we assume that $\Upsilon^c$ has more than one component.
That means that we can ignore types $\widetilde{C}$ and $\widetilde{G}$ and rule out some choices of $c$ in type $\widetilde{A}$.

The horizontal reflections and factored translations for one component of $\Upsilon^c$ commute with the horizontal reflections and factored translation for a different component.
Thus, we will verify Property~\ref{right component} of Proposition~\ref{good bij} in every case by checking that $f_\beta$ fails to commute with $t_\beta$.
For Property~\ref{ft tf}, in each case we will determine whether the relevant two-letter word is a prefix in $\C_c^+$ by determining whether the product of the two letters is in $[1,c]_{T_H\cup F}$.

\smallskip
\noindent
\textbf{Type $\widetilde A_{n-1}$.}
The group generated by $T\cup F$ is the group of permutations $\pi$ of $\integers$ with the property that $\pi(i+n)=\pi(i)+n$ for all~$i$.
We represent elements of the group in cycle notation.
Finite cycles come as part of infinite products.
We write $(a_1\,\,\,a_2\,\cdots\,a_i)_n$ for the product of all cycles of the form $(a_1+kn\,\,\,a_2+kn\,\cdots\,a_i+kn)$ for integers~$k$.
Because $\pi(i+n)=\pi(i)+n$, an infinite cycle can be described completely by listing finitely many of its elements.
The choice of a Coxeter element is equivalent to designating each integer from $1$ to $n$ as an inner point or an outer point, with at least one outer and one inner.
In fact, the assumption that $\Upsilon^c$ has more than one component is means that there are at least two inner points and at least two outer points.  
The numbers $1$ to $n$ are then placed on the boundary of an annulus, with inner points on the inner circle and outer points on the outer circle.

A \newword{noncrossing partition of the annulus} is a set partition of $\set{1,\ldots,n}$ together with the choice, for each block of the set partition, of an \newword{embedded block}:
a disk or annulus containing the numbered points in the block and no other numbered points.
The embedded blocks must be disjoint and there are some technical conditions.  
See \cite[Section~3.4]{affncA}.
Figure~\ref{nc ann ex} shows some examples with $n=8$.
\begin{figure}
\scalebox{0.85}{\begin{tabular}{ccccc}
\includegraphics{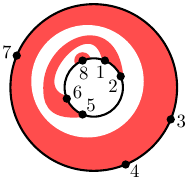}
&&\quad\includegraphics{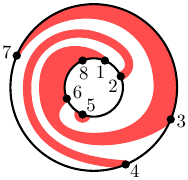}
&&\quad\includegraphics{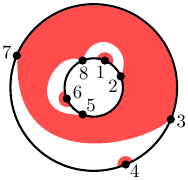}
\end{tabular}}
\caption{Some noncrossing partitions of an annulus}
\label{nc ann ex}
\end{figure}

The interval $[1,c]_{T\cup F}$ is isomorphic to the natural partial order on noncrossing partitions of the annulus \cite[Theorem~3.25]{affncA}.
The isomorphism involves reading the cycle notation of an affine permutation from the boundaries of blocks in the noncrossing partition, adding or subtracting $n$ when the boundary crosses the ``date line'', the radial line segment between $n$ and $1$. For example, numbering the noncrossing partitions in Figure~\ref{nc ann ex} as $\P_1$, $\P_2$, $\P_3$ from left to right, 
\begin{align*}
\perm(\P_1)&=(1\,\,\,-2\,\,\,-3\,\,\,2)_8(\cdots3\,\,\,4\,\,\,7\,\,\,11\cdots)(8)_8\\
\perm(\P_2)&=(1\,\,\,0\,\,\,-2\,\,\,-9\,\,\,-5)_8(2\,\,\,-4)_8(5)_8\\
\perm(\P_3)&=(1)_8(\cdots2\,\,\,0\,\,\,-3\,\,\,-6\cdots)(\cdots3\,\,\,7\,\,\,11\cdots)(4)_8(6)_8
\end{align*}
For our purposes, the following are the key facts about the combinatorial model:

\smallskip
\noindent
\textbf{Fact.}
The Coxeter element, represented by the noncrossing partition with only one embedded block, namely an annulus, is the product of two infinite cycles 
\[(\cdots\,a_1\,\,\,a_2\,\cdots\,a_k\,\,\,a_1+n\,\cdots)(\cdots\,b_1\,\,\,b_2\,\cdots\,b_{n-k}\,\,\,b_1-n\,\cdots)\]
where $a_1,\ldots, a_k$ are the outer points in increasing order and $b_1,\ldots,b_{n-k}$ are the inner points in decreasing order \cite[Lemma~3.6]{affncA}.
This description of $c$ implies a description of $c$ as a linear transformation on $V$, as explained in \cite[Section~3.1]{affncA}.
Thus, in what follows, $c$ is described as acting on integers and as acting on roots.

\smallskip
\noindent
\textbf{Fact.}
The reflections for roots in $\Xi^c$ are $(a\,\,\,c(a))_n$ for $a=1,\ldots,n$.
If $a$ is outer, then $c(a)$ is the next larger outer point or, if $a$ is the largest outer point, $c(a)$ is the smallest outer point plus $n$.
If $a$ is inner, then $c(a)$ is the next smaller inner point or, if $a$ is the smallest inner point, $c(a)$ is the largest inner point minus $n$.

\smallskip
\noindent
\textbf{Fact.}
The horizontal reflections in $[1,c]_{T\cup F}$ are $(a\,\,\, b)_n$ such that $a$ and $b$ are both congruent mod $n$ to inner points or both congruent mod $n$ to outer points and $|a-b|<n$ \cite[Proposition~5.11.3]{affncA}.
The two possibilities, outer or inner, determine the two components of $\Upsilon^c$.

\smallskip
\noindent
\textbf{Fact.}
A \newword{loop} $\ell_a$ is the infinite cycle $(\cdots\,a\,\,\,a+n\,\cdots)$.
The factored translations are the loops $\ell_a$ for $a$ outer and the inverses $\ell_a^{-1}$ for $a$ inner \mbox{\cite[Theorem~5.8.1]{affncA}}.
A factored translation and a horizontal reflection are associated to the same component of $\Upsilon^c$ if they both involve inner points or both involve outer points.
\smallskip

We define the bijection $\beta\mapsto f_\beta$ from $\Xi^c$ to $F$.
Given $\beta\in\Xi^c$ with $t_\beta=(a\,\,\,c(a))_n$, let $f_\beta$ be $\ell_a$ if $a$ is outer or $\ell_a^{-1}$ if $a$ is inner.
Property~\ref{right component} holds because $f_\beta t_\beta\neq t_\beta f_\beta$.

Suppose $\beta,\gamma\in\Xi^c$ are in the same component of $\Upsilon^c$ and choose $k<r_\beta$.
If $t_\beta=(a\,\,\,c(a))_n$, then $t_{\beta,k}=(a\,\,\,\,c^k(a))_n$ and $f_\beta=(\cdots a\,\,\,\,a+\ep n\cdots)$ and $f_{c^k(\beta)}=(\cdots c^k(a)\,\,\,\,c^k(a)+\ep n\cdots)$, where $\ep$ is $+1$ if $a$ is outer or $-1$ if $a$ is inner.
Then $f_\beta t_{\beta,k}$ and $t_{\beta,k}f_{c^k(\beta)}$ both equal the infinite cycle $(\cdots a\,\,\,\,c^k(a)\,\,\,\,a+\ep n\cdots)$.
This is an increasing sequence integers congruent mod $n$ to outer points if $a$ is outer or a decreasing sequence of integers congruent to inner points if $a$ is inner.
There is a corresponding noncrossing partition with one nontrivial block, an annular block containing only the numbered points $a$ and $c(a)$ (both outer or both inner).  
See the left picture of Figure~\ref{proof ncs}.
\begin{figure}
\scalebox{0.85}{\begin{tabular}{ccccc}
\includegraphics{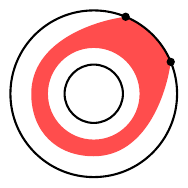}
\begin{picture}(0,0)(45,-45)
\put(18,37){$a$}
\put(35,18){$c(a)$}
\end{picture}
&&\quad\includegraphics{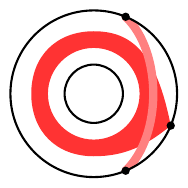}
\begin{picture}(0,0)(45,-45)
\put(18,37){$a$}
\put(34,-23){$c^i(a)$}
\put(18,-41){$c^k(a)$}
\end{picture}
&&\qquad\includegraphics{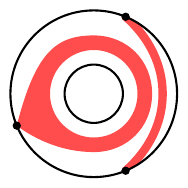}
\begin{picture}(0,0)(45,-45)
\put(18,37){$a$}
\put(-64,-13){$c^i(a)$}
\put(18,-41){$c^k(a)$}
\end{picture}
\end{tabular}}
\caption{Illustrations of the proof of Proposition~\ref{good bij}}
\label{proof ncs}
\end{figure}
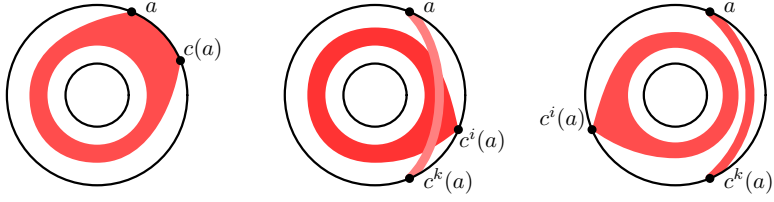
(This picture and later pictures illustrate the case where $a$ is outer.
The pictures omit many numbered points, each of which is a trivial block in the noncrossing partition.)
Thus $f_\beta t_{\beta,k}=t_{\beta,k}f_{c^k(\beta)}\in[1,c]_{T\cup F}$.

Observe that $t_{\beta,k}f_\beta=f_{c^k(\beta)}t_{\beta,k}=(\cdots a\,\,\,\,c^k(a)+\ep n\,\,\,\,a+\ep n\cdots)$.
If $a$ is outer, then $a<c^k(a)<a+n$, so $a<c^k(a)+n>a+n$.
If $a$ is inner, then ${a>c^k(a)>a-n}$, so $a>c^k(a)-n<a+n$.
Any infinite cycle that is read from a noncrossing partition of the annulus is monotone, so $t_{\beta,k}f_\beta=f_{c^k(\beta)}t_{\beta,k}\not\in[1,c]_{T\cup F}$.
We have proved Property~\ref{ft tf} in the case where $\gamma=\beta$ or $\gamma=c^k(\beta)$.

Now suppose $\gamma=c^i(\beta)$ for some $i=1,\ldots,k-1$.
As before, $t_{\beta,k}=(a\,\,\,\,c^k(a))_n$, and now $f_\gamma=(\cdots c^i(a)\,\,\,\,c^i(a)+\ep n\cdots)$ for $\ep$ as before.
The cycle $f_\gamma$ is read from an annular block containing only $c^i(a)$, while $t_{\beta,k}$ is read from a disk containing only $a$ and $c^k(a)$.
Since $a$, $c^i(a)$ and $c^k(a)$ are all outer or all inner, these two blocks intersect.
See the center picture of Figure~\ref{proof ncs}.
Thus there is no noncrossing partition containing both, and we see that $f_\gamma t_{\beta,k}=t_{\beta,k}f_\gamma\not\in[1,c]_{T\cup F}$.
We have proved Property~\ref{ft tf} in the case where $\gamma=c^i(\beta)$ for some $i=1,\ldots,k-1$.

Suppose there exists $i$ such that $\gamma=c^i(\beta)$, but no such $i$ in $\set{0,1,\ldots,k}$.
The embedded blocks for $t_{\beta,k}$ and $f_\gamma$ are as in the previous paragraph, but in this case, the two blocks do not intersect.
See the right picture of Figure~\ref{proof ncs}.
Thus there is a noncrossing partition corresponding to $f_\gamma t_{\beta,k}=t_{\beta,k}f_\gamma$, so $f_\gamma t_{\beta,k}=t_{\beta,k}f_\gamma\in[1,c]_{T\cup F}$.
We have proved Property~\ref{ft tf}.

\smallskip
\noindent
\textbf{Type $\widetilde C_{n-1}$.}
This case is ruled out by the assumption that $\Upsilon^c$ has more than one component.
However, it is useful to describe the Coxeter group of type $\widetilde C_n$ combinatorially.
An \newword{affine signed permutation} is a permutation $\pi$ of $\integers$ with $\pi(i+2n)=\pi(i)+2n$ and $\pi(-i)=-\pi(i)$ for all $i\in\integers$.
An affine signed permutation fixes each multiple of $n$, and we ignore these singleton cycles.
We use similar notation to type $\widetilde A$, with $2n$ replacing $n$, so that $(a_1\,\,\,a_2\,\cdots\,a_k)_{2n}$ is an infinite product of cycles.
Because $\pi(-i)=-\pi(i)$, we write $(\!(a_1\,\cdots\,a_k)\!)_{2n}$ for $(a_1\,\cdots\,a_k)_{2n}\cdot(-a_1\,\cdots\,-a_k)_{2n}$ when $(a_1\,\cdots\,a_k)_{2n}\neq(-a_1\,\cdots\,-a_k)_{2n}$.
Similarly, $(\!(\cdots\,a_1\,\,\,a_2\,\cdots\,a_\ell\,\,\,a_1+2n\,\cdots)\!)$ means $(\cdots\,a_1\,\,\,a_2\,\cdots\,a_\ell\,\,\,a_1+2n\,\cdots)(\cdots\,-\!a_1\,\,-\!a_2\,\cdots\,-a_\ell\,\,-\!a_1-2n\,\cdots)$.

\smallskip
\noindent
\textbf{Type $\widetilde B_{n-1}$.}
The Coxeter group of type $\widetilde B_{n-1}$ is the group of affine singly even-signed permutations (see \cite[Section~7.1]{affncD}, but we will not need that definition here.
More important for our purposes, the group generated by $T\cup F$ is the group of affine signed permutations defined above (the Coxeter group of type $\widetilde C_{n-1}$).

All Coxeter elements of $\widetilde B_{n-1}$ are conjugate to each other.
(Coxeter groups of type $\widetilde A$ are unique among affine Coxeter groups in not having this property.)
Conjugation of $c$ by some element $w$ induces an isomorphism of labeled posets from $[1,c]_T$ to $[1,wcw^{-1}]_T$ and corresponds to acting on the Coxeter axis and translation vectors by the action of $w$ on the dual space $V^*$.
Conjugation also corresponds to acting on roots by the action of $w$ on $V$, sending $\Upsilon^c$ to $\Upsilon^{wcw^{-1}}$, and inducing an isomorphism of labeled posets from $[1,c]_{T\cup F}$ to $[1,wcw^{-1}]_{T\cup F}$, preserving which labels that are reflections and which are factored translations.
Therefore, it is enough to prove Proposition~\ref{good bij} for any one choice of Coxeter element, and we introduce the combinatorial model for a specific choice of~$c$.

We model $[1,c]_{T\cup F}$ in type $\widetilde B_{n-1}$ by symmetric noncrossing partitions of an annulus with one double point. See \cite[Section~7]{affncD}.
We need very few details of this model, not even a definition of double points.
The symmetry exchanges the outer and inner boundary of the annulus and has exactly two fixed points, one of which is the double point.
We will see that the two components of $[1,c]_{T_H\cup F}$ are a component that coincides with one of the components in type $\widetilde A$ and a component that has only two factored translations and only two horizontal reflections.

Our choice of a Coxeter element means that the outer points are $1,2,\ldots,n-2$ and the inner points are $-1,-2,\ldots,{-n+2}$.
There is also a double point $\pm(n-1)$.
The horizontal reflections and factored translations for this choice of $c$ are given in \mbox{\cite[Section~7.4]{affncD}}.
The factorization $c=c_1c_2$ as in Proposition~\ref{if a fact}.\ref{fact c} is 
\[c_1=(\!(\cdots\,1\,\,\,2\,\cdots\,{n-2}\,\,\,\,\,{1+2n}\,\cdots)\!),\qquad c_2=({n-1}\,\,\,{n+1})_{2n}.\]

For one component of $\Upsilon^c$, the horizontal reflections are $(\!(i\,\,\,\,j)\!)_{2n}$ and $(\!(i\,\,\,\,{j-2n})\!)_{2n}$ for $1\le i<j\le n-2$ and the factored translations are ${(\!(\cdots\,a\,\,\,a+2n\,\cdots)\!)}$ for $a\in\set{1,\ldots,n-2}$.
The corresponding component $[1,c_1]_{T_H\cup F}$ of $[1,c]_{T_H\cup F}$ is modeled by symmetric noncrossing partitions of the annulus whose only nontrivial blocks are a collection of embedded blocks that involve only outer points and their images under the symmetry (a collection of embedded blocks that involve only inner points).
These noncrossing partitions have the same combinatorics as the noncrossing partitions in the $\widetilde A$ case that only involve outer points, taking $2n$ in place of $n$ and taking $1,\ldots,n-2$ to be the outer points.
The simple roots $\beta$ in this component correspond to reflections $t_\beta=(\!(a\,\,\,c(a))\!)_{2n}$ for $a\in\set{1,2,\ldots,n-2}$, with $c(n+2)=2n+1$ and otherwise $c(a)=a+1$.
Taking $f_\beta$ to be ${(\!(\cdots\,a\,\,\,a+2n\,\cdots)\!)}$, again $f_\beta$ and $t_\beta$ do not commute, so Property~\ref{right component} holds.
The proof of Property~\ref{ft tf} also proceeds exactly as in the $\widetilde A$ case.

In the other component, the simple roots can be labeled $\beta$ and $\gamma$ and the factored translations can be labeled $f_\beta$ and $f_\gamma$ so that
\[\begin{array}{ll}
t_\beta=(\!({-n+1}\,\,\,\,{n-1})\!)_{2n}&\qquad f_\beta=(\!(\cdots\,n-1\,\,\,\,\,-n-1\,\cdots)\!)\\
t_\gamma=(\!({-n-1}\,\,\,\,{n+1})\!)_{2n}&\qquad f_\gamma=(\!(\cdots\,-n-1\,\,\,\,\,\,n-1\,\cdots)\!).
\end{array}\]
We see, again, that Property~\ref{right component} holds.
The reflections in $T_H$ associated to positive roots in this component of $\Upsilon^c$ are $t_{\beta,1}=t_\beta$ and ${t_{\gamma,1}=t_\gamma}$.
One can check that $f_\beta t_\beta=t_\beta f_\gamma=f_\gamma t_\gamma=t_\gamma f_\beta=c_2$. 
It follows that none of the products $t_\beta f_\beta$, $f_\gamma t_\beta$, $t_\gamma f_\gamma$, or $f_\beta t_\gamma$  equal $c_2$.
We have verified Property~\ref{ft tf} in this component as well.

\smallskip
\noindent
\textbf{Type $\widetilde D_{n-1}$.}
The Coxeter group of type $\widetilde D_{n-1}$ is the group of affine doubly even-signed permutations (see \cite[Section~2.3]{affncD}), but again, we will not need that definition.
The larger group generated by $T\cup F$ is the group of affine barred even-signed permutations.
We will also not need the full details of that definition, which are found in \cite[Section~2.4]{affncD} and \cite[Section~6.3]{affncD}.
Instead, we can deal with a larger group of permutations that we will call affine barred signed permutations.
These are permutations of a set $\integers\cup\set{\bar\imath:i\in\integers}$ consisting of the integers together with a disjoint copy of the integers marked with bars.
We define an upper-barring operator on this set:
It sends an integer $i$ to $\bar\imath$ and sends $\bar\imath$ to $i+2n$.
There are natural notions of negation of barred integers and adding $2n$ to barred indices.
Specifically, $-\bar\imath=\overline{-i-2n}$ and $\bar\imath+2n=\overline{i+2n}$.
(These notions are natural in a linear-algebraic context as explained in \cite[Section~6.3]{affncD}.)
The group of \newword{affine barred signed permutations} is the group of permutations $\pi$ of $\integers\cup\set{\bar\imath:i\in\integers}$ with $\pi(i+2n)=\pi(i)+2n$, $\pi(-i)=-\pi(i)$, and $\pi(\bar\imath)=\overline{\pi(i)}$ for all $i\in\integers$.
Double parentheses $(\!(\cdots)\!)$ in cycle notation have an expanded meaning in this group.
They stand for a cycle and its negative, as well as the cycles obtained from the cycle or its negative by applying the upper-barring operator to all entries.

The relevant combinatorial model for $[1,c]_{T\cup F}$ in type $\widetilde D_{n-1}$ is the lattice of symmetric noncrossing partitions of an annulus with two double points.
(In fact, there are finitely many elements of $[1,c]_{T\cup F}$ that are not part of the model, but the model is still sufficient for our purposes.
See \cite[Sections~5--6]{affncD} for details.)
The symmetry is the same as in type $\widetilde B_{n-1}$, but both fixed points are double points.
Again, we make a special choice of $c$, which corresponds to declaring the numbers $2,\ldots,n-2$ to be outer points and $-2,\ldots,{-n+2}$ to be inner points.
The points $\pm1$ and $\pm(n-1)$ are double points.
There are three components of $[1,c]_{T_H\cup F}$.
The factorization ${c=c_1c_2c_3}$ is
\[c_1=(\!(\cdots\,2\,\,\,3\,\cdots\,n-2\,\,\,2n+2\,\cdots)\!),\quad c_2=(\!(1\,\,\,\overline{-n-1})\!)_{2n},\quad c_3=(\!(1\,\,\,\overline{-n+1})\!)_{2n}.\]

For one component of $\Upsilon^c$, the horizontal reflections are $(\!(i\,\,\,\,j)\!)_{2n}$ and $(\!(i\,\,\,\,{j-2n})\!)_{2n}$ for $2\le i<j\le n-2$ and the factored translations are ${(\!(\cdots\,a\,\,\,a+2n\,\cdots)\!)}$ for $a\in\set{2,\ldots,n-2}$.
The corresponding component $[1,c_1]_{T_H\cup F}$ of $[1,c]_{T_H\cup F}$ is modeled by symmetric noncrossing partitions of the annulus whose only nontrivial blocks involve only outer points, or symmetrically only inner points.
These have the same combinatorics as the analogous noncrossing partitions in types $\widetilde A$ and $\widetilde B$.
Properties~\ref{right component} and~\ref{ft tf} are proved for this component exactly as in those cases.

{\allowdisplaybreaks
In the second component, the simple roots can be labeled $\beta$ and $\gamma$ and the factored translations can be labeled $f_\beta$ and $f_\gamma$ so that
\begin{align*}
t_\beta&=(\!(1\,\,\,\,{n-1})\!)_{2n}\\
t_\gamma&=(\!(1\,\,\,\,{-n-1})\!)_{2n}\\
f_\beta&=(\!(\cdots\,1\,\,\,\,\bar{1}\,\,\,\,1+2n\,\cdots)\!)(\!(\cdots\,n-1\,\,\,\,\overline{-n-1}\,\,\,\,-n-1\,\cdots)\!)\\
 f_\gamma&=(\!(\cdots\,1\,\,\,\,\overline{1-2n}\,\,\,\,1-2n\,\cdots)\!)(\!(\cdots\,n-1\,\,\,\,\overline{n-1}\,\,\,\,3n-1\,\cdots)\!).
\end{align*}
As in type $\widetilde B$, Property~\ref{right component} holds, and again the only reflections in $T_H$ associated to positive roots in this component are $t_{\beta,1}=t_\beta$ and ${t_{\gamma,1}=t_\gamma}$.
We again check that $f_\beta t_\beta=t_\beta f_\gamma=f_\gamma t_\gamma=t_\gamma f_\beta=c_2$.
Therefore, none of $f_\gamma t_\beta$, $t_\beta f_\beta$, $f_\beta t_\gamma$, or $t_\gamma f_\gamma$ equal~$c_2$.
We have verified Property~\ref{ft tf} in this component.
}

{\allowdisplaybreaks
Properties~\ref{right component} and~\ref{ft tf} in the third component are verified in exactly the same way, but with $c_3$ replacing $c_2$ and with 
\begin{align*}
t_{\beta'}&=(\!(1\,\,\,\,{n+1})\!)_{2n}\\
t_{\gamma'}&=(\!(1\,\,\,\,{-n+1})\!)_{2n}\\
f_{\beta'}&=(\!(\cdots\,1\,\,\,\,\bar{1}\,\,\,\,1+2n\,\cdots)\!)(\!(\cdots\,n+1\,\,\,\,\overline{-n+1}\,\,\,\,-n+1\,\cdots)\!)\\
 f_{\gamma'}&=(\!(\cdots\,1\,\,\,\,\overline{1-2n}\,\,\,\,1-2n\,\cdots)\!)(\!(\cdots\,-n+1\,\,\,\,\overline{-n+1}\,\,\,\,n+1\,\cdots)\!).
\end{align*}
}

\smallskip
\noindent
\textbf{The exceptional affine types.}
In the exceptional types $\widetilde E_6$, $\widetilde E_7$, $\widetilde E_8$, and $\widetilde F_4$, Properties~\ref{right component} and~\ref{ft tf} are verified computationally, as we now explain.
(Type $\widetilde G_2$ is ruled out by the assumption that $\Upsilon^c$ has more than one component.)

First, we find the set $T_H$ of horizontal reflections in $[1,c]_T$.
The direction of the Coxeter axis is found as explained in \cite[Section~5.1]{affncA}, allowing us to determine which roots are horizontal (orthogonal to the Coxeter axis).
Then $T_H$ is the set of reflections $\beta$ and $\delta-\beta$ for each horizontal root $\beta$ that is in $\Phi_\fin$.

Next, we find the translations in $[1,c]_T$.
Recall that a translation $w\in[1,c]_T$ has $\ell_T(w)=2$.
Starting with a reduced $T$-word with a prefix $tt'=w$, we replace $tt'$ by the factors of $w$ to make a word satisfying the hypotheses of Proposition~\ref{if a fact}.
Therefore, the other letters of the word are horizontal reflections and constitute a reduced $T$-word for $w^{-1}c$.
We factor this word into a factor for each component of~$\Upsilon^c$.
Each factor is a Coxeter element for a Coxeter group of finite type A, with respect to some simple system.
Thus, possibly changing the simple system and applying the braid action, we can find a reduced $T$-word for $w^{-1}c$ such that each consecutive pair of entries in each component is non-orthogonal (in the sense of $K$) but every other pair is orthogonal.  
To find the translations in $[1,c]_T$, we find all sequences of reflections in $T_H$ that have that property.
Computing the product of the sequence to be $v$, we have $w=cv^{-1}$.
We can test whether $w$ is a translation of length $\ell_T(w)=2$ because each such translation is a product of parallel reflections and thus its translation vector is a multiple of a co-root in $\Phi_\fin\ck$.

Once we have all translations in $[1,c]_T$, we factor them as described in Section~\ref{McSul chain sec}, keeping track of which component of $\Upsilon^c$ each factored translation is associated with.
Furthermore, from any one translation $w$ in $[1,c]_T$ and reduced expression for $w^{-1}c$ as a product of horizontal reflections, we factor $w$ to obtain a reduced word for $c$ in the alphabet $T_H\cup F$ and commute its letters to find $c_1$, $c_2$, and (possibly)~$c_3$.

For each component and each factored translation $f$ associated to the component, we find all words for $f^{-1}c_i$ of the appropriate length, in the alphabet of horizontal reflections for that component.
If $t$ is a horizontal reflection associated to that component, then $ft$ is a prefix of $\C_c^+$ if and only if $t$ appears in one of those words.
We also find all words for $c_if^{-1}$ and thus determine which words $tf$ are prefixes of~$\C_c^+$.
Thus in each exceptional type, we verify the proposition.


\section{Para-exceptional sequences}\label{sec:brick_sequences}  
We now turn our attention to representation-theoretic interpretations of the chain system $\C_c^+$ introduced in Section~\ref{McSul chain sec}.
These interpretations eventually lead, in Section~\ref{sec:rep_theory_model}, to a representation-theoretic proof of the fact that $\C_c^+$ is a binary chain system.

A \newword{tame hereditary algebra} is a finite-dimensional hereditary algebra whose Cartan matrix is positive semi-definite, but not positive definite.
(This is not the original definition of tameness, which will not be necessary for this paper.) 
Thus~$\Lambda$ is a connected tame hereditary algebra if and only if the Cartan matrix~$A$ is affine. 
From now on, $\Lambda$ is a connected tame hereditary algebra.  
We now recall additional information about the category $\mods\Lambda$ in the connected tame case.
We refer to \cite{DlabRingel,RingelInfinite}  as general references (warning, however, that the term ``homogeneous'', as used in Lemma~\ref{lem:homogeneous} below, has a different meaning in \cite{RingelInfinite}). 

\begin{remark} 
We now explain a subtlety that arises in applying the results of \cite{DlabRingel}, which describes the representation theory of Euclidean $\field$-species, to our study of tame hereditary algebras.
Tame hereditary algebras are classified up to Morita equivalence in \cite[Theorem~1]{DlabRingelTame}, and there are two possibilities in the connected case.
If $\Lambda$ is Morita equivalent to the tensor algebra of a Euclidean $\field$-species, then we can freely apply the description from \cite{DlabRingel} to $\Lambda$. 
(If the field $\field$ is perfect then we are automatically in this case.)
Otherwise, the $\field$-species determined by $\Lambda$ is a quiver~$Q$ of type $\widetilde{A}_{n-1}$, and although the categories $\mods \Lambda$ and $\mods \field Q$ differ in the classification of the homogeneous tubes (defined below), the exceptional modules and non-homogeneous tubes of $\mods \Lambda$ and $\mods \field Q$ can be identified \cite[Theorem~2]{DlabRingelTame}. 
Thus even in this case, the description of the exceptional modules and of the non-homogeneous tubes from \cite{DlabRingel} applies to $\Lambda$.
\end{remark}

Recall from Section~\ref{nc chain sec} that the bilinear form $E_{c^{-1}}$ coincides with the \newword{Euler form} on the Grothendieck group $K_0(\mods \Lambda)$.
That is, for modules $X$ and $Y$, 
\[E_{c^{-1}}(\undim X,\undim Y) = \dim_\field(\Hom(X,Y)) - \dim_\field(\Ext^1(X,Y)).\]
An indecomposable module $X$ is \newword{preprojective} if $E_{c^{-1}}(\delta,\undim X) < 0$, \newword{preinjective} if $E_{c^{-1}}(\delta,\undim X) > 0$, or \newword{regular} if $E_{c^{-1}}(\delta,\undim X) = 0$.  
Every indecomposable module that is either preprojective or preinjective is also exceptional.
An arbitrary module is called preprojective/preinjective/regular if all of its direct summands are preprojective/preinjective/regular. 
Let $\reg$ be the subcategory of regular modules. 
This is a non-exceptional wide subcategory of $\mods\Lambda$.
It is common to use the prefix quasi- when describing analogs of notions of $\mods\Lambda$ within the subcategory~$\reg$. For example, a \newword{quasi-simple} is a module which is simple in~$\reg$.

We will use the following well-known lemma, see e.g. \cite[Lemma~4.3]{IPT}. 
(The proof given in \cite{IPT} is valid for all connected tame hereditary algebras, even though they work only in the path algebra setting.)

\begin{lemma}\label{lem:hom_ext} 
Let $X,Y$ be indecomposable modules with $X$ non-exceptional. 
\begin{enumerate}[label=\bf\arabic*., ref=\arabic*]
\item If $Y$ is preprojective then $\Hom(Y,X)$ and $\Ext^1(X,Y)$ are nonzero, while $\Hom(X,Y)$ and $\Ext^1(Y,X)$ are zero. 
\item If $Y$ is preinjective then $\Hom(X,Y)$ and $\Ext^1(Y,X)$ are nonzero, while $\Hom(Y,X)$ and $\Ext^1(X,Y)$ are zero.
\end{enumerate}
\end{lemma}

The Auslander-Reiten translate $\tau$ permutes the quasi-simple modules. 
The extension closure of a $\tau$-orbit is called a \newword{tube}. 
Each tube is a non-exceptional wide subcategory of $\mods\Lambda$, and any indecomposable regular module lies in exactly one tube. 
The \newword{rank} of a tube is the number of quasi-simple modules it contains. 
For use later, we record the following well-known fact as a lemma. 

\begin{lemma}\label{lem:tubes_ortho} 
Let $X$ and $Y$ be indecomposable regular modules which lie in different tubes. Then $\Hom(X,Y) = 0 = \Ext^1(X,Y)$.
\end{lemma}

Tubes of rank 1 are called \newword{homogeneous} and tubes of larger rank are \newword{non-homogeneous}. 
There are infinitely many homogeneous tubes. 
A module $X$ is \newword{homogeneous} if $\tau X'=X'$ for every indecomposable direct summand $X'$ of~$X$. Every indecomposable homogeneous module lies in a homogeneous tube.
Let $\homo$ be the subcategory of homogeneous modules and let $\nonhomo$ be the subcategory of regular modules which contain no homogeneous direct summand.
The following fact is well known.

\begin{lemma}\label{lem:homogeneous} 
Let $\T$ be a homogeneous tube.  
Then $\T$ contains a unique brick~$X_\T$. Moreover, $X_\T$ is homogeneous (and thus non-exceptional) and has $\T = \Filt(X_\T)$.
\end{lemma}

Non-homogeneous tubes are in bijection with components of $\Upsilon^c$. In particular, there are at most three non-homogeneous tubes. 
We denote by  $m\in\set{0,1,2,3}$ the number of non-homogeneous tubes in $\Lambda$ and fix an indexing $\T_1,\ldots\T_m$ of the non-homogeneous tubes. We then
write $\N=\set{\T_1,\ldots\T_m}$ for the set of non-homogeneous tubes. Lemma~\ref{lem:tubes_ortho} then implies the following. 
\begin{proposition}\label{prop:tubes_ortho}
    There are orthogonal decompositions $\nonhomo = \T_1 \oplus \cdots\oplus \T_m$ and $\reg = \homo \oplus \nonhomo$. In particular, every wide subcategory $\W \subseteq \reg$ admits orthogonal decompositions $\W = (\homo \cap \W) \oplus (\nonhomo \cap \W)$ and $\W = (\homo \cap \W) \oplus (\T_1\cap \W) \oplus \cdots \oplus (\T_m \cap \W)$.
\end{proposition}
For $\beta \in \Xi^c$, let $\T_\beta$ be the tube containing $X_\beta$.
If also $\gamma\in\Xi^c$, then $\T_\gamma=\T_\beta$ if and only $\gamma \equiv_{\Upsilon^c} \beta$.
The quasi-simple modules in $\T_\beta$ are the $X_\gamma$ such that $\gamma \in \Xi^c$ with $\gamma \equiv_{\Upsilon^c} \beta$.
Thus the rank of $\T_\beta$ is $r_\beta$, the length of the $c$-cycle containing $\beta$.
For $k$ a positive integer, there exists a unique indecomposable module $R_{\beta,k} \in \T_\beta$ whose quasi-top is $X_\beta$ 
and whose quasi-length is $k$. 
If $k<r_\beta$, then $R_{\beta,k} = X_{\beta^{(k)}}$, for $\beta^{(k)}$ as in~\eqref{eqn:beta_k}. 
The set of all indecomposable modules in $\T_\beta$ is $\set{R_{\gamma,k} \mid \gamma \in \Xi^c, \gamma \equiv_{\Upsilon^c} \beta, k \in \mathbb{Z}_{\geq 0}}$.
The Auslander-Reiten translate $\tau$ acts on these indecomposables as $\tau R_{\gamma,k} = R_{c(\gamma),k}$. 
The following characterization of bricks and exceptional modules in the non-homogeneous tubes is well known. 
(See e.g. \cite[Corollary~X.2.7b]{Elements2}.)

\begin{proposition}\label{prop:bricks_in_tube} 
Let $\beta \in \Xi^c$ and let $k$ be a positive integer. 
\begin{enumerate}[label=\bf\arabic*., ref=\arabic*]
\item $R_{\beta,k}$ is a brick if and only if $k \leq r_\beta$.
\item $R_{\beta,k}$ is exceptional if and only if $k < r_\beta$.
\end{enumerate}
\end{proposition}

We will also need parts of the well known computation of Hom and Ext for regular modules, found, for example, in \cite[Corollary~X.2.7a]{Elements2}.
The parts that we need are the following propositions, which follow from \cite[Corollary~X.2.7a]{Elements2}, the Auslander-Reiten formulas, and the fact that $\tau R_{\gamma,k} = R_{c(\gamma),k}$.
For each $\beta\in\Xi^c$, write $F_\beta$ for~$R_{\beta,r_\beta}$.

\begin{proposition}\label{prop:hom_ext_vanish}  
Let $\beta, \gamma \in \Xi^c$ and let $k\in\integers$.
\begin{enumerate}[label=\bf\arabic*., ref=\arabic*]
\item \label{rgam before}
$\Hom(F_\gamma,R_{\beta,k}) = 0=\Ext^1(F_\gamma,R_{\beta,k})$ if and only if $\gamma \notin \set{c^i\beta:0\le i<k}$.
\item \label{rgam after}
$\Hom(R_{\beta,k},F_\gamma) = 0 = \Ext^1(R_{\beta,k}, F_\gamma)$ if and only if $\gamma \notin \set{c^i\beta:0<i\le k}$.
\end{enumerate}
\end{proposition}

\begin{definition} \label{para exc EF def}
Let $\nrig = \{F_\beta\mid \beta \in \Xi^c\}$. By Lemma~\ref{lem:homogeneous} and Proposition~\ref{prop:bricks_in_tube}, the set~$\nrig$ is precisely the set of non-homogeneous, non-exceptional bricks.
Define the set $\pexcep$ of \newword{para-exceptional} modules to be $\E\cup\nrig$. 
A \newword{para-exceptional sequence} is a $\pexcep$-brick sequence in the sense of Definition~\ref{def:exceptional_sequence}.
We write $\Bpx$ for the binary compatibility relation on the alphabet $\pexcep$ consisting of all two-term para-exceptional sequences 
and $\Cpx$ for the set $\C(\Bpx)$ of maximal para-exceptional sequences. 
\end{definition}

\begin{remark}\label{rem:prufer_adic}
The properties in Proposition~\ref{prop:hom_ext_vanish} also hold if one replaces each~$F_\gamma$ with the infinite-dimensional adic module defined by taking the inverse limit of the regular modules whose quasi-top is $R_\gamma$. By adjusting indices, one could also obtain similar statements for the Pr\"ufer modules (defined as direct limits). Thus, throughout this paper, one could use either the adic or Pr\"ufer modules from the non-homogeneous tubes in place of the set $\mathcal{F}$.
\end{remark}

\begin{remark}\label{rem:non-homogeneous}  
Note that $(\excep \cup \nrig)$ is precisely the set of non-homogeneous bricks. 
One might call para-exceptional 
sequences \newword{non-homogeneous brick sequences}.  
\end{remark}

\begin{remark}\label{rem:schur}
The bijection $X \mapsto t_{\undim X}$ allows us to encode exceptional sequences as sequences of (real) roots. 
Derksen and Weyman~\cite[Section~4]{DW_schur} generalize such sequences of roots to ``Schur sequences'', which allow imaginary (Schur) roots as terms. 
A para-exceptional sequence may contain more than one term from $\nrig$
(as many as one for each component of $\Upsilon^c$, see Corollary~\ref{cor:maximal_non_exceptional}).
But $\undim R_{\beta,r_\beta} = \delta$ for any $\beta\in\Xi^c$, so if $(X_1,\ldots,X_k)$ is a para-exceptional sequence, then $(\undim X_1,\ldots,\undim X_k)$ may repeat the term $\delta$ and thus may fail to be a Schur sequence.
\end{remark}

In light of the bijection $X \mapsto t_{\undim X}$ from $\E$ to $\set{t\in T\mid t\le_T c}$, 
by Proposition~\ref{Cc alph} and Theorem~\ref{Ccplus chain sys}, there is a bijection $\omega: \excep \cup \nrig \rightarrow T_V\cup T_H \cup F$ given by
\[\omega(X) = \begin{cases} t_{\undim X} & \text{if $X \in \excep$}\\f_\gamma & \text{if $X = F_\gamma\in \nrig$}.\end{cases}\]  

We are now prepared to state and prove the main result of this section.

\begin{theorem}\label{thm:bijection_on_chain_systems}
Let $\Lambda$ be a connected tame hereditary algebra such that $\Lambda$ has more than one non-homogeneous tube.
\begin{enumerate}[label=\bf\arabic*., ref=\arabic*]
\item \label{Cpx bcs bij}
$\Cpx$ is a binary chain system with binary compatibility relation~$\Bpx$.
\item \label{Cpx bij}
The bijection $\omega$ is an isomorphism of chain systems from $\Cpx$ to~$\C_c^+$.
\item\label{PCpx McSul}
The map $\omega$ induces an isomorphism from $\P(\Cpx)\cong\Preeq(\Cpx)$ to $[1,c]_{T\cup F}$.
\end{enumerate}
\end{theorem}

Later, we will give a completely representation-theoretic proof of Theorem~\ref{thm:bijection_on_chain_systems}.\ref{Cpx bcs bij} as part of Theorem~\ref{thm:binary_chain_full},
by realizing $\P(\Cpx)$ as the poset of para-exceptional 
subcategories defined in Definition~\ref{def:para} (proved to be a lattice in Theorem~\ref{thm:join}).
But here we use the combinatorics of $[1,c]_{T\cup F}$.

\begin{proof}
In light of Proposition~\ref{bin isom}.\ref{not yet} and Theorem~\ref{Ccplus chain sys}, to prove Assertions~\ref{Cpx bcs bij} and~\ref{Cpx bij}, we need only to check that the bijection $\omega: \excep \cup \nrig \rightarrow T \cup F$ is an isomorphism from~$\Bpx$ to the binary compatibility relation described in Theorem~\ref{Ccplus chain sys}.

Lemma~\ref{lem:tubes_ortho} shows that $\set{(F_\beta, F_\gamma), (F_\gamma,R_{\beta_k}), (R_{\beta,k},F_\gamma)}\subseteq \Bpx$ for all $\beta,\gamma\in\Xi^c$ and~$k$ such that $\beta\not\equiv_\Upsilon\gamma$ and $1\le k<r_\beta$.

Suppose $\beta,\gamma\in\Xi^c$ have $\beta\equiv_\Upsilon\gamma$ and suppose $1\le k<r_\beta$.
Proposition~\ref{prop:hom_ext_vanish} says that $F_\gamma R_{\beta,k}\in\Bpx$ if and only if $\gamma\not\in\{c(\beta),c^2(\beta),\ldots,c^k(\beta)\}$ and $(R_{\beta,k}, F_\gamma)\in\Bpx$ if and only if $\gamma\not\in\{\beta,c(\beta),\ldots,c^{k-1}(\beta)\}$.

Finally, suppose $X,Y\in\E$.
Then Theorem~\ref{thm:IS} and Proposition~\ref{prop:complete_ex_maximal} imply that $(X,Y)\in\Bpx$ if and only if $1<_T\omega(X)\omega(Y)\le_Tc$.
Thus $\omega$ is an isomorphism as desired.

In light of Assertion~\ref{Cpx bij}, Proposition~\ref{isom isom} 
implies that $\Preeq(\Cpx)$ and $\Preeq(\C_c^+)$ are isomorphic.
More specifically, the isomorphism sends the equivalence class (as a prefix of $\Cpx$) of any para-exceptional sequence $(X_1,\ldots,X_k)$ to the word $\omega(X_1)\cdots\omega(X_k)$.
Two words are equivalent as prefixes of $\C_c^+$ if and only if their products are the same, and Assertion~\ref{PCpx McSul} follows.
\end{proof}

\begin{remark}
The proof above shows that $\omega$ induces a bijection from $\Cpx$ to~$\C_c^+$, even in the case where $\Lambda$ has only one non-homogeneous tube.
In that case, $\C_c^+$ is not a chain system, because $F$ is the set of all translations in $[1,c]_T$ and each $f\in F$ is the product of two reflections in $[1,c]_T$.
\end{remark}


\section{Wide subcategories in a tube category}\label{sec:one tube}
In preparation for a representation-theoretic proof of Theorem~\ref{thm:bijection_on_chain_systems}, we need to better understand para-exceptional sequences, especially as they relate to the non-homogeneous tubes of a connected tame hereditary algebra.
As a first step, we consider the analogous sequences and categories in a single tube.

For the rest of the section, $\T$ is an \newword{(abstract) tube category}, meaning a category $\T$ that is is exact equivalent to the category of finite-dimensional \newword{nilpotent} representations of the quiver consisting of an oriented $r$-cycle, which are defined analogously to the representations of acyclic quivers (Section~\ref{sec:reps}). 
The \newword{rank} of $\T$ is this~$r$. 
If $r=1$, then $\T$ is called \newword{homogeneous} and if $r>1$, then it is \newword{non-homogeneous}. 
In this section, we apply definitions that were earlier given in terms of $\mods\Lambda$ (including, bricks, exceptional modules, $(-)^\perp$ and so forth) 
in the category $\T$ by taking the same definitions verbatim except with  $\T$ replacing $\mods\Lambda$.
Lemma~\ref{lem:ortho_wide} holds for subcategories of $\T$ by the same proof. 

We will not make explicit use of the exact equivalence of $\T$ with a category of nilpotent representations, but we do quote some results that come from that point of view. 
Rather, our main tools for dealing with $\T$ come from the following observation:
For any choice of $r$, there exists a connected tame hereditary algebra~$\Lambda$ and a tube $\T \subseteq \mods\Lambda$ admitting an exact equivalence to the category of finite-dimensional nilpotent representations of an $r$-cycle.  Thus we are free to quote results that were proved for tubes within tame hereditary algebras.
For example, we use, for $r>1$, the characterization of indecomposable modules in homogeneous tubes from Section~\ref{sec:brick_sequences}, as well as Propositions~\ref{prop:bricks_in_tube} and~\ref{prop:hom_ext_vanish}.
For $r=1$, the same purpose is fulfilled by Lemma~\ref{lem:homogeneous}.
References to quasi-simple modules in tubes within tame hereditary algebras will be replaced by references to simple modules here, and we take $\Xi^c$ to be an indexing set for the set of simples. 

Let $\Ct$ be the set of $\brick(\T)$-brick sequences. These were previously studied in~\cite{IgusaSen} under the name ``soft exceptional sequences''.
Write $\Bt$ for the binary compatibility relation on the alphabet $\brick(\T)$ consisting of all two-term $\brick(\T)$-brick sequences.
We will prove the following theorem and, along the way, gather tools that will be used to prove an analogous theorem about regular para-exceptional sequences and their wide subcategories.

\begin{theorem}\label{thm:tube_binary_chain}
Let $\T$ be a tube category.
\begin{enumerate}[label=\bf\arabic*., ref=\arabic*]
\item \label{Ct bcs}
The set $\Ct$ is a binary chain system with binary compatibility relation $\Bt$.
\item \label{Pc to wideT}
The map $\Preeq(\Ct) \rightarrow \wide(\T)$ that sends $P$ to $\sW(X_1,\ldots,X_k)$, for any representative $(X_1,\ldots,X_k) \in P$, is an isomorphism of posets.
\item \label{Ct labels cov}
If $\V,\W\in\wide(\T)$ have $\W\subseteq\V$, then $\V$ covers $\W$ in $\wide(\T)$ if and only if $\V\cap{}^\perp\W$ contains a unique brick.
\item \label{Ct labels}
If $P\covered P'$ in $\Preeq(\Ct)$ is labeled by $X$ and sent by the  isomorphism in Assertion~\ref{Pc to wideWEH} to $\W \covered \V$, then $X$ is the unique brick in $\V \cap {}^\perp \W$.
\end{enumerate}
\end{theorem}

The proof of Theorem~\ref{thm:tube_binary_chain} consists of gathering the analogs in $\T$ for results from Section~\ref{ex chain sec} that were used in the proof of Theorem~\ref{thm:exceptional_binary_chain}, and then giving that proof verbatim with the appropriate replacements.

We begin with the following, which follows from \cite[Lemma~3.4]{IgusaSen}. To avoid confusion with the notions inside $\T$, we use the symbols $\excep_\Gamma$ and $B_{\e_\Gamma}$ for the exceptional modules $\excep$ and the binary relation $B_\e$ for the algebra $\Gamma$ appearing in the statement.

\begin{proposition}\label{prop:IgusaSen}
    Suppose that $\T$ has rank $r$. Then there exists a tensor algebra~$\Gamma$ of type $C_r$ and a bijection $\brick(\T) \rightarrow \excep_\Gamma$ which induces a bijection $B_\T \rightarrow B_{\e_\Gamma}$. Thus, by Proposition~\ref{bin isom}.\ref{not yet}, $\C_\T$ is a binary chain system which is isomorphic to~$\C_{\e_\Gamma}$.
\end{proposition}

We next quote five important results from \cite{dichev_thesis,IPT} about wide subcategories in tube categories. The results from \cite{dichev_thesis} were proved for abstract tube categories and the results of \cite{IPT} were proved in the context of tubes in tame hereditary algebras.  Both papers work over algebraically closed fields, but this assumption is not necessary for the results quoted here. 
The first quoted result is \cite[Proposition~2.3.24(ii)]{dichev_thesis}.

\begin{proposition}\label{prop:one tube_anti_isom}    
 The map $\W \mapsto \W^\perp$ is an anti-automorphism of $\wide(\T)$ with inverse $\W \mapsto {}^\perp \W$.
\end{proposition}

The second result is a weak form of \cite[Theorem~2.2.10]{dichev_thesis}. 

\begin{theorem}\label{dichev 2.2.10}
Let $\W \subseteq \T$ be a wide subcategory which is  representation-infinite and connected. Then $\W$ is a tube category.
\end{theorem}

The third result is a restatement of \cite[Lemma~8.1]{IPT}. 

\begin{lemma}\label{IPT 8.1}
Let $\W \subseteq \T$ be a wide subcategory.
Then~$\W$ is representation-finite if and only if there is a simple module of $\T$ such that $\W$ is contained in the wide subcategory generated by the other simples. 
\end{lemma}

The fourth result is part of \cite[Proposition~9.10]{IPT} together with a fact that was established in the proof of \cite[Proposition~9.10]{IPT}.

\begin{proposition}\label{IPT 9.10} 
Let $\W \subseteq \T$ be a wide subcategory.
Then there exists an orthogonal decomposition $\W=\U\oplus\V$ such that every indecomposable module in $\V$ is exceptional and $\U$ is either 0 or a tube category.
\end{proposition}

For the fifth result, say a subcategory $\W$ is \newword{bounded} if there exists an integer~$k$ such that every indecomposable $X\in\W$ has length at most $k$.  
The following is a weak form of \cite[Theorem~2.3.25]{dichev_thesis}.
\begin{theorem}\label{dichev 2.3.25}
Let $\W \subseteq \T$ be a wide subcategory.
Then~$\W$ is bounded if and only if $\W^\perp$ is unbounded if and only if ${}^\perp \W$ is unbounded.
\end{theorem}

Some of the crucial results from Section~\ref{ex chain sec} that were stated in the context of $\mods\Lambda$ also hold in $\T$ because we can identify $\T$ with a tube in $\mods\Lambda$ for some connected tame hereditary algebra $\Lambda$.
For example, Lemma~\ref{lem:rep_finite_exceptional} holds for wide subcategories in $\T$ and we will use this in that context below.
(Since $\T \subseteq \mods\Lambda$ is a wide subcategory, any wide subcategory of $\T$ is also a wide subcategory of $\mods\Lambda$. Lemma~\ref{lem:rep_finite_exceptional} then holds for wide subcategories of $\T$ because the property of being exceptional is an intrinsic property of a wide subcategory.) 

We will also use the following characterization of exceptional subcategories of~$\T$, whose proof draws heavily on the results of \cite{dichev_thesis,IPT} quoted above.

\begin{proposition}\label{prop:wide_tube}  
Let $\W \subseteq \T$ be a wide subcategory.
Then the following are equivalent.
\begin{enumerate}[label=\rm(\roman*), ref=(\roman*)]
\item \label{tube exc}
$\W$ is exceptional.
\item \label{tube rep-fin}
$\W$ is representation-finite.
\item \label{tube exists right}
There exists a non-exceptional brick $X \in \T$ such that $\W \subseteq X^\perp$.
\item \label{tube exists left}
There exists a non-exceptional brick $X \in \T$ such that $\W \subseteq {}^\perp X$.
\item \label{tube all exc}
Every indecomposable module in $\W$ is exceptional.
\item \label{tube perp nonexc}
Both $\W^\perp$ and ${}^\perp \W$ are non-exceptional.
\item \label{tube perp nonexc alt}
At least one of $\W^\perp$ and ${}^\perp \W$ is non-exceptional.
\end{enumerate}
If these equivalent conditions fail, then there is a unique orthogonal decomposition $\W = \U \oplus \V$ with $\V$ exceptional and $\U$ a tube category.  
\end{proposition}

\begin{proof}
If $\T$ is homogeneous, then Lemma~\ref{lem:homogeneous} says there is a unique brick $X_\T$ in $\T$ and $X_\T$ is non-exceptional.
By Proposition~\ref{prop:semibrick}, 
the wide subcategories of $\T$ are $0$ and $\T = \Filt(X_\T)$. 
Now 0 is exceptional and $\T$ is not, and it is straightforward to verify the proposition in this case.
Suppose for the remainder of the proof that $\T$ is non-homogeneous.

Suppose $\W$ is representation-infinite.
Since $|\brick(\T)| < \infty$ by Proposition~\ref{prop:bricks_in_tube}, also $|\wide(\T)| < \infty$ by Proposition~\ref{prop:semibrick}. 
Thus there is an orthogonal decomposition $\W = \U \oplus \V$ such that $\U$ is a connected, representation-infinite wide subcategory. 
Theorem~\ref{dichev 2.2.10} says that $\U$ is a tube category. 
It follows that $\U$, and thus also $\W$, is not exceptional.  
We have proved that \ref{tube exc}$\implies$\ref{tube rep-fin}.  
The fact that \ref{tube rep-fin}$\implies$\ref{tube exc} is Lemma~\ref{lem:rep_finite_exceptional}.

Suppose $\W$ is representation-finite. 
Lemma~\ref{IPT 8.1} says that there exists a simple $R_{\gamma,1} \in \T$ such that $\W$ is contained in the wide subcategory generated by the other simples.
Proposition~\ref{prop:hom_ext_vanish} then implies that $\W \subseteq R_{\gamma,r_\gamma}^\perp$. 
Since $R_{\gamma,r_\gamma}$ is a non-exceptional brick by Proposition~\ref{prop:bricks_in_tube}, we have shown that \ref{tube rep-fin}$\implies$\ref{tube exists right}.

In light of Proposition~\ref{prop:bricks_in_tube}, comparison of the two parts of Proposition~\ref{prop:hom_ext_vanish} shows that \ref{tube exists right}$\iff$\ref{tube exists left}.
If \ref{tube exists right} holds, then as before, there exists $R_{\gamma,r} \in \T$ such that $\W \subseteq R_{\gamma,r}^\perp$. 
For any indecomposable $X_{\beta,k} \in \W \subseteq X_{\gamma,r}^\perp$,  Proposition~\ref{prop:hom_ext_vanish} implies that $k < r$. 
Thus $X_{\beta,k}$ is exceptional by Proposition~\ref{prop:bricks_in_tube}. 
We see that \ref{tube exists right}$\implies$\ref{tube all exc}.

By Proposition~\ref{prop:bricks_in_tube}, there are only finitely many exceptional modules in $\T$.
Thus, if \ref{tube all exc} holds, then $\W$ is bounded, and therefore, by Theorem~\ref{dichev 2.3.25}, $\W^\perp$ and ${}^\perp \W$ are unbounded and therefore representation-infinite. The equivalence \ref{tube exc}$\iff$\ref{tube rep-fin} already proved then implies that $\W^\perp$ and ${}^\perp \W$ are non-exceptional.
Thus \mbox{\ref{tube all exc}$\implies$\ref{tube perp nonexc}}.
Trivially, \ref{tube perp nonexc}$\implies$\ref{tube perp nonexc alt}.

Now suppose \ref{tube perp nonexc alt} holds.
In light of Lemma~\ref{lem:ortho_wide}, the equivalence \ref{tube exc}$\iff$\ref{tube rep-fin} already proved implies that at least one of $\W^\perp$ and ${}^\perp \W$ is representation-infinite.
The indecomposable module $R_{\beta,k}$ has length~$k$, so there are only finitely many indecomposable modules of any give length.
We conclude that at least one of $\W^\perp$ and ${}^\perp \W$ is unbounded.
By Theorem~\ref{dichev 2.3.25}, $\W$ is bounded and thus representation-finite.
We see that \ref{tube perp nonexc alt}$\implies$\ref{tube rep-fin}, and we have proven the equivalence of the seven conditions.

Finally, suppose these conditions fail for $\W$.
Decomposing $\W = \U \oplus \V$ as in Proposition~\ref{IPT 9.10}, the equivalence of the seven conditions (for $\V$) implies that $\V$ is exceptional and representation-finite. Since $\W$ is representation-infinite by assumption, we then have that $\U \neq 0$, and so $\U$ is a tube category.
\end{proof}

We now begin to gather analogs, in $\T$, of some results of Section~\ref{ex chain sec} about exceptional sequences and wide subcategories.
We begin with an analog of Lemma~\ref{lem:exceptional_index_arbitrary}, which follows immediately from Lemma~\ref{lem:exceptional_index_arbitrary} and Proposition~\ref{prop:IgusaSen}.

\begin{lemma}\label{lem:one soft_index_arbitrary}
Let $(X_1,\ldots,X_k)$ be a $\brick(\T)$-brick sequence that is not maximal. 
Then for any $0 \leq i \leq k$ there exists a maximal $\brick(\T)$-brick sequence of the form $(X_1,\ldots,X_i,Y_1,\ldots,Y_j,X_{i+1},\ldots,X_k)$.
\end{lemma}

The following can be compared to Proposition~\ref{prop:wide_fg}.

\begin{proposition}\label{prop:tube_generation_1} 
Let $\W \subseteq \T$ be a wide subcategory.
The following are equivalent.
\begin{enumerate}[label=\rm(\roman*), ref=(\roman*)]
\item \label{tg1 exc}
$\W$ is exceptional.
\item \label{tg1 rep-fin}
$\W$ is representation-finite.
\item \label{tg1 right perp}
There exists a non-exceptional $\brick(\T)$-brick sequence $(X_1,\ldots,X_k)$ such that $\W = (X_1,\ldots,X_k)^\perp$.
\item \label{tg1 left perp}
There exists a non-exceptional $\brick(\T)$-brick sequence $(X_1,\ldots,X_k)$ such that $\W = {}^\perp(X_1,\ldots,X_k)$.
\item \label{tg1 no nonexc}
There does not exist a non-exceptional $\brick(\T)$-brick sequence $(X_1,\ldots,X_k)$ such that $\W = \sW(X_1,\ldots,X_k)$.
\item \label{tg1 no exc right}
There does not exist an exceptional $\brick(\T)$-brick sequence $(X_1,\ldots,X_k)$ such that $\W = (X_1,\ldots,X_k)^\perp$.
\item \label{tg1 no exc left}
There does not exist an exceptional $\brick(\T)$-bricksequence $(X_1,\ldots,X_k)$ such that $\W = {}^\perp(X_1,\ldots,X_k)$.
\end{enumerate}
\end{proposition}

\begin{proof}
The equivalence of \ref{tg1 exc} and \ref{tg1 rep-fin} is Proposition~\ref{prop:wide_tube}.

Suppose \ref{tg1 exc}.
Then Proposition~\ref{prop:wide_tube} says that there is a non-exceptional brick $X \in \T$ with $\W \subseteq X^\perp$. 
In light of Proposition~\ref{prop:one tube_anti_isom}, Proposition~\ref{prop:wide_tube}, applied to $X^\perp$, says that $X^\perp$ is exceptional. 
Viewing $\T$ as a tube in $\mods\Lambda$ for some connected tame hereditary algebra, Proposition~\ref{prop:wide_fg} says that $X^\perp$ is exact equivalent to $\mods\Lambda'$ for some finite-dimensional algebra $\Lambda'$.
(Some care must be taken to apply Proposition~\ref{prop:wide_fg}:  
The wide subcategory of~$\T$ that we here call $X^\perp$ using the $(-)^\perp$ operator in $\T$ coincides with the wide subcategory of $\mods\Lambda$ that is described as $X^\perp\cap\T$ using the $(-)^\perp$ operator in $\mods\Lambda$.
We apply Proposition~\ref{prop:wide_fg} to $X^\perp\cap\T$.) 
This exact equivalence to $\mods\Lambda'$ allows us to use Proposition~\ref{prop:wide_fg} in $X^\perp$ to produce an exceptional sequence $(X_1,\ldots,X_k)$ in $X^\perp$ such that $(X_1,\ldots,X_k)^\perp \cap X^\perp= (X_1,\ldots,X_k,X)^\perp = \W$. 
We see that \ref{tg1 exc}$\implies$\ref{tg1 right perp}.

Conversely, suppose $(X_1,\ldots,X_k)$ is a non-exceptional $\brick(\T)$-brick sequence with $\W = (X_1,\ldots,X_k)^\perp$.
Then one of the $X_i$ is non-exceptional and $\W \subseteq X_i^\perp$, and thus $\W$ is exceptional by Proposition~\ref{prop:wide_tube}.
We have proved \ref{tg1 exc}$\iff$\ref{tg1 rep-fin}$\iff$\ref{tg1 right perp}, and the symmetric argument shows the equivalence of these with \ref{tg1 left perp}.

Suppose \ref{tg1 exc} fails.
Then ${}^\perp\W$ is exceptional by Proposition~\ref{prop:wide_tube}, so there is an exceptional sequence $(X_1,\ldots,X_k)$ with $\sW(X_1,\ldots,X_k) = {}^\perp \W$.
By Lemma~\ref{lem:ortho_wide} and Proposition~\ref{prop:one tube_anti_isom}, we see that ${(X_1,\ldots,X_k)^\perp = \W}$, so that \ref{tg1 no exc right} fails.

Now suppose \ref{tg1 no exc right} fails, so that $\W = (X_1,\ldots,X_k)^\perp$ for some exceptional sequence $(X_1,\ldots,X_k)$ in~$\T$. 
By the equivalence of \ref{tg1 exc} and \ref{tg1 left perp}, $\sW(X_1,\ldots,X_k)$ is ${}^\perp(Y_1,\ldots,Y_j)$ for some non-exceptional $\brick(\T)$-brick sequence $(Y_1,\ldots,Y_j)$.
Now Lemma~\ref{lem:ortho_wide} and Proposition~\ref{prop:one tube_anti_isom} imply that \ref{tg1 no nonexc} fails, because
\[\sW(Y_1,\ldots,Y_j) = ({}^\perp(Y_1,\ldots,Y_j))^\perp = (X_1,\ldots,X_k)^\perp = \W.\]

Finally, suppose \ref{tg1 no nonexc} fails, so that $\W = \sW(X_1,\ldots,X_k)$ for some non-exceptional $\brick(\T)$-brick sequence $(X_1,\ldots,X_k)$. 
Thus, one of the $X_i$ is a non-exceptional brick in $\W$, so \ref{tg1 exc} fails by Proposition~\ref{prop:wide_tube}.

We have proved \ref{tg1 exc}$\iff$\ref{tg1 no nonexc}$\iff$\ref{tg1 no exc right}, and the symmetric argument shows the equivalence of these with \ref{tg1 no exc left}.
\end{proof}

Next we prove analogs of Proposition~\ref{prop:complete_ex_maximal} and of Proposition~\ref{prop:maximal_exceptional}.

\begin{proposition}\label{prop:one complete_generates_tube}
Let $(X_1,\ldots,X_k)$ be a $\brick(\T)$-brick sequence. Then the following are equivalent.
\begin{enumerate}[label=\rm(\roman*), ref=(\roman*)]
\item \label{ct1} $(X_1,\ldots,X_k) \in \Ct$, that is, $(X_1,\ldots,X_k)$ is maximal.
\item \label{ct2} $k = r$ (the rank of $\T$).
\item \label{ct3} $\sW(X_1,\ldots,X_k) = \T$.
\item \label{ct4} $(X_1,\ldots,X_k)^\perp = 0$.
\item \label{ct5} ${}^\perp(X_1,\ldots,X_k) = 0$.
\end{enumerate}
\end{proposition}

\begin{proof}
The equivalence \ref{ct1}$\iff$\ref{ct2} follows from Propositions~\ref{prop:IgusaSen} and~\ref{prop:complete_ex_maximal}.

Now observe that ${}^\perp(X_1,\ldots,X_k)=0$ if and only if no brick can be appended to the end of $(X_1,\ldots,X_k)$ to make a $\brick(\T)$-brick sequence. 
Thus Lemma~\ref{lem:one soft_index_arbitrary} implies that $(X_1,\ldots,X_k)$ is maximal if and only if ${}^\perp(X_1,\ldots,X_k)=0$. We see that \ref{ct1}$\iff$\ref{ct5}. A similar argument shows \ref{ct1}$\iff$\ref{ct4}.

By Proposition~\ref{prop:one tube_anti_isom} and Lemma~\ref{lem:ortho_wide}, we have 
$\sW(X_1,\ldots,X_k) = ({}^\perp(X_1,\ldots,X_k))^\perp$ and ${}^\perp \sW(X_1,\ldots,X_k) = {}^\perp(X_1,\ldots,X_k).$
This implies the equivalence \ref{ct3}$\iff\ref{ct5}$. A similar argument shows \ref{ct3}$\iff$\ref{ct4}.
\end{proof}

\begin{proposition}\label{prop:one maximal_soft}  
Let $(X_1,\ldots,X_k)$ be a $\brick(\T)$-brick sequence.   
Then each of the following conditions is equivalent to $(X_1,\ldots,X_k)$ being maximal.
\begin{enumerate}[label=\rm(\roman*), ref=(\roman*)]
\item For $0\le i\le j\le k$, $\sW(X_{i+1},\ldots,X_j) = {}^\perp(X_1,\ldots,X_i) \cap (X_{j+1},\ldots,X_k)^\perp$.
\item There exist $i\le j$ with $\sW(X_{i+1},\ldots,X_j) = {}^\perp(X_1,\ldots,X_i) \cap (X_{j+1},\ldots,X_k)^\perp$.
\end{enumerate}
\end{proposition}

\begin{proof}
Fix indices $i \leq j \in \set{0,\ldots,k}$ and let $\W = {}^\perp(X_1,\ldots,X_i) \cap (X_{j+1},\ldots,X_k)^\perp$.
Lemma~\ref{lem:one soft_index_arbitrary} implies that $(X_1,\ldots,X_k)$ is maximal if and only if $(X_{i+1},\ldots,X_j)$ is maximal in $\W$. 
Proposition~\ref{IPT 8.1} says that there is an orthogonal decomposition $\W=\U\oplus\V$ such that $\U$ is a tube category (or 0) and every module in $\V$ is exceptional.
Therefore $\V$ is exceptional by Proposition~\ref{prop:wide_tube}.
Moreover (taking care as explained in the proof of Proposition~\ref{prop:tube_generation_1}), Proposition~\ref{prop:wide_fg} says that there is a finite-dimensional algebra $\Lambda'$ such that $\mods\Lambda'$ is exact equivalent to~$\V$.
Proposition~\ref{prop:one complete_generates_tube} (applied in~$\U$) and Proposition~\ref{prop:complete_ex_maximal} (applied in $\V \simeq \mods\Lambda')$ then imply that $(X_{i+1},\ldots,X_j)$ is maximal in $\W$ if and only if $\sW(X_{i+1},\ldots,X_j) = \W$. 
\end{proof}

We conclude our preparations by pointing out that the analog of Proposition~\ref{prop:perp_prefix} holds in~$\T$.
The proposition can be proved by repeating the proof of Proposition~\ref{prop:perp_prefix} verbatim but replacing Lemma~\ref{lem:exceptional_index_arbitrary} by Lemma~\ref{lem:one soft_index_arbitrary} and Proposition~\ref{prop:maximal_exceptional} by Proposition~\ref{prop:one maximal_soft}.

\begin{proposition}\label{prop:perp_prefix_soft}
Let $(X_1,\ldots,X_k)$ and $(X'_1,\ldots,X'_i)$ be $\brick(\T)$-brick sequences. Then the following are equivalent.
\begin{enumerate}[label=\rm(\roman*), ref=(\roman*)]
\item \label{pps1} $(X_1,\ldots,X_k)$ and $(X'_1,\ldots,X'_i)$ are equivalent as prefixes in $\Ct$.
\item \label{pps2} $(X_1,\ldots,X_k)$ and $(X'_1,\ldots,X'_i)$ are equivalent as postfixes in $\Ct$.
\item \label{pps3} $\sW(X_1,\ldots,X_k) = \sW(X'_1,\ldots,X'_i)$.
\item \label{pps4} $(X_1,\ldots,X_k)^\perp = (X'_1,\ldots,X'_i)^\perp$.
\item \label{pps5} ${}^\perp(X_1,\ldots,X_k) = {}^{\perp}(X'_1,\ldots,X'_i)$.
\end{enumerate}
\end{proposition}

Now, the main result of this section, Theorem~\ref{thm:tube_binary_chain}, is proved as follows.

\begin{proof}[Proof of Theorem~\ref{thm:tube_binary_chain}]
    Assertion \ref{Ct bcs} was proved as part of Proposition~\ref{prop:IgusaSen}. For the other assertions we repeat the proof of Theorem~\ref{thm:exceptional_binary_chain}.\ref{Pc to fwide}-\ref{Cx labels} except replacing ``exceptional sequence'' by ``$\brick(\T)$-brick sequence'' throughout and replacing references to 
Propositions~
\ref{prop:wide_fg},  \ref{prop:maximal_exceptional}, and~\ref{prop:perp_prefix} and Lemma~\ref{lem:exceptional_index_arbitrary} by their analogs in~$\T$ proved in this section (Propositions~
\ref{prop:tube_generation_1}, \ref{prop:one maximal_soft}, and~\ref{prop:perp_prefix_soft} and Lemma~\ref{lem:one soft_index_arbitrary}).
\end{proof}

We conclude this section by recording the following consequence of Theorem~\ref{thm:tube_binary_chain}.\ref{Pc to wideT}, Proposition~\ref{prop:IgusaSen},  Theorem~\ref{thm:exceptional_binary_chain}.\ref{Pc to fwide}, and Proposition~\ref{isom isom}.

\begin{corollary}\label{cor:tube_type_C}
    Let $\T$ be a tube category of rank $r \geq 1$. Then $\wide(\T) \cong \wide(\Gamma)$ for $\Gamma$ a tensor algebra of type $C_r$.
\end{corollary}


\section{Regular para-exceptional sequences}\label{sec:rep_theory_tubes}
We now return to the setting of connected tame hereditary algebras. In this section, we use the results of Section~\ref{sec:one tube} to study those para-exceptional sequences whose terms are all regular.

Let $\excep_\reg\subseteq\excep$ be the set of all regular exceptional modules, so that 
$\excep_\reg\cup\nrig$ is the set of regular para-exceptional modules. As a consequence of \cite[Proposition~7.14]{affscat}, an exceptional module $X$ is regular if and only if $\undim X\in \Upsilon^c$.
Thus, the bijection $X \mapsto t_{\undim X}$ restricts to bijections $\excep_\reg \to T_H$ and ${\excep\setminus\excep_\reg \to T_V}$.
A \newword{regular para-exceptional sequence} is an $(\excep_\reg\cup\nrig)$-brick sequence.  
We write $\Brpx$ for the binary compatibility relation on the alphabet $\excep_\reg\cup\nrig$ consisting of all two-term regular para-exceptional sequences and $\Crpx$ for the set $\C(\Brpx)$ of maximal regular para-exceptional sequences.

Recall that $\N=\set{\T_1,\ldots\T_m}$ is the set of non-homogeneous tubes.  
The important insight about regular para-exceptional sequences is most readily phrased in terms of shuffles, as in Section~\ref{shuffle sec}:
$\Crpx=\C_{\T_1}\shuffle\cdots\shuffle\C_{\T_m}$.
(This is immediate from Lemma~\ref{lem:tubes_ortho}.)  
It is also apparent that $\Brpx$ is obtained from $B_{\T_1},\dots,B_{\T_m}$ by the construction described in Proposition~\ref{shuf prod}.\ref{shuf bin}, iterated if $m>2$.

Recall that $\nonhomo$ denotes the wide subcategory of $\mods\Lambda$ consisting of regular modules whose indecomposable direct summands lie in the non-homogeneous tubes. We note that $\nonhomo = \sW(\excep_\reg) = \sW(\excep_\reg \cup \nrig)$. We also recall from Proposition~\ref{prop:tubes_ortho}
that there is an orthogonal decomposition $\nonhomo = \T_1 \oplus \cdots \oplus \T_m$. This induces a decomposition of the interval $\wide(\nonhomo) = [0,\nonhomo] \subseteq \wide \Lambda$ as $\wide(\nonhomo) \cong \prod_{i = 1}^m \wide(\T_i)$. 

These considerations, together with Proposition~\ref{shuf prod} and Theorem~\ref{thm:tube_binary_chain}, amount to an entirely representation-theoretic proof of the following theorem, which holds without any assumption on the number of non-homogeneous tubes of~$\Lambda$.

\begin{theorem}\label{thm:soft_binary_chain}
Let $\Lambda$ be a connected tame hereditary algebra. 
\begin{enumerate}[label=\bf\arabic*., ref=\arabic*]
\item \label{Crpx bcs}
The set $\Crpx$ is a binary chain system with binary compatibility relation $\Brpx$.
\item \label{Pc to wideWEH}
The map $\Preeq(\Crpx) \rightarrow \wide(\exreg)$ that sends $P$ to $\sW(X_1,\ldots,X_k)$, for any representative $(X_1,\ldots,X_k) \in P$, is an isomorphism of posets.
\item \label{Crpx labels cov}
If $\V,\W\in\wide \exreg$ have $\W\subseteq\V$, then $\V$ covers $\W$ in $\wide \exreg$ if and only if $\V\cap{}^\perp\W$ contains a unique brick.
\item \label{Crpx labels}
If $P\covered P'$ in $\Preeq(\Crpx)$ is labeled by $X$ and sent by the  isomorphism in Assertion~\ref{Pc to wideWEH} to $\W \covered \V$, then $X$ is the unique brick in $\V \cap {}^\perp \W$.
\end{enumerate}
\end{theorem}

The remainder of the section gathers facts about regular para-exceptional sequences that are obtained easily from results of Section~\ref{sec:one tube} and the observations (based on Proposition~\ref{prop:tubes_ortho}) that $\Crpx=\C_{\T_1}\shuffle\cdots\shuffle\C_{\T_m}$ and $\wide(\nonhomo) \cong \prod_{i = 1}^m \wide(\T_i)$. 
When the proofs are immediate from these observations and a result from Section~\ref{sec:one tube}, we will merely identify that result.  
Many of these facts are analogs, for regular para-exceptional sequences and wide subcategories in $\wide(\sW(\exreg))$, of facts about exceptional sequences and exceptional subcategories that were reviewed in Section~\ref{ex chain sec}.

\begin{proposition}\label{prop:tube_anti_isom}  \quad
\begin{enumerate}[label=\bf\arabic*., ref=\arabic*]
\item \label{tai 1}
The map $\W \mapsto \W^\perp \cap \nonhomo$ is an anti-automorphism of $\wide(\nonhomo)$, with inverse $\W \mapsto {}^\perp \W \cap \nonhomo$.
\item \label{tai 2}
The map $\W \mapsto \W^\perp \cap \reg$ is an anti-automorphism of $\wide(\reg)$, with inverse $\W \mapsto {}^\perp \W \cap \reg$.  
\end{enumerate}
\end{proposition}
\begin{proof}
Follows from Proposition~\ref{prop:one tube_anti_isom}, Lemma~\ref{lem:tubes_ortho}, and Proposition~\ref{prop:tubes_ortho}.
\end{proof}

\begin{definition}\label{def:full}
Let $(X_1,\ldots,X_k)$ be a para-exceptional sequence. We say that $(X_1,\ldots,X_k)$ is \newword{fully non-exceptional} if it is not an exceptional sequence and contains a non-exceptional module from every non-homogeneous tube. 
Note that Proposition~\ref{prop:hom_ext_vanish} implies that a fully non-exceptional para-exceptional sequence contains \emph{exactly one} non-exceptional module from each non-homogeneous tube.
\end{definition}

\begin{remark}\label{rem:exclude_kronecker}
    The requirement that $(X_1,\ldots,X_k)$ is not an exceptional sequence in Definition~\ref{def:full} is made to exclude the empty sequence in the case where $\N = \emptyset$. (This only occurs in rank~$2$.) Indeed, in this case, the empty sequence is the only regular para-exceptional sequence. It is also an exceptional sequence, even though, vacuously, it contains a non-homogeneous brick from every non-homogeneous tube.
\end{remark}

\begin{corollary}\label{cor:maximal_non_exceptional}
Let $(X_1,\ldots,X_k)$ be a maximal regular para-exceptional sequence. If $\N \neq \emptyset$, then $(X_1,\ldots,X_k)$ is fully non-exceptional.
\end{corollary}

\begin{proof}
Suppose $(X_1,\ldots,X_k)$ is a regular para-exceptional sequence that is not fully non-exceptional. 
Thus for some $j\in\set{1,\ldots,m}$, no $X_i$ is both non-exceptional and contained in $\T_j$. Proposition~\ref{prop:wide_tube} and Lemma~\ref{lem:tubes_ortho} then imply that there exists a non-exceptional brick $X \in \T$ such that $\sW(X_1,\ldots,X_k) \subseteq X^\perp$. 
Then $(X_1,\ldots,X_k,X)$ is a regular para-exceptional sequence, so $(X_1,\ldots,X_k)$ is not maximal.
\end{proof}
\begin{lemma}\label{lem:soft_index_arbitrary}
Let $(X_1,\ldots,X_k)$ be a regular para-exceptional sequence which is not maximal. 
Then for any $0 \leq i \leq k$ there exist a maximal regular para-exceptional sequence of the form $(X_1,\ldots,X_i,Y_1,\ldots,Y_j,X_{i+1},\ldots,X_k)$.
\end{lemma}
\begin{proof}
Follows from Lemmas~\ref{lem:one soft_index_arbitrary} and~\ref{lem:tubes_ortho}.
\end{proof}

\begin{proposition}\label{prop:tube_generation}
Let $\W \in \wide(\exreg)$. 
\begin{enumerate}[label=\bf\arabic*., ref=\arabic*]
\item There exists a regular para-exceptional sequence $(X_1,\ldots,X_k)$ such that $\W = \sW(X_1,\ldots,X_k)$.
\item There exists a regular para-exceptional sequence $(X_1,\ldots,X_k)$ such that $\W = (X_1,\ldots,X_k)^\perp \cap \exreg$.
\item There exists a regular para-exceptional sequence $(X_1,\ldots,X_k)$ such that $\W = {}^\perp(X_1,\ldots,X_k) \cap \exreg$.
\end{enumerate}
\end{proposition}
\begin{proof}
Follows from Proposition~\ref{prop:tube_generation_1}, Lemma~\ref{lem:tubes_ortho}, and Proposition~\ref{prop:tubes_ortho}.
\end{proof}

\begin{proposition}\label{prop:tube_generation_2}
Let $(X_1,\ldots,X_k)$ be a regular para-exceptional sequence.
\begin{enumerate}[label=\bf\arabic*., ref=\arabic*]
\item \label{tg2 1}
The following are equivalent:
\begin{enumerate}[label=\rm(\alph*), ref=(\alph*)]
\item $\sW(X_1,\ldots,X_k)$ is representation-finite.
\item $\sW(X_1,\ldots,X_k)$ is an exceptional subcategory.
\item $\sW(X_1,\ldots,X_k) \cap \T$ is representation-finite for all $T\in\N$.
\item $\sW(X_1,\ldots,X_k) \cap \T$ is an exceptional subcategory for all $\T\in\N$.
\item $(X_1,\ldots,X_k)^\perp \cap\T$ is representation-infinite for all $\T\in\N$.
\item ${}^\perp(X_1,\ldots,X_k) \cap\T$ is representation-infinite for all $\T\in\N$.
\item $(X_1,\ldots,X_k)$ is an exceptional sequence.
\end{enumerate}
\item \label{tg2 2}
The following are equivalent:
\begin{enumerate}[label=\rm(\alph*), ref=(\alph*)]
\item $(X_1,\ldots,X_k)^\perp \cap \nonhomo$ is representation-finite.
\item ${}^\perp(X_1,\ldots,X_k) \cap \nonhomo$ is representation-finite.
\item $(X_1,\ldots,X_k)^\perp \cap \nonhomo$ is an exceptional subcategory
\item ${}^\perp(X_1,\ldots,X_k) \cap \nonhomo$ is an exceptional subcategory.
\item $(X_1,\ldots,X_k)^\perp \cap \T$ is representation-finite for all $\T\in\N$.
\item ${}^\perp(X_1,\ldots,X_k)\cap \T$ is representation-finite for all $\T\in\N$.
\item $(X_1,\ldots,X_k)^\perp \cap \T$ is an exceptional subcategory for all $\T\in\N$.
\item ${}^\perp(X_1,\ldots,X_k) \cap \T$ is an exceptional subcategory for all $\T\in\N$.
\item $(X_1,\ldots,X_k)$ is fully non-exceptional.
\end{enumerate}
\end{enumerate}
\end{proposition}
\begin{proof}
Follows from Proposition~\ref{prop:tube_generation_1}, Lemma~\ref{lem:tubes_ortho}, and Proposition~\ref{prop:tubes_ortho}.
\end{proof}

\begin{proposition}\label{prop:complete_generates_tube}
Let $(X_1,\ldots,X_k)$ be a regular para-exceptional sequence. Then the following are equivalent.
\begin{enumerate}[label=\rm(\roman*), ref=(\roman*)]
\item \label{crp1} $(X_1,\ldots,X_k) \in \Crpx$, that is, $(X_1,\ldots,X_k)$ is maximal.
\item \label{crp2} $k = n-2 + m$.
\item \label{crp3} $\sW(X_1,\ldots,X_k) = \nonhomo$.
\item \label{crp4} $(X_1,\ldots,X_k)^\perp \cap \nonhomo = 0$.
\item \label{crp5} ${}^\perp(X_1,\ldots,X_k) \cap \nonhomo = 0$.
\end{enumerate}
\end{proposition}

\begin{proof}
The sum of the ranks of the non-homogeneous tubes is known to be $n - 2 + m$, see  \cite[Section~6]{DlabRingel}. The result thus follows from Propositions~\ref{prop:one complete_generates_tube} and~\ref{prop:tubes_ortho}.
\end{proof}

\begin{proposition}\label{prop:maximal_soft}
Let $(X_1,\ldots,X_k)$ be a regular para-exceptional sequence. 
Then each of the following conditions is equivalent to $(X_1,\ldots,X_k)$ being maximal.
\begin{enumerate}[label=\rm(\roman*), ref=(\roman*)]
\item $\sW(X_{i+1},\ldots,X_j) = {}^\perp(X_1,\ldots,X_i) \cap (X_{j+1},\ldots,X_k)^\perp \cap \exreg$ for all $i \leq j \in \{0,\ldots,k\}$.
\item $\sW(X_{i+1},\ldots,X_j) = {}^\perp(X_1,\ldots,X_i) \cap (X_{j+1},\ldots,X_k)^\perp \cap \exreg$ for some $i \leq j \in \{0,\ldots,k\}$.
\end{enumerate}
\end{proposition}
\begin{proof}
Follows from Propositions~\ref{prop:one maximal_soft} and~\ref{prop:tubes_ortho}.
\end{proof}

\begin{proposition}\label{prop:perp_prefix_reg}
Let $(X_1,\ldots,X_k)$ and $(Y_1,\ldots,Y_i)$ be regular para-exceptional sequences. Then the following are equivalent.
\begin{enumerate}[label=\rm(\roman*), ref=(\roman*)]
\item\label{reg1} $(X_1,\ldots,X_k)$ and $(Y_1,\ldots,Y_i)$ are equivalent as prefixes in $\Crpx$.
\item\label{reg2}  $(X_1,\ldots,X_k)$ and $(Y_1,\ldots,Y_i)$ are equivalent as postfixes in $\Crpx$.
\item\label{reg3}  $\sW(X_1,\ldots,X_k) = \sW(Y_1,\ldots,Y_i)$.
\item\label{reg4}  $(X_1,\ldots,X_k)^\perp \cap \nonhomo = (Y_1,\ldots,Y_i)^\perp \cap \nonhomo$.
\item\label{reg5}  $(X_1,\ldots,X_k)^\perp = (Y_1,\ldots,Y_i)^\perp$.
\item\label{reg6}  ${}^{\perp}(X_1,\ldots,X_k)\cap \nonhomo = {}^{\perp} (Y_1,\ldots,Y_i) \cap \nonhomo$.
\item\label{reg7}  ${}^{\perp}(X_1,\ldots,X_k) = {}^{\perp}(Y_1,\ldots,Y_i)$.
\end{enumerate}
\end{proposition}
\begin{proof}
The equivalence of \ref{reg1}, \ref{reg2}, \ref{reg3}, \ref{reg4}, and \ref{reg6} follows from Proposition~\ref{prop:perp_prefix_soft} and~\ref{prop:tubes_ortho}. The implications \ref{reg5}$\implies$\ref{reg4} and \ref{reg7}$\implies$\ref{reg6} are trivial. The implications \ref{reg3}$\implies$\ref{reg5} and \ref{reg3}$\implies$\ref{reg7} follow from Lemma~\ref{lem:ortho_wide}.
\end{proof}

\section{A wide subcategory model for the McCammond-Sulway lattice}\label{sec:rep_theory_model}

In this section, we construct a wide subcategory model for the lattice $\P(\Cpx) \cong [1,c]_{T \cup F}$. 
We prove analogs of Proposition~\ref{prop:perp_prefix} and Theorem~\ref{thm:exceptional_binary_chain} for para-exceptional sequences using representation-theoretic arguments.
The first goal of this section is to understand why the obvious idea for a wide subcategory model does not work.
Recall that $\Lambda$ is always a connected tame hereditary algebra.

The considerations of Section~\ref{McSul chain sec} show that $\C_c^+$ is the disjoint union $\C_c \sqcup \C_c^F$. 
The following proposition includes the representation-theoretic analog of this fact. 

\begin{proposition}\label{prop:chain_system_union}
Let $(X_1,\ldots,X_k)$ be a para-exceptional sequence.
\begin{enumerate}[label=\bf\arabic*., ref=\arabic*]
\item \label{Cpx Cx}
If $(X_1,\ldots,X_k)$ is an exceptional sequence, then it is maximal as a para-exceptional sequence (that is, is is in $\Cpx$) if and only if it is maximal as an exceptional sequence (that is, it is in $\Cx$). Moreover, if it is maximal, then it is not a regular para-exceptional sequence.
\item \label{Cpx Crpx}
Suppose $\N \neq \emptyset$. If $(X_1,\ldots,X_k)$ is a regular para-exceptional sequence, then it is maximal as a para-exceptional sequence (that is, it is in~$\Cpx$) if and only if it is maximal as a regular para-exceptional sequence (that is, it is  in $\Crpx$). 
Moreover, if it is maximal, then it is fully non-exceptional.
\item \label{index_arbitrary_full}
If $(X_1,\ldots,X_k)$ is a non-maximal para-exceptional sequence and $0 \leq i \leq k$, then there exist $Y_1,\ldots,Y_j$ such~that $(X_1,\ldots,X_i,Y_1,\ldots,Y_j,X_{i+1},\ldots,X_k)$ is a maximal para-exceptional sequence.
\end{enumerate}
\end{proposition}

\begin{proof}
An exceptional sequence that is in $\Cpx$ is also maximal as an exceptional sequence. 
Conversely, suppose $(X_1,\ldots,X_k)\in\Cx$.
Then $\sW(X_1,\ldots,X_k) = \mods\Lambda$ by Proposition~\ref{prop:complete_ex_maximal}, so some $X_i$ is not regular.
Thus $(X_1,\ldots,X_k)$ is not a regular para-exceptional sequence. 
Also, Lemma~\ref{lem:hom_ext} says that any non-exceptional brick $X$ has $X \notin X_i^\perp$ and $X \notin {}^\perp X_i$. 
Thus $(X_1,\ldots,X_k)\in\Cpx$.
This is Assertion~\ref{Cpx Cx}.

Suppose for this paragraph that $\N \neq \emptyset$. A regular para-exceptional sequence that is in $\Cpx$ is also maximal as a regular para-exceptional sequence. 
Conversely, suppose that $(X_1,\ldots,X_k)\in\Crpx$, so that $(X_1,\ldots,X_k)$ is fully non-exceptional by Corollary~\ref{cor:maximal_non_exceptional}.
In particular, some $X_i$ is non-exceptional. 
Lemma~\ref{lem:hom_ext} says that any non-regular exceptional module $X$ has $X \notin X_i^\perp$ and $X \notin {}^\perp X_i$.
Thus $(X_1,\ldots,X_k)$ is also maximal as a para-exceptional sequence.
This is Assertion~\ref{Cpx Crpx}.

If $(X_1,\ldots,X_k)$ is a non-maximal para-exceptional sequence, then by Lemma~\ref{lem:hom_ext}, either Assertion~\ref{Cpx Cx} or Assertion~\ref{Cpx Crpx} applies.
In the former case, Assertion~\ref{index_arbitrary_full} follows from Lemma~\ref{lem:exceptional_index_arbitrary}.
In the latter case, Assertion~\ref{index_arbitrary_full} follows from Lemma~\ref{lem:soft_index_arbitrary}.
\end{proof}

We are now able to fill in the converse to Proposition~\ref{ttprime}.

\begin{corollary}\label{ttprime iff}
Let $W$ be a Coxeter group of affine type and let $c$ be a Coxeter element of~$W$ such that $\Upsilon^c$ has more than one component.
If $t,t'\in T_H$, then $tt'$ is a prefix of a word in $\C_c^F$ if and only if $tt'$ is a prefix of a word in $\C_c$.
\end{corollary}
\begin{proof}
One direction is Proposition~\ref{ttprime}.
If $tt'$ is a prefix of $\C_c$, then Theorem~\ref{thm:bijection_on_chain_systems} implies that $(\omega^{-1}(t),\omega^{-1}(t'))$ is a prefix of $\Cpx$, or in other words a para-exceptional sequence (and in fact an exceptional sequence).  
Since $t,t'\in T_H$, in fact $(\omega^{-1}(t),\omega^{-1}(t'))$ is a regular para-exceptional sequence, and is therefore a prefix of a sequence $(X_1,\ldots,X_k)$ that is maximal as a regular para-exceptional sequence.
Proposition~\ref{prop:chain_system_union}.\ref{Cpx Crpx} says that $(X_1,\ldots,X_k)$ is maximal as a para-exceptional sequence.
Thus $\omega(X_1)\cdots\omega(X_k)$ is a word in $\C_c^F$ with $tt'$ as a prefix.
\end{proof}

Proposition~\ref{prop:chain_system_union} allows us to restrict the isomorphism $\Cpx \rightarrow \C_c^+$ from Theorem~\ref{thm:bijection_on_chain_systems}.\ref{Cpx bij} to isomorphisms $\Cx \rightarrow \C_c$ and $\Crpx \rightarrow \C_c^F$.
The first of these isomorphisms is Corollary~\ref{cor:bijection_on_chain_systems_exceptional}.
We record the second of these isomorphisms as part of the following corollary. The second assertion of the corollary amounts to a representation-theoretic proof of \cite[Proposition~7.6]{McSul}.

\begin{corollary}\label{CcF cor}
Let $W$ be a Coxeter group of affine type and let $c$ be a Coxeter element of~$W$ such that $\Upsilon^c$ has more than one component.
\begin{enumerate}[label=\bf\arabic*., ref=\arabic*]
\item \label{CcF isom}
The isomorphism in  Theorem~\ref{thm:bijection_on_chain_systems}.\ref{Cpx bij} restricts to an isomorphism $\Crpx \rightarrow \C_c^F$ and this restriction induces an isomorphism $\P(\Crpx) \rightarrow [1,c]_{T_H \cup F}$.
\item \label{THF type B}
$[1,c]_{T_H\cup F}$ is isomorphic to a product of noncrossing partition lattices of type~B, with ranks equal to the number of simple roots in the components of~$\Upsilon^c$. 
\end{enumerate}
\end{corollary}

\begin{proof}
The isomorphism $\Crpx \rightarrow \C_c^F$ is obtained from Proposition~\ref{prop:chain_system_union} and Theorem~\ref{thm:bijection_on_chain_systems}.\ref{Cpx bij}, as described above.  
The isomorphism $\P(\Crpx) \rightarrow [1,c]_{T_H \cup F}$ then follows from Propositions~\ref{isom isom} and~\ref{factorable bin}.
Now $\P(\Crpx) \cong \prod_{\T \in \mathfrak{T}} \wide(\T)$ by Theorem~\ref{thm:soft_binary_chain}.\ref{Pc to wideWEH} and Proposition~\ref{prop:tubes_ortho}. 
For $\T$ a tube of rank $r$, $\wide(\T)$ is isomorphic to the noncrossing partition lattice of type $B_r$ by Corollary~\ref{cor:tube_type_C} and Theorem~\ref{thm:NC}.
(Recall that, on the level of Coxeter groups, types B and C coincide.)
\end{proof}

Proposition~\ref{prop:chain_system_union} suggests that we can construct a wide subcategory model for $\P(\Cpx)$ by simply combining the models $\P(\Crpx) \cong \wide(\nonhomo)$ and $\P(\Cx) \cong \ewide \Lambda$. 
The analog of Proposition~\ref{prop:perp_prefix} for para-exceptional sequences would be obtained by simply replacing ``exceptional'' with ``para-exceptional'' and ``$\Cx$'' with ``$\Cpx$'' throughout.
The following proposition implies that this verbatim analog does not work.
Recall that $\homo$ is the wide subcategory of homogeneous modules.

\begin{proposition}\label{prop:not_big_enough}
Suppose $\N \neq \emptyset$, let $(X_1,\ldots,X_n)$ be a maximal exceptional sequence, and let $(Y_1,\ldots,Y_k)$ be a maximal regular para-exceptional sequence. Then:
\begin{enumerate}[label=\bf\arabic*., ref=\arabic*]
\item \label{wide subsetneq}
$\sW(Y_1,\ldots,Y_k) = \nonhomo \subsetneq \mods\Lambda = \sW(X_1,\ldots,X_n)$.
\item \label{right perp subsetneq}
$(X_1,\ldots,X_n)^\perp = 0 \subsetneq \homo = (Y_1,\ldots,Y_k)^\perp$.
\item \label{left perp subsetneq}
${}^\perp(X_1,\ldots,X_n) = 0 \subsetneq \homo = {}^\perp(Y_1,\ldots,Y_k)$.
\end{enumerate}
\end{proposition}

\begin{proof}
Assertion~\ref{wide subsetneq} follows from Proposition~\ref{prop:complete_generates_tube} and Proposition~\ref{prop:complete_ex_maximal}. 
Lemma~\ref{lem:ortho_wide} says that $(X_1,\ldots,X_n)^\perp = 0 = {}^{\perp}(X_1,\ldots,X_n)$, but Lemmas~\ref{lem:hom_ext} and~\ref{lem:tubes_ortho} imply that $(Y_1,\ldots,Y_k)^\perp = \homo = {}^{\perp}(Y_1,\ldots,Y_k)$
\end{proof}

By Proposition~\ref{prop:chain_system_union}, any maximal exceptional sequence is also in $\Cpx$, and thus the sequences $(X_1,\ldots,X_n)$ and $(Y_1,\ldots,Y_k)$ in Proposition~\ref{prop:not_big_enough} are trivially equivalent as prefixes of $\Cpx$ and as postfixes of $\Cpx$.
A verbatim restatement of Proposition~\ref{prop:perp_prefix} for $\Cpx$ would thus say that 
$\sW(X_1,\ldots,X_n) = \sW(Y_1,\ldots,Y_i)$, $(X_1,\ldots,X_n)^\perp = (Y_1,\ldots,Y_i)^\perp$, and
${}^\perp(X_1,\ldots,X_n) = {}^{\perp}(Y_1,\ldots,Y_i)$, contradicting Proposition~\ref{prop:not_big_enough}.
We see that the isomorphisms $\P(\Crpx) \rightarrow \wide(\nonhomo)$ and $\P(\Cx) \rightarrow \ewide \Lambda$ (from Theorems~\ref{thm:soft_binary_chain} and~\ref{thm:exceptional_binary_chain}) do not induce an isomorphism $\P(\Cpx) \rightarrow \wide(\nonhomo) \cup \ewide\Lambda$. 
We now introduce the correct analogs of Proposition~\ref{prop:perp_prefix} and Theorem~\ref{thm:exceptional_binary_chain} for~$\Cpx$.
Recall that $\reg$ is the subcategory of regular modules. 
The notation $\join_{\wide \Lambda}$ stands for the join in the lattice of wide subcategories.

\begin{definition}\label{def:E_perp}
For $\Lambda$ a connected tame hereditary algebra and $\W\in\wide\Lambda$, define
\begin{align*}\aug(\W) &= \begin{cases}
\W \join_{\wide \Lambda} \homo & \text{if $\W$ is representation-infinite}\\
\W & \text{if $\W$ is representation-finite.}
\end{cases}\\
\red(\W) &= \begin{cases}
\W \cap \nonhomo & \text{if $\W \subseteq \reg$ and $\W \cap \nonhomo$ is representation-finite}\\
\W & \text{otherwise.}
\end{cases}\\
\overline{\W} &= {}^\perp(\aug(\W)^\perp).
\end{align*}
\end{definition}

The next lemma justifies the description of the operator $\W \mapsto \overline{\W}$ as ``closure-like'' in the introduction.

\begin{lemma}\label{lem:closure}
    Let $\W$ be a wide subcategory. Then the following hold.
\begin{enumerate}[label=\bf\arabic*., ref=\arabic*]
        \item\label{closure 1} 
        $\W \subseteq \overline{\W}$.
        \item\label{closure 2} 
        $\overline{\V} \subseteq \overline{\W}$ for any wide subcategory $\V \subseteq \W$.
        \item\label{closure 3} $\overline{\overline{\W}} = \overline{\W}$.
    \end{enumerate}
\end{lemma}

\begin{proof}
    Let $X \in \W$ and $Y \in \W^\perp$. By definition, this means $\Hom(X,Y) = 0 = \Ext^1(X,Y)$. Thus $X \in {}^\perp(\W^\perp)$. Thus \begin{equation}\label{eqn:closure 1}\W \subseteq {}^\perp(\W^\perp).\end{equation}
    Now clearly $\W \subseteq \aug(\W)$ by definition, so Assertion~\ref{closure 1} follows from applying~\eqref{eqn:closure 1} to $\aug(\W)$.

    Now let $\V \subseteq \W$. Again, it is immediate from the definition that $\aug(\V) \subseteq \aug(\W)$. Assertion~\ref{closure 2} thus follows from the fact that $(-)^\perp$ and ${}^\perp(-)$ are order-reversing.

    To prove Assertion~\ref{closure 3}, we first claim that
    \begin{equation}\label{eqn:closure 2} \aug(\overline{\W}) = \overline{\W}.
    \end{equation}
    We verify \eqref{eqn:closure 2} in two cases.

    Suppose first that $\W$ is representation-finite. Then $\aug(\W) = \W$ by definition. Moreover, $\W$ is exceptional by Lemma~\ref{lem:rep_finite_exceptional}, and so Corollary~\ref{cor:anti_isom_fg} implies that ${}^\perp(\W^\perp) = \W$. We conclude that $\aug(\W) = \W = \overline{\W}$.

    Now suppose that $\W$ is representation-infinite. Let $\U = \aug(\W)$. Applying \eqref{eqn:closure 1} to $\U$, and using the definition of $\aug(-)$, we see that $\homo \subseteq \U \subseteq {}^\perp(\U^\perp)$. Thus \[\aug(\overline{\W}) = \aug({}^\perp(\U^\perp)) = {}^\perp(\U^\perp) = \overline{\W}.\]
    We have verified \eqref{eqn:closure 2} in all cases. Returning to the case where $\W$ is arbitrary, we apply \eqref{eqn:closure 1} and \eqref{eqn:closure 2} to $\overline{\W}$ to obtain
    \[\overline{\W} = {}^\perp(\U^\perp) \subseteq {}^\perp(({}^\perp(\U^\perp))^\perp) = \overline{\overline{\W}}.\]
    Now let $X \in \overline{\overline{\W}} = {}^\perp((^\perp(\U^\perp))^\perp)$. This means $\Hom(X,Y) = 0 = \Ext^1(X,Y)$ for all $Y \in (^\perp(\U^\perp))^\perp$. Moreover, an argument symmetric to the one used to show \eqref{eqn:closure 1} implies $\U^\perp \subseteq (^\perp(\U^\perp))^\perp$. We conclude that $X \in {}^\perp(\U^\perp) = \overline{\W}$. This shows Assertion~\ref{closure 3}.
\end{proof}

\begin{definition}\label{def:para}
A wide subcategory is \newword{para-exceptional} if it is of the form $\overline{\sW(X_1,\ldots,X_k)}$ for some para-exceptional sequence $(X_1,\ldots,X_k)$.
We write $\pewide\Lambda$ for the set of para-exceptional subcategories of $\Lambda$, partially ordered by containment.
\end{definition}

We now state para-exceptional analogs of Proposition~\ref{prop:perp_prefix} and Theorem~\ref{thm:exceptional_binary_chain}.

\begin{proposition}\label{prop:perp_prefix_full}  
Let $(X_1,\ldots,X_k)$ and $(Y_1,\ldots,Y_i)$ be para-exceptional sequences. Then the following are equivalent.
\begin{enumerate}[label=\rm(\roman*), ref=(\roman*)]
\item \label{ppf 1}
$(X_1,\ldots,X_k)$ and $(Y_1,\ldots,Y_i)$ are equivalent as prefixes in $\Cpx$.
\item \label{ppf 2}
 $(X_1,\ldots,X_k)$ and $(Y_1,\ldots,Y_i)$ are equivalent as postfixes in $\Cpx$.
\item \label{ppf 3}
 $\overline{\sW(X_1,\ldots,X_k)} = \overline{\sW(Y_1,\ldots,Y_i)}$.
\item \label{ppf 4}
 $\red((X_1,\ldots,X_k)^\perp) = \red((Y_1,\ldots,Y_i)^\perp)$.
\item \label{ppf 5}
 $\red({}^\perp(X_1,\ldots,X_k)) = \red(^{\perp}(Y_1,\ldots,Y_i))$.
\end{enumerate}
\end{proposition}

\begin{theorem}\label{thm:binary_chain_full}
Let $\Lambda$ be a connected tame hereditary algebra with more than one non-homogeneous tube.
\begin{enumerate}[label=\bf\arabic*., ref=\arabic*]
\item \label{Cpx bcs}
The set $\Cpx$ of maximal para-exceptional sequences is a binary chain system with binary compatibility relation~$\Bpx$.
\item \label{Pc to pewide}
The map $\Preeq(\Cpx)\rightarrow\pewide\Lambda$ that sends $P \in \Preeq(\Cpx)$ to $\overline{\sW(X_1,\ldots,X_k)}$, for any representative $(X_1,\ldots,X_k) \in P$, is an isomorphism of posets.
\item \label{Cpx labels cov}
If $\V,\W\in\pewide\Lambda$ have $\W\subseteq\V$, then $\V$ covers $\W$ in $\pewide\Lambda$ if and only if $\V\cap{}^\perp\W$ contains a unique non-homogeneous brick.
\item \label{Cpx labels}
If $P\covered P'$ in $\Preeq(\Cpx)$ is labeled by $X$ and sent by the  isomorphism in Assertion~\ref{Pc to pewide} to $\W \covered \V$, then $X$ is the unique non-homogeneous brick in $\V \cap {}^\perp \W$.
\end{enumerate}
\end{theorem}

Before proving Proposition~\ref{prop:perp_prefix_full} and Theorem~\ref{thm:binary_chain_full}, we mention a corollary that combines Theorem~\ref{thm:binary_chain_full} with Theorem~\ref{thm:bijection_on_chain_systems}.

The bijection $\omega$ from $\E\cup\nrig$ to $T_V\cup T_H\cup F$ given by 
\[\omega(X) = \begin{cases} t_{\undim X} & \text{if $X \in \excep$}\\f_\gamma & \text{if $X = F_\gamma\in \nrig$}\end{cases}\]
has an obvious inverse mapping $f_\gamma\in F$ to $F_\gamma\in\nrig$ and $t\in T_V\cup T_H$ to $X_t$. 
We use $\omega^{-1}$ to define a map $\pi$ from the McCammond-Sulway lattice to the set $\pewide\Lambda$ of para-exceptional subcategories. 
Given an element $x\in[1,c]_{T\cup F}$, choose a reduced word $a_1\cdots a_k$ for $x$ in the alphabet $T\cup F$ and define $\pi(x)$ to be the para-exceptional subgategory $\overline{\sW(\omega^{-1}(a_1),\ldots,\omega^{-1}(a_k))}$.
Combining  Theorem~\ref{thm:bijection_on_chain_systems}.\ref{PCpx McSul} with Theorem~\ref{thm:binary_chain_full}.\ref{Pc to pewide}, we obtain the following corollary.

\begin{corollary}\label{McSul to pewide}
Let $\Lambda$ be a connected tame hereditary algebra with more than one non-homogeneous tube.
The map $\pi$ is an isomorphism from the McCammond-Sulway lattice $[1,c]_{T\cup F}$ to the poset $\pewide\Lambda$ of para-exceptional subcategories.
\end{corollary}

We begin preparing for the proofs of Proposition~\ref{prop:perp_prefix_full} and Theorem~\ref{thm:binary_chain_full} by considering representation-infinite and representation-finite subcategories of $\mods\Lambda$.
We quote two lemmas and prove two more lemmas and a proposition.  
The first lemma is \cite[Lemma~4.5]{IPT}.
(The statement in \cite{IPT} assumes that $\field$ is algebraically closed, but the proof given there is valid without that assumption.)

\begin{lemma}\label{lem:finite_type}  
Let $X$ be an exceptional module. If $X^{\perp}$ is representation-infinite, then $X$ is regular and $X^\perp \not \subseteq \reg$.
\end{lemma}

For the next lemma, see \cite[Theorem~3.2.15]{dichev_thesis}, \cite[Theorem~9.8]{IPT}, or \cite[Proposition~6.14]{kohler}.

\begin{lemma}\label{lem:non_finite_reg} 
Let $\W$ be a non-exceptional wide subcategory. Then $\W \subseteq \reg$.
\end{lemma}

For the next lemma, see also the proof of \cite[Lemma~4.5]{IPT}.

\begin{lemma}\label{lem:homogeneous_stable}  
 Let $\W \in \ewide \Lambda$. If $\W$ is representation-infinite, then $\homo \subseteq \W$.
\end{lemma}

\begin{proof}
Since $\W$ is exceptional, Proposition~\ref{prop:wide_fg} implies that there is an exceptional sequence $(X_1,\ldots,X_k)$ with $\W = (X_1,\ldots,X_k)^\perp$. 
Since $\W$ is representation-infinite and $(X_1,\ldots,X_k)^\perp=\bigcap_{i=1}^kX_i^\perp$, each $X_i$ is regular by Lemma~\ref{lem:finite_type}.
Thus $(X_1,\ldots,X_k)$ is a regular para-exceptional sequence. 
Lemma~\ref{lem:tubes_ortho} completes the proof. 
\end{proof}

\begin{lemma}\label{lem:perp_in_tubes}  
Let $\W$ be an representation-infinite exceptional subcategory. 
Then $\W^\perp = (\exreg \cap \W)^\perp \cap \exreg$ and ${}^\perp\W = {}^\perp(\exreg\cap \W) \cap \exreg$.
\end{lemma}

\begin{proof}
We prove the identity for $\W^\perp$; the proof for ${}^\perp\W$ is similar. 
By Proposition~\ref{prop:wide_fg}, there is an exceptional sequence $(X_1,\ldots,X_k)$ with $\W = {}^\perp(X_1,\ldots,X_k)$. Corollary~\ref{cor:anti_isom_fg} (and Lemma~\ref{lem:ortho_wide}) then implies that $\W^\perp = \sW(X_1,\ldots,X_k)$.

Lemma~\ref{lem:homogeneous_stable} says that $\W$ contains all of the homogeneous tubes. 
By Lemma~\ref{lem:homogeneous}, the modules in these tubes are non-exceptional, so $\W^\perp \subseteq \exreg$ by Lemma~\ref{lem:hom_ext}.
Therefore $(X_1,\ldots,X_k)$ is a regular para-exceptional sequence.

By Lemma~\ref{lem:soft_index_arbitrary}, we extend $(X_1,\ldots,X_k)$ to a maximal regular para-exceptional sequence $(X_1,\ldots,X_k,Y_1,\ldots,Y_j)$. 
Proposition~\ref{prop:maximal_soft} 
implies that 
\[\sW(Y_1,\ldots,Y_j) = {}^\perp(X_1,\ldots,X_k) \cap \exreg = \W \cap \exreg.\]
Again applying Proposition~\ref{prop:maximal_soft} (and Lemma~\ref{lem:ortho_wide}), we conclude that 
\[(\exreg \cap \W)^\perp \cap \exreg = (Y_1,\ldots,Y_j)^\perp \cap \exreg = \sW(X_1,\ldots,X_k) = \W^\perp.\qedhere\]
\end{proof}

\begin{proposition}\label{prop:unique_intersection_with_wide} 
A representation-infinite exceptional subcategory is determined by its intersection with the non-homogeneous tubes.
More specifically, two representation-infinite exceptional subcategories $\W$ and $\V$ are equal if and only if $\W \cap \reg = \V \cap \reg$ if and only if $\W \cap \exreg = \V \cap \exreg$.
\end{proposition}

\begin{proof}
One direction of implications is clear.
Suppose $\W \cap \exreg = \V \cap \exreg$.
Then Corollary~\ref{cor:anti_isom_fg} and Lemma~\ref{lem:perp_in_tubes} imply that
\[
\W = {}^\perp(\W^\perp) = {}^\perp((\exreg \cap \W)^\perp \cap \exreg)
={}^\perp((\exreg \cap \V)^\perp \cap \exreg)={}^\perp(\V^\perp) =\V. \qedhere
\]
\end{proof}

We continue preparing for the proofs of Proposition~\ref{prop:perp_prefix_full} and Theorem~\ref{thm:binary_chain_full} by establishing properties of and formulas for the constructions in Definition~\ref{def:E_perp}.
The first property is immediate from Proposition~\ref{prop:tubes_ortho}.

\begin{lemma}\label{lem:join oplus}
If $\W$ is a representation-infinite wide subcategory and $\W\subseteq\nonhomo$, then $\aug(\W)=\W\oplus\homo$.
\end{lemma}

\begin{lemma}\label{lem:exp_compute} 
If $(X_1,\ldots,X_k)$ is a para-exceptional sequence such that $\W = \sW(X_1,\ldots,X_k)$, then
\[\aug(\W) = \begin{cases}
\W & \text{if $(X_1,\ldots,X_k)$ is an exceptional sequence}\\
\W \oplus \homo & \text{otherwise.}
\end{cases}\]
\end{lemma}

\begin{proof}
Suppose $(X_1,\ldots,X_k)$ is an exceptional sequence. 
If $\W$ is representation-finite, then $\aug(\W)=\W$.
If $\W$ is representation-infinite, then $\homo \subseteq \W$ by Lemma~\ref{lem:homogeneous_stable}, and so $\aug(\W) = \W$.
Now suppose that $(X_1,\ldots,X_k)$ is not an exceptional sequence. 
Then $\W \subseteq \nonhomo$ by Lemma~\ref{lem:hom_ext}, so $\aug(\W) = \W \oplus \homo$ by Proposition~\ref{prop:tube_generation_2}.\ref{tg2 1} and Lemma~\ref{lem:join oplus}.
\end{proof}

\begin{lemma}\label{lem:reduce_fg}  
Let $\W \in \ewide \Lambda$. Then $\red(\W) = \W$.
\end{lemma}

\begin{proof}
If $\W \subseteq \reg$, then $\W \subseteq \nonhomo$ by Proposition~\ref{prop:wide_tube} and Proposition~\ref{prop:tubes_ortho}.
\end{proof}

\begin{lemma}\label{lem:double_red} 
If $\W$ and $\V$ are wide subcategories, then 
\[\red(\W \cap \V) = \red(\red(\W) \cap \red(\V)).\]
\end{lemma}

\begin{proof}
The lemma is trivial when $\red(\W)=\W$ and $\red(\V)=\V$.
Otherwise, without loss of generality $\W \subseteq \reg$ and $\W \cap \nonhomo$ is representation-finite. 
Then also $\W \cap \V\subseteq\R$ and $\red(\W) \cap \red(\V)$ is representation-finite. 
Thus 
\[\red(\W \cap \V) = \W \cap \V \cap \nonhomo = \red(\red(\W) \cap \red(\V)). \qedhere\] 
\end{proof}

\begin{lemma}\label{lem:overline_compute} 
Suppose $(X_1,\ldots,X_k)$ is a para-exceptional sequence and suppose ${\W = \sW(X_1,\ldots,X_k)}$. 
\begin{enumerate}[label=\bf\arabic*., ref=\arabic*, leftmargin=23pt]
\item \label{oc ne}
If $(X_1,\ldots,X_k)$ is non-exceptional, then 
\begin{enumerate}[label=\bf\alph*., ref=\alph*, leftmargin=16pt]
\item \label{oc ne 1}
${(\aug(\W))^\perp = (\W \oplus \homo)^\perp = \W^\perp \cap \nonhomo}$.
\item \label{oc ne 2}
${}^{\perp}(\aug(\W)) = {}^\perp(\W \oplus \homo) = {}^{\perp}\W \cap \nonhomo$.
\item \label{oc ne 3}
$\W \subseteq \nonhomo$.
\end{enumerate}
\item \label{oc fne}
If $(X_1,\ldots,X_k)$ is fully non-exceptional, then
\begin{enumerate}[label=\bf\alph*., ref=\alph*, leftmargin=16pt]
\item \label{oc fne 1}
$(\aug(\W))^\perp$ is exceptional, representation-finite, and contained in $\nonhomo$.
\item \label{oc fne 2}
${}^{\perp}(\aug(\W))$ is exceptional, representation-finite, and contained in $\nonhomo$.
\item \label{oc fne 3}
$\overline{\W}$ is exceptional and representation-infinite.
\item \label{oc fne 4}
$\overline{\W}^\perp = (\aug(\W))^\perp$ and ${}^\perp\overline{\W} = {}^\perp(\aug(\W))$.
\item \label{oc fne 5}
$\W = \overline{\W} \cap \nonhomo \subsetneq \overline{\W}$.
\item \label{oc fne 6}
$\overline{\W} = ({}^\perp(\aug(\W))^\perp$.
\end{enumerate}
\item \label{oc nfne}
If $(X_1,\ldots,X_k)$ is non-exceptional, but not fully non-exceptional, then
\begin{enumerate}[label=\bf\alph*., ref=\alph*, leftmargin=16pt]
\item \label{oc nfne 1}
There exists a non-homogeneous tube $\T$ such that $(\aug(\W))^\perp \cap \T$ and ${}^\perp(\aug(\W))\cap \T$ are representation-finite.
\item \label{oc nfne 2}
There exists a non-homogeneous tube $\T$ such that $(\aug(\W))^\perp \cap \T$ and ${}^\perp(\aug(\W))\cap \T$ are representation-infinite.
\item \label{oc nfne 3}
$\W = \overline{\W} \cap \nonhomo \subsetneq \overline{\W} = \aug(\W) = \W \oplus \homo$.
\item \label{oc nfne 4}
$\overline{\W} = ({}^\perp(\aug(\W))^\perp.$
\end{enumerate}
\item \label{oc e}
If $(X_1,\ldots,X_k)$ is exceptional, then $\overline{\W} = \W = ({}^\perp(\aug(\W))^\perp$.
\end{enumerate}
\end{lemma}

\begin{proof}
Suppose $(X_1,\ldots,X_k)$ is non-exceptional.
Then $\aug(\W) = \W \oplus \homo$ by Lemma~\ref{lem:exp_compute}, so $(\aug(\W))^\perp = \W^\perp \cap \homo^\perp$. 
Now $\homo^\perp = \nonhomo$ by Lemmas~\ref{lem:hom_ext} and~\ref{lem:tubes_ortho}.
This is Assertion~\ref{oc ne}\ref{oc ne 1}, and Assertion~\ref{oc ne}\ref{oc ne 2} is symmetric.
Lemma~\ref{lem:hom_ext} implies that each $X_i$ is regular, so $\W \subseteq \nonhomo$.
This is Assertion~\ref{oc ne}\ref{oc ne 3}.

Suppose $(X_1,\ldots,X_k)$ is fully non-exceptional.
By Assertion~\ref{oc ne}\ref{oc ne 1}, $(\aug(\W))^\perp = \W^\perp \cap \nonhomo\subseteq\nonhomo$.
By Lemma~\ref{lem:ortho_wide}, $\W^\perp = (X_1,\ldots,X_k)^\perp$, so by Proposition~\ref{prop:tube_generation_2}.\ref{tg2 2}, $(\aug(\W))^\perp$ is exceptional and representation-finite. 
This is Assertion~\ref{oc fne}\ref{oc fne 1}, and Assertion~\ref{oc fne}\ref{oc fne 2} is symmetric.
By Corollary~\ref{cor:anti_isom_fg}, $\overline{\W} = {}^\perp((\aug(\W))^\perp)$ is also exceptional. 
Since $\W\subseteq\aug(\W)$, Proposition~\ref{prop:tube_anti_isom} implies that $\W \subseteq \overline{\W}$. 
By Proposition~\ref{prop:tube_generation_2}.\ref{tg2 1}, $\W$ is representation-infinite, so $\overline{\W}$ is also representation-infinite.
This is Assertion~\ref{oc fne}\ref{oc fne 3}.
Because ${\overline{\W} = {}^\perp((\aug(\W))^\perp)}$ is exceptional, Assertion~\ref{oc fne}\ref{oc fne 4} follows by Corollary~\ref{cor:anti_isom_fg}.
Since $\W \subseteq \nonhomo$ and $\overline{\W}$ is exceptional and representation-infinite, Assertions~\ref{oc ne}\ref{oc ne 1} and~\ref{oc fne}\ref{oc fne 4} combine with Lemma~\ref{lem:perp_in_tubes} to yield
\[\W^\perp \cap \nonhomo = \overline{\W}^\perp = (\nonhomo \cap \overline{\W})^\perp \cap \nonhomo.\]
Proposition~\ref{prop:tube_anti_isom}.\ref{tai 1} then implies that $\W = \nonhomo \cap \overline{\W}$.
The inclusion $\W \subseteq \overline{\W}$ is proper because $\overline{\W}$ is exceptional and, by Proposition~\ref{prop:tube_generation_2}.\ref{tg2 1}, $\W$ is not exceptional.
This is Assertion~\ref{oc fne}\ref{oc fne 5}.
The wide subcategories $\overline{\W}$ and $({}^\perp(\aug(\W))^\perp$ are related by the symmetry of exchanging left and right duals.
Thus the symmetric version of the argument above shows that $({}^\perp(\aug(\W))^\perp$ is exceptional, representation-infinite, and satisfies $\W = ({}^\perp(\aug(\W))^\perp \cap \nonhomo$. 
Since these properties also hold for $\overline{\W}$, Proposition~\ref{prop:unique_intersection_with_wide} implies that $\overline{\W} = ({}^\perp(\aug(\W))^\perp$.
This is Assertion~\ref{oc fne}\ref{oc fne 6}.

Assertion~\ref{oc ne} and Lemma~\ref{lem:ortho_wide} say that $(\aug(\W))^\perp = (X_1,\ldots,X_k)^\perp \cap \nonhomo$ and ${}^\perp(\aug(\W)) = {}^\perp(X_1,\ldots,X_k) \cap \nonhomo$. 
Since $(X_1,\ldots,X_k)$ is not exceptional, Assertion~\ref{oc nfne}\ref{oc nfne 1} follows by Proposition~\ref{prop:tube_generation_2}.\ref{tg2 1}.
Since $(X_1,\ldots,X_k)$ is not fully non-exceptional, Assertion~\ref{oc nfne}\ref{oc nfne 2} follows by Proposition~\ref{prop:tube_generation_2}.\ref{tg2 2}.
 Now by Assertion \ref{oc nfne}\ref{oc nfne 2} and Proposition~\ref{prop:wide_tube}, there exists a non-exceptional indecomposable $Y \in (\aug(\W))^\perp$. Thus $\overline{\W} \subseteq \reg$ by Lemma~\ref{lem:hom_ext}. Also, by Assertion~\ref{oc ne}, $(\aug(\W))^\perp = (\W \oplus \homo)^\perp\subseteq\reg$, so applying ${}^\perp$ to the left, we obtain $\overline{\W}={}^\perp((\W \oplus \homo)^\perp) = {}^\perp((\W \oplus \homo)^\perp\cap\reg)\cap\reg$.
By Proposition~\ref{prop:tube_anti_isom}.\ref{tai 2}, $\overline{\W}=\W \oplus \homo$.
In particular, $\W\subsetneq\overline{\W}$.
Applying the same argument to the definition $\overline{\W}={}^\perp((\aug(\W))^\perp)$, we obtain $\overline{\W}=\aug(\W)$.
Intersecting both sides of $\overline{\W}=\W \oplus \homo$ with $\nonhomo$, we obtain
$\overline{\W}\cap\nonhomo=(\W \oplus \homo)\cap\nonhomo=\W$ by Assertion~\ref{oc ne}.\ref{oc ne 3}.
This is Assertion~\ref{oc nfne}\ref{oc nfne 3}.
An analogous argument shows that $({}^\perp(\aug(\W))^\perp = \W \oplus \homo$, and Assertion~\ref{oc nfne}\ref{oc nfne 4} follows.
Assertion~\ref{oc e} follows from Lemma~\ref{lem:exp_compute} and Corollary~\ref{cor:anti_isom_fg}.
\end{proof}

\begin{lemma}\label{lem:perp_formula}
Suppose $\W = \sW(X_1,\ldots,X_k)$ for some para-exceptional sequence $(X_1,\ldots,X_k)$.
Then $\red(\W^\perp) = \aug(\overline{\W}^\perp)$ and $\red({}^\perp\W) = \aug({}^\perp\overline{\W}).$
\end{lemma}

\begin{proof}
If $(X_1,\ldots,X_k)$ is an exceptional sequence, then $\aug(\overline{\W}^\perp)=\aug(\W^\perp)$ by Lemma~\ref{lem:overline_compute}.\ref{oc e}, which equals $\W^\perp $ by Corollary~\ref{cor:anti_isom_fg} and Lemma~\ref{lem:exp_compute}, which further equals $\red(\W^\perp)$ by Lemma~\ref{lem:reduce_fg}.

If $(X_1,\ldots,X_k)$ is fully non-exceptional, Lemma~\ref{lem:overline_compute}.\ref{oc fne}\ref{oc fne 4} and Lemma~\ref{lem:overline_compute}.\ref{oc ne}\ref{oc ne 1} say that $\overline{\W}^\perp = (\aug(\W))^\perp = \W^\perp \cap \nonhomo$.
This subcategory is representation-finite by Lemma~\ref{lem:overline_compute}.\ref{oc fne}\ref{oc fne 1}, so $\overline{\W}^\perp = \aug(\overline{\W}^\perp)$. 
Moreover, $\W^\perp \subseteq \reg$ by Lemma~\ref{lem:hom_ext}, so $\W^\perp \cap \nonhomo = \red(\W^\perp)$. 
Thus $\aug(\overline{\W}^\perp) = \overline{\W}^\perp = \W^\perp \cap \nonhomo = \red(\W^\perp)$.

If $(X_1,\ldots,X_k)$ is non-exceptional, but not fully non-exceptional, Lemma~\ref{lem:overline_compute}.\ref{oc nfne}\ref{oc nfne 3} says
$\overline{\W} = \aug(\W) = \W \oplus \homo$. 
Lemma~\ref{lem:overline_compute}.\ref{oc ne}\ref{oc ne 1} then implies that $\overline{\W}^\perp = \W^\perp \cap \nonhomo$. 
This subcategory is representation-infinite by Lemma~\ref{lem:overline_compute}.\ref{oc nfne}\ref{oc nfne 2}, so $\aug(\overline{\W}^\perp) = \aug(\W^\perp \cap \nonhomo) = (\W^\perp \cap \nonhomo) \oplus \homo = \W^\perp$ by Lemmas~\ref{lem:join oplus}, \ref{lem:tubes_ortho}, and~\ref{lem:hom_ext}.
Also because $\W^\perp \cap \nonhomo$ is representation-infinite, ${\red(W^\perp) = \W^\perp}$. 
\end{proof}

We now use Lemma~\ref{lem:overline_compute} to prove an analog of Propositions~\ref{prop:maximal_exceptional} and~\ref{prop:maximal_soft} for para-exceptional sequences.

\begin{proposition}\label{prop:maximal_characterization}
Let $(X_1,\ldots,X_k)$ be a para-exceptional sequence. 
Then each of the following conditions is equivalent to $(X_1,\ldots,X_k)$ being maximal.
\begin{enumerate}[label=\rm(\roman*), ref=(\roman*)]
\item $\overline{\sW(X_{i+1},\ldots,X_j)} = \red({}^{\perp}(X_1,\ldots,X_i) \cap (X_{j+1},\ldots,X_k)^{\perp})$ for all $i,j$ with $0\le i\le j\le k$.
\item $\overline{\sW(X_{i+1},\ldots,X_j)} = \red({}^{\perp}(X_1,\ldots,X_i) \cap (X_{j+1},\ldots,X_k)^{\perp})$ for some $i,j$ with $0\le i\le j\le k$.
\end{enumerate}
\end{proposition}

\begin{proof}
Suppose  $\W = \sW(X_{i+1},\ldots,X_k)$ and $\V = {}^{\perp}(X_1,\ldots,X_k) \cap (X_{j+1},\ldots,X_k)^{\perp}$.
We show that $(X_1,\ldots,X_k)$ is maximal if and only if $\overline{\W}=\red(\V)$.
If $(X_1,\ldots,X_k)$ is exceptional, then Proposition~\ref{prop:chain_system_union}.\ref{Cpx Cx} says that $(X_1,\ldots,X_k)$ is maximal as a para-exceptional sequence if and only if it is maximal as an exceptional sequence. 
Since $\V$ is exceptional by Corollary~\ref{cor:middle_exceptional}, $\red(\V) = \V$ by Lemma~\ref{lem:reduce_fg}.
Moreover, $\overline{\W} = \W$ by Lemma~\ref{lem:overline_compute}.\ref{oc e}, so the result follows from Proposition~\ref{prop:maximal_exceptional}.

Suppose from now on that $(X_1,\ldots,X_k)$ is non-exceptional, so that in particular, $(X_1,\ldots,X_k)$ is a regular para-exceptional sequence by Lemma~\ref{lem:hom_ext}. The existence of a non-exceptional para-exceptional sequence in particular implies $\N \neq \emptyset$.
Proposition~\ref{prop:chain_system_union}.\ref{Cpx Crpx} then says that $(X_1,\ldots,X_k)$ is maximal as a para-exceptional sequence if and only if it is maximal as a regular para-exceptional sequence.
We break into cases based on the status of $(X_{i+1},\ldots,X_j)$ and $(X_1,\ldots,X_i,X_{j+1},\ldots,X_k)$ as exceptional, non-exceptional but not fully non-exceptional, or fully non-exceptional.
Of the nine possible cases, we rule out four:
At least one of the sequences is non-exceptional and if one of the sequences is fully non-exceptional, then Proposition~\ref{prop:hom_ext_vanish} implies that the other sequence is exceptional.

\smallskip\noindent
\textbf{Case 1.}
$(X_{i+1},\ldots,X_j)$ is fully non-exceptional and $(X_1,\ldots,X_i,X_{j+1},\ldots,X_k)$ is exceptional.
Lemma~\ref{lem:reduce_fg} says that $\red(\V) = \V$. 
Moreover, Lemma~\ref{lem:overline_compute}.\ref{oc fne}\ref{oc fne 2} says that $\overline{\W}$ is representation-infinite and $\W = \overline{\W} \cap \nonhomo$. 
Since $\W \subseteq \V$, also $\red(\V) = \V$ is representation-infinite. 
Proposition~\ref{prop:unique_intersection_with_wide} then implies that $\red(\V) = \V = \overline{\W}$ if and only if ${\V \cap \nonhomo = \overline{\W} \cap \nonhomo = \W}$. 
By Proposition~\ref{prop:maximal_soft}, also $(X_1,\ldots,X_k)$ is maximal if and only if ${\V \cap \nonhomo = \W}$.

\smallskip\noindent
\textbf{Case 2.}
$(X_{i+1},\ldots,X_j)$ is non-exceptional, but not fully non-exceptional and $(X_1,\ldots,X_i,X_{j+1},\ldots,X_k)$ is exceptional.
By Lemma~\ref{lem:overline_compute}.\ref{oc nfne} and Lemma~\ref{lem:exp_compute}, $\overline{\W} = \aug(\W) = \W \oplus \homo$ is representation-infinite and contained in $\reg$. 
Moreover,~$\overline{\W}$ is not exceptional by Proposition~\ref{prop:tube_generation_2}.\ref{tg2 1}.
Proposition~\ref{prop:chain_system_union}.\ref{Cpx Cx} says that $(X_1,\ldots,X_k)$ is not maximal.
Again Lemma~\ref{lem:reduce_fg} says that $\red(\V) = \V$ and this is an exceptional subcategory.
Since $\overline{\W}$ is not exceptional, $\overline{\W} \neq \red(\V)$.

\smallskip\noindent
\textbf{Case 3.}
$(X_{i+1},\ldots,X_j)$ is non-exceptional, but not fully non-exceptional and $(X_1,\ldots,X_i,X_{j+1},\ldots,X_k)$ is non-exceptional, but not fully non-exceptional.
Proposition~\ref{prop:tube_generation_2}.\ref{tg2 1} says that $\W$ is representation-infinite. 
Since $\W\subseteq\V$, also $\V$ is representation-infinite, so $\red(\V) = \V$.
As in Case 2, $\overline{\W} = \aug(\W) = \W \oplus \homo$ is representation-infinite and non-exceptional.
Since $(X_1,\ldots,X_i,X_{j+1},\ldots,X_k)$ is non-exceptional, Lemma~\ref{lem:hom_ext} implies that $\V\subseteq\reg$.  
Also by Lemma~\ref{lem:hom_ext}, each module in the sequence $(X_1,\ldots,X_k)$ is in $\nonhomo$, so $\homo\subseteq\V$ by Lemma~\ref{lem:tubes_ortho}.
We conclude that $\V = (\V \cap \nonhomo) \oplus \homo$, see Proposition~\ref{prop:tubes_ortho}.
Since $\overline{\W} = \W \oplus \homo$ and $\V = (\V \cap \nonhomo) \oplus \homo$, it follows that $\overline{\W} = \red(\V) = \V$ if and only if $\W = \V \cap \nonhomo$. 
Finally, we have that $\W = \V \cap \nonhomo$ if and only if $(X_1,\ldots,X_k)$ is maximal by Proposition~\ref{prop:maximal_soft}.

\smallskip\noindent
\textbf{Case 4.}  
$(X_{i+1},\ldots,X_j)$ is exceptional and $(X_1,\ldots,X_i,X_{j+1},\ldots,X_k)$ is non-exceptional, but not fully non-exceptional. Proposition~\ref{prop:chain_system_union}.\ref{Cpx Crpx} says that $(X_1,\ldots,X_k)$ is not maximal. Now as in Case 3, we have that $\homo \subseteq \red(\V) = \V \subseteq \reg$. Thus $\V$ is non-exceptional by Propositions~\ref{prop:tubes_ortho} and~\ref{prop:wide_tube}. Moreover, since $(X_{i+1},\ldots,X_j)$ is exceptional, we have that $\overline{\W} = \W$ by Lemma~\ref{lem:overline_compute}.\ref{oc e}. This is an exceptional subcategory, and so $\overline{\W} \neq \V$.

\smallskip\noindent
\textbf{Case 5.}  
$(X_{i+1},\ldots,X_j)$ is exceptional and $(X_1,\ldots,X_i,X_{j+1},\ldots,X_k)$ is fully non-exceptional. Since $(X_1,\ldots,X_i,X_{j+1},\ldots,X_k)$ is non-exceptional, ${\V \subseteq \reg}$ by Lemma~\ref{lem:hom_ext}. Moreover, because $(X_1,\ldots,X_i,X_{j+1},\ldots,X_k)$ is fully non-exceptional, Propositions~\ref{prop:wide_tube} and~\ref{prop:tubes_ortho} imply that $\V \cap \nonhomo$ is representation-finite. Thus $\red(\V) = \V \cap \nonhomo$. On the other hand, as in Case 4, we have that $\overline{\W} = \W$. Now the fact that $\W = \V \cap \nonhomo$ if and only if $(X_1,\ldots,X_k)$ is maximal follows from Proposition~\ref{prop:maximal_soft}. 
\end{proof}

We now give several characterizations of para-exceptional subcategories.

\begin{proposition}\label{prop:reduced_characterization}
Let $\W$ be a wide subcategory.  
Then the following are equivalent.
\begin{enumerate}[label=\rm(\roman*), ref=(\roman*)]
\item \label{rc 1}
$\W \in \pewide\Lambda$.
\item \label{rc 2}
$\W=\red((X_1,\ldots,X_k)^{\perp})$ for some para-exceptional sequence $(X_1,\ldots,X_k)$.
\item \label{rc 3}
$\W=\red({}^\perp(X_1,\ldots,X_k))$ for some para-exceptional sequence $(X_1,\ldots,X_k)$.
\item \label{rc 4}
Either $\W$ is exceptional or the following all hold.
\begin{enumerate}[label=\rm(\alph*), ref=(\alph*)]
\item $\W \subseteq \reg$.
\item $\W$ contains all of the homogeneous tubes.
\item There is a non-homogeneous tube $\T$ such that $\W \cap \T$ is exceptional.  
\item There is a non-homogeneous tube $\S$ such that $\W \cap \S$ is not exceptional.
\end{enumerate}
 \item \label{rc 5}
$\W = (X_1,\ldots,X_k)^\perp$ for some para-exceptional sequence $(X_1,\ldots,X_k)$ that is not fully non-exceptional
 \item \label{rc 6}
$\W = {}^\perp(X_1,\ldots,X_k)$ for some para-exceptional sequence $(X_1,\ldots,X_k)$ that is not fully non-exceptional
\end{enumerate}
\end{proposition}

\begin{proof}
We prove that \ref{rc 1}$\implies$\ref{rc 4}$\implies$\ref{rc 5}$\implies$\ref{rc 2}$\implies$\ref{rc 1}.
The proof that \mbox{\ref{rc 4}$\implies$\ref{rc 6}$\implies$\ref{rc 3}$\implies$\ref{rc 1}} is symmetric.

Suppose 
$\W = \overline{\sW(X_1,\ldots,X_k)}$ for some para-exceptional sequence $(X_1,\ldots,X_k)$.
If $(X_1,\ldots,X_k)$ is exceptional or fully non-exceptional, then $\W$ is exceptional by Lemma~\ref{lem:overline_compute}.\ref{oc e} or Lemma~\ref{lem:overline_compute}.\ref{oc fne}\ref{oc fne 3}. 
If $(X_1,\ldots,X_k)$ is non-exceptional, but not fully non-exceptional, then Lemma~\ref{lem:overline_compute}.\ref{oc nfne}\ref{oc nfne 3} and Proposition~\ref{prop:tube_generation_2} imply that $\W$ satisfies properties (a)--(d) in condition~\ref{rc 4}.
We see that \ref{rc 1}$\implies$\ref{rc 4}.

Now suppose \ref{rc 4}.
If $\W$ is exceptional \ref{rc 5} follows from Proposition~\ref{prop:wide_fg}.
Suppose that $\W$ is not exceptional. 
Applying Proposition~\ref{prop:tube_generation} to $\W \cap \nonhomo$, there exists a regular para-exceptional sequence $(X_1,\ldots,X_k)$ with $\W \cap \nonhomo = (X_1,\ldots,X_k)^\perp \cap \nonhomo$. 
Proposition~\ref{prop:tube_generation_2} implies that $(X_1,\ldots,X_k)$ is non-exceptional, but not fully non-exceptional.
Since this sequence is non-exceptional, Lemmas~\ref{lem:hom_ext} and~\ref{lem:tubes_ortho} imply that 
$(X_1,\ldots,X_k)^\perp\subseteq\reg$ and $(X_1,\ldots,X_k)^\perp$ contains all the homogenous tubes.
Thus $(X_1,\ldots,X_k)^\perp = ((X_1,\ldots,X_k)^\perp \cap \nonhomo) \oplus \homo$, see Proposition~\ref{prop:tubes_ortho}.
Likewise ${\W=(\W\cap\nonhomo)\oplus\homo}$, and therefore $\W = (X_1,\ldots,X_k)^\perp$.
Thus \ref{rc 4}$\implies$\ref{rc 5}.

Suppose $\W = (X_1,\ldots,X_k)^\perp$ for a para-exceptional sequence $(X_1,\ldots,X_k)$ that is not fully non-exceptional.
If $(X_1,\ldots,X_k)$ is exceptional, then Proposition~\ref{prop:wide_fg} says that $(X_1,\ldots,X_k)^\perp$ is exceptional, so $\red((X_1,\ldots,X_k)^\perp) = (X_1,\ldots,X_k)^\perp=\W$ by Lemma~\ref{lem:reduce_fg}.
If $(X_1,\ldots,X_k)$ is non-exceptional, then each $X_i$ is regular by Lemma~\ref{lem:hom_ext}. Moreover, since $(X_1,\ldots,X_k)$ is not fully non-exceptional, Proposition~\ref{prop:tube_generation_2}.\ref{tg2 2} implies that $(X_1,\ldots,X_k)^\perp \cap \nonhomo$ is representation-infinite. Thus $\W = (X_1,\ldots,X_k)^\perp = \red((X_1,\ldots,X_k)^\perp)$. 
We see that \ref{rc 5}$\implies$\ref{rc 2}.

Finally, suppose $\W = \red((X_1,\ldots,X_k)^{\perp})$ for some para-exceptional sequence  $(X_1,\ldots,X_k)$.
By Proposition~\ref{prop:chain_system_union}.\ref{index_arbitrary_full}, there is a maximal para-exceptional sequence $(Y_1,\ldots,Y_i,X_1,\ldots,X_k)$.
Now $\W = \overline{\sW(Y_1,\ldots,Y_i)}$ by Proposition~\ref{prop:maximal_characterization}.
Thus \ref{rc 2}$\implies$\ref{rc 1}.
\end{proof}

We pause to mention a pair of useful corollaries.
The following result is immediate from Proposition~\ref{prop:reduced_characterization}.

\begin{corollary}\label{excep subs}
$\ewide\Lambda\subseteq\pewide\Lambda$, with strict containment if and only if $\Lambda$ has more than one non-homogenous tube.
\end{corollary}

\begin{corollary}\label{para reduced}
Let $\W \in \pewide\Lambda$. Then $\red(\W) = \W$.
\end{corollary}

\begin{proof}
	By Proposition~\ref{prop:reduced_characterization}, there exists a wide subcategory $\U$ such that $\W = \red(\U)$. It is then clear from the definition of $\red(-)$ that $\red(\W) = \red(\red(\U)) = \red(\U) = \W$.
\end{proof}

We now prove the main results of the section.

\begin{proof}[Proof of Proposition~\ref{prop:perp_prefix_full}]
We prove that \ref{ppf 1}$\implies$\ref{ppf 3}$\implies$\ref{ppf 5}$\implies$\ref{ppf 1}. 
The proof that \ref{ppf 2}$\implies$\ref{ppf 3}$\implies$\ref{ppf 4}$\implies$\ref{ppf 2} is symmetric.

Suppose \ref{ppf 1} holds, and specifically suppose that both $(X_1,\ldots,X_k,Z_1,\ldots,Z_j)$ and $(Y_1,\ldots,Y_i,Z_1,\ldots,Z_j)$ are maximal para-exceptional sequences.
It follows that $\overline{\sW(X_1,\ldots,X_k)} = \red((Z_1,\ldots,Z_j)^{\perp}) = \overline{\sW(Y_1,\ldots,Y_i)}$, by Proposition~\ref{prop:maximal_characterization}. Thus \ref{ppf 1}$\implies$\ref{ppf 3}.

Suppose \ref{ppf 3} holds.
Then $\aug({}^\perp\overline{\sW(X_1,\ldots,X_k)}) = \aug({}^\perp\overline{\sW(Y_1,\ldots,Y_i)})$, so Lemmas~\ref{lem:perp_formula} and~\ref{lem:ortho_wide} combine to imply~\ref{ppf 5}.

Suppose \ref{ppf 5} holds. 
By Proposition~\ref{prop:chain_system_union}.\ref{index_arbitrary_full}, there is a maximal para-exceptional sequence $(X_1,\ldots,X_k,Z_1,\ldots,Z_j)$.
By definition, \ref{ppf 5}, and Proposition~\ref{prop:maximal_characterization}, 
\[{}^\perp(Y_1,\ldots,Y_j)\supseteq\red({}^\perp(Y_1,\ldots,Y_j))= \red({}^\perp(X_1,\ldots,X_j))=\overline{\sW(Z_1,\ldots,Z_j)}.\]
Since $\overline{\sW(Z_1,\ldots,Z_j)}\supseteq \sW(Z_1,\ldots,Z_j)$ by Lemma~\ref{lem:closure}.\ref{closure 1}, $(Y_1,\ldots,Y_i,Z_1,\ldots,Z_j)$ is also para-exceptional. 
By \ref{ppf 5} and Proposition~\ref{prop:maximal_characterization}, $(Y_1,\ldots,Y_i,Z_1,\ldots,Z_j)$ is also maximal. 
Thus $(X_1,\ldots,X_k)$ and $(Y_1,\ldots,Y_i)$ are equivalent as prefixes in $\Cpx$.
\end{proof}

\begin{proof}[Proof of Theorem~\ref{thm:binary_chain_full}]
\textit{Assertion~\ref{Cpx bcs}.}
We will show that $\Cpx$ is a chain system. 
The fact that the chain system is binary follows immediately from the definition.

Condition \ref{bound} in Definition~\ref{chain sys def} holds because $\Cpx = \Crpx\cup \Cx$ (Proposition~\ref{prop:chain_system_union}) and because $\Crpx$ and $\Cx$ are both chain systems (Theorems~\ref{thm:exceptional_binary_chain} and~\ref{thm:soft_binary_chain}).

To prove Condition~\ref{conv}, let $\X\in \Preeq(\Cpx)$ and $\Y\in \Posteq(\Cpx)$ and suppose $(X_1,\ldots,X_k)\in\X$ and $(Y_1,\ldots,Y_\ell)\in\Y$ have $(X_1,\ldots,X_k,Y_1,\ldots,Y_\ell) \in \Cpx$. 
For any $(X'_1,\ldots,X'_{k'})\in\X$ and $(Y'_1,\ldots,Y'_{\ell'})\in\Y$, by Proposition~\ref{prop:perp_prefix_full}, then Proposition~\ref{prop:maximal_characterization}, then Proposition~\ref{prop:perp_prefix_full} again,
\[\overline{\sW(X'_1,\ldots,X'_{k'})}=\overline{\sW(X_1,\ldots,X_k)}=\red((Y_1,\ldots,Y_\ell)^{\perp})= \red((Y'_1,\ldots,Y'_{\ell'})^{\perp}).\]
Thus $(X'_1,\ldots,X'_{k'},Y'_1,\ldots,Y'_{\ell'}) \in \Cpx$ by Proposition~\ref{prop:maximal_characterization}. 
This is Condition~\ref{conv}.

Now suppose $(X_1,\ldots,X_n)$, $(X_1,\ldots,X_{i-1},Y_1,\ldots,Y_j,X_{i+1},\ldots,X_n) \in \Cpx$.
Then $\overline{\sW(X_i)} = \overline{\sW(Y_1,\ldots,Y_j)}$ by Proposition~\ref{prop:maximal_characterization}.
By the hypothesis that there is more than one non-homogeneous tube, 
the sequence $(X_i)$ is not fully non-exceptional. We thus have two cases to consider.
If $X_i$ is exceptional, then $\overline{\sW(X_i)} = \sW(X_i)$ by Lemma~\ref{lem:overline_compute}.\ref{oc e}.
Corollary~\ref{cor:semibrick} says that $X_i$ is the unique brick in $\sW(X_i)$ and thus the unique brick in $\overline{\sW(X_i)}$.
If $X_i$ is non-exceptional, then $\overline{\sW(X_i)}=\sW(X_i)\oplus \homo$ by Lemma~\ref{lem:overline_compute}.\ref{oc nfne}\ref{oc nfne 3}. 
Thus $X_i$ is the only non-homogeneous brick in $\overline{\sW(X_i)}$ by Corollary~\ref{cor:semibrick}. In either case, Lemma~\ref{lem:closure}.\ref{closure 1} implies that $Y_1,\ldots,Y_j \in \overline{\sW(Y_1,\ldots,Y_j)} = \overline{\sW(X_i)}$.
But $Y_1,\ldots,Y_j$ are non-homogenous bricks, so $(Y_1,\ldots,Y_j)=(X_i)$ as sequences.
This is Condition~\ref{non subst}.

\smallskip \noindent
\textit{Assertion~\ref{Pc to pewide}.}
Proposition~\ref{prop:perp_prefix_full} implies that the map $\Preeq(\Cx) \rightarrow \pewide\Lambda$ is well-defined and bijective. 
It remains to show that the bijection and its inverse are order-preserving.
Let $\W,\V\in\pewide\Lambda$ and write $\X_\W$ and $\X_\V$ for their pre-images in $\Preeq(\Cpx)$.

Suppose first that $\X_\W \leq_\pre \X_\V$. 
Then there exists a maximal para-exceptional sequence $(X_1,\ldots,X_k)$ such that $\W = \overline{\sW(X_1,\ldots,X_i)}$ and $\V = \overline{\sW(X_1,\ldots,X_j)}$ for some $i\le j$.
In particular, $\sW(X_1,\ldots,X_i) \subseteq \sW(X_1,\ldots,X_j)$. 
If $(X_1,\ldots,X_i)$ is exceptional, then by Lemma~\ref{lem:overline_compute}.\ref{oc e}, then because $i\le j$, then by Lemma~\ref{lem:closure}.\ref{closure 1}, 
\[\overline{\sW(X_1,\ldots,X_i)} = \sW(X_1,\ldots,X_i) \subseteq \sW(X_1,\ldots,X_j) \subseteq \overline{\sW(X_1,\ldots,X_j)}.\]
If $(X_1,\ldots,X_i)$ is non-exceptional, then so is $(X_1,\ldots,X_j)$. 
Lemma~\ref{lem:exp_compute} then says that $\aug(\sW(X_1,\ldots,X_i)) = \sW(X_1,\ldots,X_i) \oplus \homo$, and likewise for $(X_1,\ldots,X_j)$. 
Therefore, $\aug(\sW(X_1,\ldots,X_i)) \subseteq \aug(\sW(X_1,\ldots,X_j))$.
Since the operators $(-)^\perp$ and ${}^\perp(-)$ are order-reversing, it follows that $\overline{\sW(X_1,\ldots,X_i)} \subseteq \overline{\sW(X_1,\ldots,X_j)}$. 

Conversely, suppose that $\W \subseteq \V$. 
By definition and by Proposition~\ref{prop:reduced_characterization}, there exist para-exceptional sequences $(X_1,\ldots,X_i)$ and $(Z_1,\ldots,Z_j)$ such that $\W = \overline{\sW(X_1,\ldots,X_i)}$ and $\V = \red((Z_1,\ldots,Z_j)^{\perp})$. 
Then $\W \subseteq \V \subseteq (Z_1,\ldots,Z_j)^\perp$, so also $(X_1,\ldots,X_i,Z_1,\ldots,Z_j)$ is a para-exceptional sequence. 
By Proposition~\ref{prop:chain_system_union}.\ref{index_arbitrary_full}, there is a maximal para-exceptional sequence $(X_1,\ldots,X_i,Y_1,\ldots,Y_k,Z_1,\ldots,Z_j)$. 
By Proposition~\ref{prop:maximal_characterization}, $\overline{\sW(X_1,\ldots,X_i,Y_1,\ldots,Y_k)} = \red((Z_1,\ldots,Z_j)^{\perp}) = \V$, so ${\X_\W \leq_{\pre} \X_\V}$.

\smallskip \noindent
\textit{Assertions~\ref{Cpx labels cov} and~\ref{Cpx labels}.}
Suppose that $\W \subseteq \V$ in $\pewide\Lambda$. 
By Assertions~\ref{Cpx bcs} and~\ref{Pc to pewide}, there is a maximal para-exceptional sequence $(X_1,\ldots,X_k)$ and an indices $i\leq j \in \{0,\ldots,k\}$ such that $\W = \overline{\sW(X_1,\ldots,X_{i})}$ and $\V = \overline{\sW(X_1,\ldots,X_j)}$. 
In this case, the cover relation is labeled by $X_j$. We will show that $\overline{\sW(X_{i+1},\ldots,X_j)}$ and ${}^\perp \W \cap \V$ contain the same non-homogeneous bricks.

For readability, write $\U_1 = \sW(X_1,\ldots,X_{i})$ and $\U_2 = \sW(X_{j+1},\ldots,X_k)$.
Now since $(X_1,\ldots,X_k)$ is maximal, and freely using Lemma~\ref{lem:ortho_wide}, by Proposition~\ref{prop:maximal_characterization}, then Lemma~\ref{lem:double_red}, then Proposition~\ref{prop:maximal_characterization} again, and finally Lemma~\ref{lem:perp_formula}, 
\begin{multline*}
\overline{\sW(X_{i+1},\ldots,X_j)}= \red({}^\perp\U_1 \cap \U_2^\perp) = \red(\red({}^\perp\U_1) \cap \red(\U_2^\perp)) \\
= \red(\red({}^\perp\U_1) \cap \V) = \red(\aug({}^\perp\W) \cap \V).
\end{multline*}
Now note that if $\red(\aug({}^\perp\W)\cap \V) \subsetneq \aug({}^\perp\W)\cap \V$, then $\aug({}^\perp\W)\cap \V \subseteq \reg$ and $\red(\aug({}^\perp\W)\cap \V) = \aug({}^\perp\W)\cap \V \cap \nonhomo$. 
Thus the non-homogeneous bricks in $\overline{\sW(X_{i+1},\ldots,X_j)} = \red(\aug({}^\perp\W)\cap \V)$ and in $\aug({}^\perp\W)\cap \V$ are the same. 
We now consider two cases, depending on whether $\W$ is exceptional.

Suppose first that $\W$ is exceptional. Then ${}^\perp \W$ is also exceptional by Corollary~\ref{cor:anti_isom_fg} and $\aug({}^\perp\W) = {}^\perp\W$ by Lemma~\ref{lem:exp_compute}. We conclude that the non-homogeneous bricks in $\aug({}^\perp\W) \cap \V$ and in ${}^\perp\W \cap \V$ are the same in this case.

Now suppose that $\W$ is not an exceptional subcategory. Then $(X_1,\ldots,X_i)$ is non-exceptional, but not fully non-exceptional, by Lemma~\ref{lem:overline_compute}.\ref{oc fne}\ref{oc fne 3} or \ref{lem:overline_compute}.\ref{oc e}. 
Then $\W = \overline{\U_1} = \aug(\U_1) = \U_1 \oplus \homo$ by Lemma~\ref{lem:overline_compute}.\ref{oc nfne}\ref{oc nfne 3}. 
Lemma~\ref{lem:overline_compute}.\ref{oc ne} then implies that ${}^\perp \W = {}^\perp \U_1 \cap \nonhomo$. Since $(X_1,\ldots,X_{i})$ is not fully non-exceptional, Proposition~\ref{prop:tube_generation_2}.\ref{tg2 2} then says that ${}^\perp\W$ is representation-infinite. Since ${}^\perp\W \subseteq \nonhomo$, Proposition~\ref{prop:tubes_ortho} and the definition thus yield $\aug({}^\perp\W) = {}^\perp\W \oplus \homo$. Again, we conclude that the non-homogeneous bricks in ${}^\perp\W \cap \V$ and in $\aug({}^\perp\W) \cap \V = ({}^\perp\W \cap \homo)\cap \V$ are the same.

We have shown that the non-homogeneous bricks in $\overline{\sW(X_{i+1},\ldots,X_{j})}$ and ${}^\perp \W \cap \V$ coincide. Now, by Assertions~\ref{Cpx bcs} and~\ref{Pc to pewide}, there is a cover relation $\W \covered \V$ if and only if $j = i+1$. When these equivalent conditions hold, the cover relation is labeled by $X_j$, and we showed in the proof of Assertion~\ref{Cpx bcs} that $X_j$ is the unique non-homogeneous brick in $\overline{\sW(X_j)}$. Thus suppose that $\W$ is not covered by $\V$. If $\W = \V$, then ${}^\perp\W \cap \V$ contains no non-homogeneous bricks. Otherwise $j > i+1$ and, by Lemma~\ref{lem:closure}.\ref{closure 1}, $X_{i+1},\ldots,X_j$ are all non-homogeneous bricks in $\overline{\sW(X_{i+1},\ldots,X_j)}$.
\end{proof}


\section{Lattice and Garside properties}\label{sec:lattice}
One of the key results of \cite{McSul} is that the labeled poset $[1,c]_{T \cup F}$ is a combinatorial Garside structure. 
(See Section~\ref{Garside sec}.)
In this section, we prove the following theorem, which combines with Theorem~\ref{thm:bijection_on_chain_systems} to yield a new proof of the McCammond-Sulway result. Recall that $\Lambda$ is always a connected tame hereditary algebra.

\begin{theorem}\label{thm:garside}
Suppose $\Lambda$ is a connected tame hereditary algebra with more than one non-homogeneous tube.
Then $\P(\Cpx)$ is a combinatorial Garside structure.
\end{theorem}

We begin with some lemmas and propositions, leading to a proof that $\pewide\Lambda$ is a lattice.
The proof of the lattice property does not need the hypothesis that $\Lambda$ has more than one non-homogeneous tube.

\begin{lemma}\label{lem:infinite_tubes}
Let $\W, \V \in \ewide \Lambda$. If $\W \cap \V \cap \T$ is representation-infinite for every non-homogeneous tube~$\T$, then $\W \cap \V \in \ewide \Lambda$.
\end{lemma}

\begin{proof}
Since $\W$ and $\V$ are both exceptional, by Proposition~\ref{prop:wide_fg} there exist exceptional sequences $(X_1,\ldots,X_k)$ and $(X_{k+1},\ldots,X_j)$ such that $\W = {}^\perp(X_1,\ldots,X_k)$ and $\V = {}^\perp(X_{k+1},\ldots,X_j)$. 
Thus $\W \cap \V = {}^\perp(X_1,\ldots,X_j)$. 
(Although $(X_1,\ldots,X_j)$ may not be an exceptional sequence, it still has a perpendicular category.)
We now show that $\sW(X_1,\ldots,X_j) \subseteq \nonhomo$ and that $\sW(X_1,\ldots,X_j)$ is representation-finite.

Let $\T$ be a non-homogeneous tube, so that $\W \cap \V \cap \T$ is representation-infinite by hypothesis.
Proposition~\ref{prop:tube_generation_1} implies that there exists a non-exceptional brick $Y \in \W \cap \V \cap \T$. 
By Lemma~\ref{lem:hom_ext}, each $X_i$ is regular, so $\sW(X_1,\ldots,X_j) \subseteq \nonhomo$. 
Moreover, each $X_i$ lies in $Y^\perp$, so $\sW(X_1,\ldots,X_j) \cap \T \subseteq Y^\perp \cap \T$ is representation-finite by Proposition~\ref{prop:wide_tube}. 
Since $\T$ was chosen arbitrarily, Proposition~\ref{prop:tubes_ortho} implies that $\sW(X_1,\ldots,X_j)$ is representation-finite.

Since $\sW(X_1,\ldots,X_j)$ is representation-finite, it is also exceptional by Lemma~\ref{lem:rep_finite_exceptional}. 
Thus $\W \cap \V = {}^\perp(X_1,\ldots,X_j)$ is exceptional by Corollary~\ref{cor:anti_isom_fg} and Lemma~\ref{lem:ortho_wide}.
\end{proof}

\begin{lemma}\label{lem:reduced_tube_infinite}
Let $\W$ be a wide subcategory.
If $\W$ contains all of the homogeneous tubes
and there exists a non-homogeneous tube $\T$ such that $\W \cap \T$ is representation-finite, then $\red(\W) \in \pewide\Lambda$.
\end{lemma}

\begin{proof}
If $\W$ is exceptional then $\red(\W) = \W \in \pewide\Lambda$ by Lemma~\ref{lem:reduce_fg} and Proposition~\ref{prop:reduced_characterization}, so suppose $\W$ is not exceptional. Thus $\W \subseteq \reg$ by Lemma~\ref{lem:non_finite_reg}. 
If there exists a non-homogeneous tube $\S$ such that $\W \cap \S$ is representation-infinite, then since $\W \cap \S \subseteq \W \cap \nonhomo$,
also $\W \cap \nonhomo$ is representation-infinite.
Thus $\red(\W) = \W$, and $\W \in \pewide\Lambda$ by Proposition~\ref{prop:reduced_characterization}.
On the other hand, if $\W \cap \S$ is representation-finite for every non-homogeneous tube $\S$, then Proposition~\ref{prop:wide_tube} and Proposition~\ref{prop:tubes_ortho} imply that $\W \cap \nonhomo$ is representation-finite and exceptional. 
Then $\red(\W) = \W \cap \nonhomo \in \pewide\Lambda$ by Proposition~\ref{prop:reduced_characterization}.
\end{proof}

\begin{proposition}\label{prop:intersection}
Let $\W, \V \in \pewide\Lambda$. Then $\red(\W \cap \V) \in \pewide\Lambda$.
\end{proposition}

\begin{proof}
If at least one of $\W$ or $\V$ is representation-finite, then $\W \cap \V$ is representation-finite and thus exceptional by Lemma~\ref{lem:rep_finite_exceptional}. 
Thus $\red(\W \cap \V) = \W \cap \V \in \pewide\Lambda$ by Lemma~\ref{lem:reduce_fg} and Proposition~\ref{prop:reduced_characterization}.
Thus for the rest of the proof, we suppose that both $\W$ and $\V$ are representation-infinite. 
By Lemma~\ref{lem:homogeneous_stable} and Proposition~\ref{prop:reduced_characterization}, $\W \cap \V$ contains all of the homogeneous tubes. 
If one of the two (without loss of generality $\W$) is not exceptional, then Proposition~\ref{prop:reduced_characterization} says that $\W \cap \T$ is representation-finite for some non-homogeneous tube $\T$. 
Then $\W \cap \V \cap \T$ is also representation-finite, so $\red(\W \cap \V) \in \pewide\Lambda$ by Lemma~\ref{lem:reduced_tube_infinite}.

Finally, suppose that both $\W$ and $\V$ are exceptional. 
If $\W \cap \V \cap \T$ is representation-infinite for every non-homogeneous tube $\T$, then $\W \cap \V$ is exceptional by Lemma~\ref{lem:infinite_tubes}, and so $\red(\W \cap \V) = \W \cap \V\in \pewide\Lambda$ by Lemma~\ref{lem:reduce_fg} and Proposition~\ref{prop:reduced_characterization}. 
On the other hand, if there exists a non-homogeneous tube $\T$ such that $\W \cap \V \cap \T$ is representation-finite, then $\red(\W \cap \V) \in \pewide\Lambda$ by Lemmas~\ref{lem:homogeneous_stable} and~\ref{lem:reduced_tube_infinite}.
\end{proof}

\begin{proposition}\label{prop:meet_semilattice}
The poset $\pewide\Lambda$ is a meet semi-lattice with meet operation $\W \meet \V = \red(\W \cap \V)$.
\end{proposition}

\begin{proof}
Suppose $\W, \V \in \pewide\Lambda$. 
Then $\red(\W \cap \V) \in \pewide\Lambda$ by Proposition~\ref{prop:intersection}.
It remains to show that if $\U \subseteq \W \cap \V$ for some $\U \in \pewide\Lambda$, then $\U \subseteq \red(\W \cap \V)$.
If $\W \cap \V$ is exceptional, then $\red(\W \cap \V) = \W \cap \V$ by Lemma~\ref{lem:reduce_fg} and we are done. 
If $\W \cap \V$ is not exceptional, then $\U \subseteq \W \cap \V \subseteq \reg$ by Lemma~\ref{lem:non_finite_reg}. 
Noting that if $\W \cap \V \cap \nonhomo$ is representation-finite then $\U \cap \nonhomo$ is representation-finite, Corollary~\ref{para reduced} and the definition of $\red(-)$ then implies that $\U = \red(\U) \subseteq \red(\V \cap \W)$.
\end{proof}

\begin{proposition}\label{prop:anti_isom}
The map $\W \mapsto \aug(\W^\perp)$ is an anti-automorphism of the poset $\pewide\Lambda$, with inverse $\W \mapsto \aug({}^\perp\W)$.
\end{proposition}

\begin{proof}
Recall from Remark~\ref{ortho rem} that $(-)^\perp$ and ${}^\perp(-)$ are order-reversing. 
It is also clear from the definition that $\aug(-)$ is order-preserving. 
Thus both maps in the statement are order-reversing. 
Suppose $\W = \overline{\sW(X_1,\ldots,X_k)} \in \pewide\Lambda$ for some para-exceptional sequence $(X_1,\ldots,X_k)$.
The equivalence \mbox{\ref{rc 1}$\iff$\ref{rc 2}} in Proposition~\ref{prop:reduced_characterization} combines with Lemma~\ref{lem:perp_formula} and Lemma~\ref{lem:ortho_wide} to imply that $\aug(\W^\perp) = \red((X_1,\ldots,X_k)^\perp) \in \pewide\Lambda$. 
Symmetrically, $\aug({}^\perp\W) \in \pewide\Lambda$. 
By symmetry, to show that the maps are inverse to one another, it is enough to prove ${\aug({}^\perp(\aug(\W^\perp))) = \W}$.

Suppose $(X_1,\ldots,X_k)$ is either an exceptional sequence or is fully non-exceptional.
Then $\W$ is an exceptional subcategory by Lemma~\ref{lem:overline_compute} (part~\ref{oc fne}\ref{oc fne 2}~or~\ref{oc e}) and $\W^\perp$ and~${}^\perp \W$ are also exceptional by Corollary~\ref{cor:anti_isom_fg}. 
By Lemma~\ref{lem:exp_compute}, $\aug(\W^\perp) = \W^\perp$ and $\aug({}^\perp\W) = {}^\perp \W$. 
Thus by Corollary~\ref{cor:anti_isom_fg} and then Proposition~\ref{prop:reduced_characterization}, 
\[\aug({}^\perp(\aug(\W^\perp))) = \aug({}^\perp(\W^\perp)) = \aug(\W) = \W.\]

Suppose, on the other hand, that $(X_1,\ldots,X_k)$ is non-exceptional, but not fully non-exceptional. 
Let $\V = \sW(X_1,\ldots,X_k)$, so that $\W = \overline{\V} = \aug(\V) = \V \oplus \homo$ by Lemma~\ref{lem:overline_compute}.\ref{oc nfne}\ref{oc nfne 3}. 
Lemma~\ref{lem:overline_compute}.\ref{oc ne}\ref{oc ne 1} then implies that $\W^\perp = \V^\perp \cap \nonhomo$.
This category is representation-infinite by Lemma~\ref{lem:overline_compute}.\ref{oc nfne}\ref{oc nfne 2}, so Lemma~\ref{lem:join oplus} implies that 
\[\aug(\W^\perp) = \W^\perp \oplus \homo = (\V ^\perp \cap \nonhomo) \oplus \homo \subseteq \reg.\]
Thus $\aug({}^\perp(\aug(\W^\perp))) = \aug({}^\perp((\V^\perp \cap \nonhomo) \oplus \homo))$, and by Lemmas~\ref{lem:hom_ext} and~\ref{lem:tubes_ortho}, then Proposition~\ref{prop:tube_anti_isom}.\ref{tai 1}, then Lemma~\ref{lem:overline_compute}.\ref{oc nfne}\ref{oc nfne 3}, this equals
\[\aug({}^\perp(\V^\perp \cap \nonhomo) \cap \nonhomo) = \aug(\V) = \W.\qedhere\]
\end{proof}

Combining Proposition~\ref{prop:meet_semilattice} and Proposition~\ref{prop:anti_isom}, we obtain the following.

\begin{theorem}\label{thm:join}
Let $\Lambda$ be a connected tame hereditary algebra. Then the poset $\pewide\Lambda$ is a lattice with meet $\W \meet \V = \red(\W \cap \V)$ and join  
\begin{align*}\W \join \V &= \aug({}^\perp(\red(\aug(\W^\perp) \cap \aug(\V^\perp))))\\
&= \aug((\red(\aug({}^\perp\W) \cap \aug({}^\perp\V)))^\perp)\end{align*}
\end{theorem}

\begin{remark}\label{highlight}
We emphasize that Theorem~\ref{thm:join} is true without the assumption, which appeared in some earlier results, that $\Lambda$ has more than one non-homogeneous tube.
Since Theorem~\ref{thm:binary_chain_full} needs that assumption, we will make that assumption in the remaining results in this section.
\end{remark}

\begin{corollary}\label{lattice iff}
Suppose $\Lambda$ is a connected tame hereditary algebra.
Then the following are equivalent.
\begin{enumerate}[label=\rm(\roman*), ref=(\roman*)]
\item \label{garside_one_tube} $\P(\Cx)$ is a combinatorial Garside structure.
\item \label{lattice_one_tube} $\ewide\Lambda$ is a lattice.
\item \label{excep_one_tube} $\ewide\Lambda=\pewide\Lambda$.
\item \label{one_tube_condition} $\Lambda$ has fewer than two non-homogeneous tubes.
\end{enumerate}
\end{corollary}
\begin{proof}
The equivalence \ref{garside_one_tube}$\iff$\ref{lattice_one_tube} follows from Corollary~\ref{cor:garside_if_lattice} and the definition of a combinatorial Garside structure. The equivalence \ref{excep_one_tube}$\iff$\ref{one_tube_condition} is part of Corollary~\ref{excep subs}. The implication \ref{excep_one_tube}$\implies$\ref{lattice_one_tube} follows from Theorem~\ref{thm:join}. We finish the proof by showing the contrapositive of \ref{lattice_one_tube}$\implies$\ref{one_tube_condition}.

Suppose that $\Lambda$ has more than one non-homogeneous tube, and denote them by $\T_1,\ldots,\T_m$. Let $X_1,\ldots,X_r$ be the quasi-simples in $\T_1$ and let $Y_1,\ldots,Y_s$ be the quasi-simples in the other non-homogeneous tube(s). By Propositions~\ref{prop:bricks_in_tube} and~\ref{prop:wide_fg}, we have that $\sW(X_i)$ and $Y_j^\perp$ are both exceptional for any $i$ and $j$. Moreover, $\sW(X_i) \subseteq Y_j^\perp$ by Lemma~\ref{lem:tubes_ortho}. Now suppose for a contradiction that $\sW(X_1),\ldots,\sW(X_r)$ have a join $\W$ in $\ewide \Lambda$. Then $\T_1 = \sW(X_1,\ldots,X_r) \subseteq \W$. Similarly, by Lemmas~\ref{lem:hom_ext} and~\ref{lem:tubes_ortho}, and~\ref{lem:ortho_wide}, $\W \subseteq \bigcap_{j = 1}^s Y_j^\perp =  \sW(Y_1,\ldots,Y_s)^\perp = \left(\bigoplus_{k = 2}^m \T_k\right)^\perp = \T_1 \oplus \homo.$ This is a contradiction, since there is no exceptional subcategory which both contains $\T_1$ and is contained in $\T_1 \oplus \homo$. (See Propositions~\ref{prop:wide_tube} and~\ref{prop:tubes_ortho}.)
\end{proof}

Recall from Section~\ref{chain sys sec} the notation $\C|_{[x,y]}$ for $x\le y$ in $\P(\C)$.
The following lemma can be seen as an extension of Theorem~\ref{thm:binary_chain_full}.\ref{Cpx labels cov} and an analog of Lemma~\ref{lem:interval_labels_excep}.
We use the notation $\W\mapsto P_\W$ for the inverse of the isomorphism in Theorem~\ref{thm:binary_chain_full}.\ref{Pc to pewide}.

\begin{lemma}\label{lem:interval_labels}
Suppose $\W\subseteq\V$ in $\pewide\Lambda$.
Then $\Cpx|_{[P_\W,P_\V]}$ is the set of para-exceptional sequences such that $\overline{\sW(X_1,\ldots,X_k)} = \red(\aug({}^\perp\W)\cap\V)$.
\end{lemma}

\begin{proof}
Suppose first that $(X_1,\ldots,X_k) \in \Cpx|_{[P_\W,P_\V]}$. 
Then, by Theorem~\ref{thm:binary_chain_full}, $(X_1,\ldots,X_k)$ is a para-exceptional sequence and  there exists a maximal para-exceptional sequence $(Y_1,\ldots,Y_i,X_1,\ldots,X_k,Z_1,\ldots,Z_j)$ with ${\overline{\sW(Y_1,\ldots,Y_i)} = \W}$ and $\overline{\sW(Y_1,\ldots,Y_i,X_1,\ldots,X_k)} = \V$. 
Arguing as in the proof of Theorem~\ref{thm:binary_chain_full}.\ref{Cpx labels cov}, we have
$\overline{\sW(X_1,\ldots,X_k)} = \red(\V \cap \aug({}^\perp\W))$, as desired.

Conversely, suppose that $(X_1,\ldots,X_k)$ is a para-exceptional sequence such that ${\overline{\sW(X_1,\ldots,X_k)} = \red(\aug({}^\perp\W)\cap\V)}$. 
Since $\W,\V\in\pewide\Lambda$, there is a para-exceptional sequence $(Y_1,\ldots,Y_i)$ with $\W = \overline{\sW(Y_1,\ldots,Y_i)}$ and Proposition~\ref{prop:reduced_characterization} says there is a para-exceptional sequence $(Z_1,\ldots,Z_j)$ with $\V = \red((Z_1,\ldots,Z_j)^\perp)$. 
By Lemma~\ref{lem:closure}.\ref{closure 1} and the definition of $\red(-)$,
\begin{equation}\label{eqn:restriction}\sW(X_1,\ldots,X_k) \subseteq \overline{\sW(X_1,\ldots,X_k)} = \red(\V \cap \aug({}^\perp\W)) \subseteq \V \cap \aug({}^\perp\W).\end{equation}
Again by Lemma~\ref{lem:closure}.\ref{closure 1}, we have $\sW(Y_1,\ldots,Y_i) \subseteq \overline{\sW(Y_1,\ldots,Y_i)} = \W \subseteq \V$.
Also $\V = \red((Z_1,\ldots,Z_j)^\perp) \subseteq (Z_1,\ldots,Z_j)^\perp$, by the definition of $\red$.
Furthermore, combining~\eqref{eqn:restriction} with Lemma~\ref{lem:perp_formula}, we obtain
\[\sW(X_1,\ldots,X_k) \subseteq \aug({}^\perp\W) = \red({}^\perp(Y_1,\ldots,Y_i)) \subseteq {}^\perp(Y_1,\ldots,Y_i).\]
Combining these containments, we have $\set{X_1,\ldots,X_k,Y_1,\ldots,Y_i}\subseteq(Z_1,\ldots,Z_j)^\perp$ and $\set{X_1,\ldots,X_k}\subseteq{}^\perp(Y_1,\ldots,Z_i)$, so $(Y_1,\ldots,Y_i,X_1,\ldots,X_k,Z_1,\ldots,Z_j)$ is a para-exceptional sequence. 
Furthermore, by Lemmas~\ref{lem:perp_formula} and~\ref{lem:double_red}, 
\begin{align*}
\overline{\sW(X_1,\ldots,X_k)} &= \red(\aug({}^{\perp}\W)\cap \V) \\
&= \red(\red({}^{\perp}(Y_1,\ldots,Y_i))\cap \red ((Z_1,\ldots,Z_j)^{\perp}))\\
&= \red({}^{\perp}(Y_1,\ldots,Y_i)\cap (Z_1,\ldots,Z_j)^{\perp}),
\end{align*}
so $(Y_1,\ldots,Y_i,X_1,\ldots,X_k,Z_1,\ldots,Z_j)$ is maximal by Proposition~\ref{prop:maximal_characterization}. 
Applying Proposition~\ref{prop:maximal_characterization} once more, we conclude that 
\[\V = \red((Z_1,\ldots,Z_j)^\perp) = \overline{\sW(Y_1,\ldots,Y_i,X_1,\ldots,X_k)}.\]
Therefore $(X_1,\ldots,X_k)\in\Cpx|_{[P_\W,P_\V]}$.
\end{proof}

\begin{lemma}\label{Cpx weak graded}
Suppose that $\Lambda$ contains $m > 0$ non-homgeneous tubes. Assign each exceptional module the weight $1$ and assign each element of $\nrig$ the weight $2/m$.
The sum of these weights, over any maximal para-exceptional sequence, is~$n$.
\end{lemma}

\begin{proof}
Let $(X_1,\ldots,X_k)$ be a maximal para-exceptional sequence. If $(X_1,\ldots,X_k)$ is exceptional, then it is also maximal as an exceptional sequence by Proposition~\ref{prop:chain_system_union}.\ref{Cpx Cx}. Proposition~\ref{prop:complete_ex_maximal} then implies that the sum of the weights in $n$.

If $(X_1,\ldots,X_k)$ is not an exceptional sequence, then it is a maximal regular para-exceptional sequence, and is fully non-exceptional, by Proposition~\ref{prop:chain_system_union}.\ref{Cpx Crpx}.
Therefore, the sequence contains exactly $m$ non-exceptional modules, contributing $m\cdot\frac2m$ to the weight.
Proposition~\ref{prop:complete_generates_tube} says that the length of the sequence is $n-2+m$, so the exceptional modules in the sequence contribute $n-2$ to the weight.
\end{proof}

We can now prove the main result of this section.

\begin{proof}[Proof of Theorem~\ref{thm:garside}]  
Lemma~\ref{Cpx weak graded} says that $\P(\Cpx)$ is weighted-graded.
By Proposition~\ref{chain sys poset} and Theorem~\ref{thm:binary_chain_full}.\ref{Cpx bcs}, $\P(\Cpx)$ has finite height, a unique minimal element, and a unique maximal element. 
Proposition~\ref{prop:chain_system_union}.\ref{index_arbitrary_full} implies that $\P(\Cpx)$ is balanced. Suppose $P_\W\le P_\V$ and $P_{\W'}\le P_{\V'}$ are order relations in $\P(\Cpx)$ such that ${\Cpx|_{[P_\W,P_\V]} \cap \Cpx|_{[P_{\W'},P_{\V'}]} \neq \emptyset}$.
Lemma~\ref{lem:interval_labels} implies that ${\red(\V \cap \aug({}^\perp\W))} = \red(\V' \cap \aug({}^\perp\W'))$, but then applying Lemma~\ref{lem:interval_labels} again, we see that $\Cpx|_{[P_\W,P_\V]} = \Cpx|_{[P_{\W'},P_{\V'}]}$.
We see that $\Cpx$ has the restriction property, so $\P(\Cpx)$ is group-like by Proposition~\ref{res prop}.
Theorem~\ref{thm:binary_chain_full}.\ref{Pc to pewide} and Theorem~\ref{thm:join} combine to say that $\P(\Cpx)$ is a lattice.
\end{proof}

\section{Related constructions}\label{sec:other}

We conclude this paper with a brief discussion of other constructions from the literature that are related to the
lattice of para-exceptional  subcategories.

We first recall that
exceptional subcategories are examples of the more general \newword{semi-stable wide subcategories}. In \cite[Proposition~9.12]{IPT}, Ingalls, Paquette, and Thomas show that, for $\Lambda$ a path algebra over a Euclidean quiver,
a wide subcategory~$\W$ is semi-stable if either (i) $\W$ is exceptional, or (ii) $\W \cap \T$ is representation-infinite for every tube $\T$. They then close the set of semi-stable wide subcategories under intersections to form a lattice. Comparing Proposition~\ref{prop:reduced_characterization} with \cite[Proposition~9.12 and Theorem~10.8]{IPT}, we see that $\pewide\Lambda$ is a proper subposet of the lattice constructed in \cite{IPT}. 
Moreover, one can check that, outside of rank $2$, $\pewide\Lambda$ is not closed under intersections and thus not a sublattice of the lattice constructed in \cite{IPT}. 

    We now briefly explain how one could construct a poset isomorphic to $\pewide\Lambda$ (for $\Lambda$ any connected tame hereditary algebra) by making different choices about when to include the homogeneous tubes. Define an operator $\eta: \pewide\Lambda \rightarrow \ewide\Lambda \cup \wide(\nonhomo)\subseteq \wide \Lambda$ by
    \[\eta(\W) = \begin{cases} \W  & \text{if $\W \in \ewide \Lambda$}\\\W\cap \nonhomo & \text{otherwise}.\end{cases}\]
    Explicitly, for $(X_1,\ldots,X_k)$ a para-exceptional sequence, we have
    \[\eta\left(\overline{\sW(X_1,\ldots,X_k)}\right) = \begin{cases} \overline{\sW(X_1,\ldots,X_k)} & \text{if $(X_1,\ldots,X_k)$ is fully non-exceptional}\\\sW(X_1,\ldots,X_k) & \text{otherwise}
    \end{cases}\]
    by Assertions \ref{oc fne}\ref{oc fne 3}, \ref{oc nfne}\ref{oc nfne 3}, and \ref{oc e} of Lemma~\ref{lem:overline_compute}.
    Now by Propositions~\ref{prop:wide_tube} and~\ref{prop:tubes_ortho}, we see that the complement of the image of $\eta$ consists of those wide subcategories $\W \subseteq \nonhomo$ for which $\W \cap \T$ is representation-infinite for every non-homogeneous tube~$\T$. In particular, $\eta$ is not surjective as long as there is at least one non-homogeneous tube. On the other hand, it follows from Proposition~\ref{prop:reduced_characterization} and Lemma~\ref{lem:homogeneous_stable} that $\eta$ is an order-isomorphism onto its image.

Suppose from now on that $\Lambda$ is a path algebra of type $\widetilde{A}$.  Then all para-exceptional modules are also \newword{string modules}. The poset of wide subcategories generated by string modules is studied in the recent preprint \cite{page}. (To be precise, the cited preprint studies those \newword{thick subcategories} of the bounded derived category of a \newword{gentle algebra} which are generated by string objects. Path algebras of type $\widetilde{A}$ arise as a special case. Moreover, since path algebras are hereditary, there is an order-preserving isomorphism between the wide subcategories of $\mods\Lambda$ and the thick subcategories of the bounded derived category, see \cite{bruning}.) The map~$\eta$ defined above thus allows us to identify $\pewide \Lambda$ with a proper subposet of the poset considered in \cite{page} (for $\Lambda$ a path algebra of type $\widetilde{A}$).

The above paragraph raises another interesting consideration. 
For path algebras of type $\widetilde{A}$, one could also consider $\mathcal{S}$-brick sequences for $\mathcal{S}$ the set of bricks which are string modules. There are always exactly two tubes that contain string modules, which we denote $\mathcal{S}_1$ and $\mathcal{S}_2$. If the algebra has two non-homogeneous tubes, they are $\mathcal{S}_1$ and $\mathcal{S}_2$, and so $\mathcal{S}$-brick sequences are precisely para-exceptional sequences. Thus suppose there are fewer than two non-homogeneous tubes. (There cannot be three non-homogeneous tubes in type $\widetilde{A}$). One can then repeat the considerations of Sections~\ref{sec:rep_theory_tubes}--\ref{sec:lattice}, replacing $\nrig$ with $\mathcal{S} \setminus \excep$, replacing $\nonhomo$ with $\mathcal{S}_1 \oplus \mathcal{S}_2$, and replacing $\homo$ with the additive closure of the other (homogeneous) tubes. The result is a binary chain system $\C_{\mathrm{string}}$ whose corresponding labeled poset can be realized as a subposet of $\wide \Lambda$ which properly contains $\ewide \Lambda$. 
In this case, both $\ewide\Lambda$ (by Corollary~\ref{cor:garside_if_lattice} and Theorem~\ref{thm:exceptional_binary_chain}.\ref{Pc to fwide}) and the larger poset (by analogs of Theorem~\ref{thm:garside} and Theorem~\ref{thm:binary_chain_full}.\ref{Pc to pewide}) are combinatorial Garside structures. 
The case where there are fewer than two non-homogeneous tubes also has meaning in the combinatorial model of noncrossing partitions of an annulus, described in Section~\ref{type sec}: It is the case where there is exactly one inner point or exactly one outer point.
In this case, the lattice of noncrossing partitions of the annulus models the larger poset obtained from $\S$-brick sequences, so \cite[Theorem~3.15]{affncA} provides a different proof that this larger poset is a graded lattice.

\end{document}